\documentclass[10pt]{article}
\usepackage{latexsym, amsmath, amssymb, amsthm}
\usepackage{enumerate, mathrsfs, bbm}
\usepackage{graphicx, tikz}
\usepackage{authblk, setspace, cite, floatrow}
\usepackage[top=1 in, bottom=1 in, left=1 in, right=1 in]{geometry}
\usepackage[pdftex,linkcolor=red,citecolor=blue,backref=page]{hyperref}

\newtheorem{theorem}{Theorem}[section]
\newtheorem{lemma}[theorem]{Lemma}

\newtheorem{proposition}[theorem]{Proposition}
\newtheorem{remark}[theorem]{Remark}
\newtheorem{definition}[theorem]{Definition}

\numberwithin{equation}{section}

\makeatletter
\def\blfootnote{\xdef\@thefnmark{}\@footnotetext}
\makeatother
\newfloatcommand{capbtabbox}{table}[][\FBwidth]

\def\m{\mathbb}		\def\mcal{\mathcal}		\def\mb{\mathbbm}		
\def\lam{\lambda}  \def\eps{\epsilon}  
\def\a{\alpha}	\def\b{\beta}	\def\ls{\lesssim}
\def\p{\partial}  
\def\wh{\widehat}	\def\la{\langle}	\def\ra{\rangle}
\def\ls{\lesssim}	\def\gs{\gtrsim}
\def\i{\int\limits}
\def\be{\begin{equation}}     \def\ee{\end{equation}}
\def\bp{\begin{pmatrix}}	\def\ep{\end{pmatrix}}

\begin{document}
	\title{Local Well-posedness of the  Coupled KdV-KdV Systems on $\mathbb{R}$}
	\author[]{Xin Yang and Bing-Yu Zhang }
	\date{}
	\maketitle
	
	\begin{abstract}
		Inspired by  the recent  successful completion of the study of the well-posedness theory for the Cauchy problem of the Korteweg-de Vries (KdV) equation
		\[ u_t +uu_x +u_{xxx}=0, \quad  \left. u \right |_{t=0}=u_{0} \]
		in the space $H^{s} (\mathbb{R})$ (or $H^{s} (\mathbb{T})$),   we study   the well-posedness of the Cauchy problem for a class of coupled KdV-KdV (cKdV) systems 
		\[\left\{\begin{array}{rcl}
			u_t+a_{1}u_{xxx} &=& c_{11}uu_x+c_{12}vv_x+d_{11}u_{x}v+d_{12}uv_{x},\vspace{0.03in}\\
			v_t+a_{2}v_{xxx}&=& c_{21}uu_x+c_{22}vv_x +d_{21}u_{x}v+d_{22}uv_{x}, \vspace{0.03in}\\
			\left. (u,v)\right |_{t=0} &=& (u_{0},v_{0})
		\end{array}\right.\]
		in the space $\mathcal{H}^s (\mathbb{R}) := H^s (\mathbb{R})\times H^s (\mathbb{R})$. Typical examples include the Gear-Grimshaw  system, the Hirota-Satsuma system and the Majda-Biello system, to name a few. They  usually serve as models
		to describe the interaction of two long waves with different dispersion relations.
		
		In this paper we look for those values of  $s\in \mathbb{R}$ for which the cKdV systems are well-posed in $\mathcal{H}^s (\m{R})$. Our findings enable us  to provide a  complete classification for the cKdV systems in terms of the analytical well-posedness in  $\mathcal{H}^s (\mathbb{R})$  based on its coefficients $a_i$, $c_{ij}$ and $d_{ij}$ for $i,j=1,2$. The key ingredients in the proofs are the bilinear estimates in both divergence and non-divergence forms under the Fourier restriction  space norms. There are four types of the bilinear estimates that need to be investigated. Sharp results  are established  for  all of them.  In contrast to the lone critical index $-\frac{3}{4}$ for the single KdV equation, the critical indexes for the cKdV systems are $-\frac{13}{12}$, $-\frac{3}{4}$, $0$ and $\frac{3}{4}$. 
		
		As a result, the cKdV systems are classified into four classes, each of which corresponds to a unique index $s^{*}\in\{-\frac{13}{12},\,-\frac{3}{4},\,0,\,\frac{3}{4}\}$ such that any system in this class is locally analytically well-posed if $s>s^{*}$ while the bilinear estimate fails if $s<s^{*}$.
	\end{abstract}
	
	\blfootnote{2010 Mathematics Subject Classification. 35Q53; 35E15; 35G55; 35L56; 35D30.}
	
	\blfootnote{Key words and phrases. KdV-KdV systems; Gear-Grimshaw system;  Hirota-Satsuma system; Majda-Biello system; Local well-posedness; Fourier restriction space; Bilinear estimates.}
	
	
	
	\section{Introduction}  
	\label{Sec, intro}
	
	\subsection{Problem to study}
	\label{Subsec, prob}\quad
	This paper studies   the  Cauchy problem of  a  class of coupled KdV-KdV systems  posed on the whole line   $\m{R}$ of the following general form, 
	\be\label{ckdv, general}
	\left\{\begin{array}{ll}
		\bp u_t\\v_t  \ep + A_{1}\bp u_{xxx}\\v _{xxx}\ep + A_{2}\bp u_x\\v_x \ep = A_{3}\bp uu_x\\vv_x \ep + A_{4}\bp u_{x}v\\uv_{x}\ep, & x\in\m{R}, t\in\m{R}, \vspace{0.1in}\\
		\left.  \bp u\\v\ep\right |_{t=0} =\bp u_{0}\\v_{0}\ep,
	\end{array}\right.\ee
	where $\{A_{i}\}_{1\leq i\leq 4}$ are $2\times 2$ real  constant matrices,   $u=u(x,t)$,  $ 
	v =v(x,t) $ are real-valued unknown  functions of the two real variables  $x$ and $t$, and subscripts adorning $u$ and $v$ connote partial differentiations $\partial _t$ or $\partial _x$.  It is assumed  that  there exists an invertible real matrix $M$ such that $$A_{1}=M\bp a_{1}& 0\\0& a_{2}\ep M^{-1}, $$with  $a_{1}a_2 \ne 0$.  By regarding $M^{-1}\bp u\\v\ep$ as the new unknown functions (still denoted by $u$ and $v$),  the system (\ref{ckdv, general})   can be rewritten in  the following form,
	\be\label{ckdv, coef form}
	\left\{\begin{array}{rcl}
		u_t+a_{1}u_{xxx}+b_{11}u_x &=& -b_{12}v_x+c_{11}uu_x+c_{12}vv_x+d_{11}u_{x}v+d_{12}uv_{x},\vspace{0.03in}\\
		v_t+a_{2}v_{xxx}+b_{22}v_x &=& -b_{21}u_x+c_{21}uu_x+c_{22}vv_x +d_{21}u_{x}v+d_{22}uv_{x}, \vspace{0.03in}\\
		\left. (u,v)\right |_{t=0} &=& (u_{0},v_{0}).
	\end{array}\right.\ee
	This system   is called in {\em divergence form}  if $d_{11}=d_{12}$ and $d_{21}=d_{22}$.  Otherwise, it is called in {\em non-divergence form}.  
	
	Listed below  are  a few specializations of (\ref{ckdv, general})   appeared in the literature.
	\begin{itemize}
		\item Majda-Biello system: 
		\be\label{M-B system}
		\left\{\begin{array}{rll}
			u_{t}+u_{xxx} &=& -vv_{x}, \\
			v_{t}+a_{2}v_{xxx} &=& -(uv)_{x}, \\
			\left. (u,v)\right |_{t=0} &=& (u_{0},v_{0}),
		\end{array}\right.\ee
		where $a_{2}\neq 0$. This system was  proposed by Majda and Biello in \cite{MB03} as a reduced asymptotic model to study the nonlinear resonant interactions of long wavelength equatorial Rossby waves and barotropic Rossby waves.
		
		\item Hirota-Satsuma system: 
		\be\label{H-S system}
		\left\{\begin{array}{rll}
			u_{t}+ a_{1}u_{xxx} &=& -6a_{1}uu_{x}+c_{12}vv_{x},   \vspace{0.03in}\\
			v_{t}+v_{xxx} &=& -3uv_{x}, \vspace{0.03in}\\
			\left. (u,v)\right |_{t=0} &=& (u_{0},v_{0}),
		\end{array}\right.\ee
		where $a_{1}\neq 0$. This system was proposed by Hirota-Satsuma in \cite{HS81} to describe the interaction of two long waves with different dispersion relations.
		
		\item Gear-Grimshaw system: 
		\be\label{G-G system}
		\left\{\begin{array}{rcl}
			u_{t}+u_{xxx}+\sigma_{3}v_{xxx} &=&-uu_{x}+\sigma_{1}vv_{x}+\sigma_{2}(uv)_{x},\vspace{0.03in}\\
			\rho_{1}v_{t}+\rho_{2}\sigma_{3}u_{xxx}+v_{xxx}+\sigma_{4}v_{x} &=& \rho_{2}\sigma_{2}uu_{x}-vv_{x}+\rho_{2}\sigma_{1}(uv)_{x}, \vspace{0.03in}\\
			\left. (u,v)\right |_{t=0} &=& (u_{0},v_{0}),
		\end{array}\right.\ee
		where $\sigma_{i}\in\m{R}(1\leq i\leq 4)$ and $\rho_{1},\,\rho_{2}>0$. This system is a special case of (\ref{ckdv, general}) by setting 
		\be\label{A1 in G-G system}
		A_{1}=\bp 1& \sigma_{3}\\ \frac{\rho_{2}\sigma_{3}}{\rho_{1}} & \frac{1}{\rho_{1}} \ep.\ee
		Note that $A_{1}$ in (\ref{A1 in G-G system}) is diagonalizable over $\m{R}$ for any $\sigma_{3}\in\m{R}$ and $\rho_{1},\,\rho_{2}>0$. Moreover, the eigenvalues of $A_{1}$ are nonzero unless $\rho_{2}\sigma_{3}^{2}=1$. So (\ref{G-G system}) can be reduced to the form (\ref{ckdv, coef form}) as long as $\rho_{2}\sigma_{3}^{2}\neq 1$. This system was derived by Gear-Grimshaw in \cite{GG84} (also see \cite{BPST92} for the explanation about the physical context) as a model to describe the strong interaction of two-dimensional, weakly nonlinear, long, internal gravity waves propagating on neighboring pycnoclines in a stratified fluid, where the two waves correspond to different modes. 
	\end{itemize}

	In this paper  we study the well-posedness of the Cauchy problem (\ref{ckdv, coef form}) in the space $$H^s (\m{R})\times H^s(\m{R})\triangleq \mathcal{H}^{s}(\m{R}).$$
	The Cauchy problem  (\ref{ckdv, coef form})  can be viewed as a special  example of the following abstract Cauchy problem,
	\begin{equation} \label{x-1}
		\frac{dw}{dt} +L w = N(w) , \qquad w(0)= \phi, 
	\end{equation}
	where $L$  is a linear operator, $N $ is a possibly time-dependent nonlinear operator
	and the initial datum $\phi $  belongs to a Banach space $\m{X}_s $ with index $s \in\m{R}$. The scale of Banach spaces $\m{X}_s $ has the property $\m{X}_{s_2}\subset \m{X}_{s_1}$ if  $s_1 \leq  s_2$. The well-posedness considered in this paper is understood in the following sense.
	
	\begin{definition} \label{def2}  The Cauchy problem  (\ref{x-1}) is said to be well-posed in the
		space $\m{X}_s$  if for any $ \delta  > 0 $ there is a  $T = T(\delta ) > 0 $ such that
		\begin{itemize}
			\item[(a)]  for any $\phi \in \m{X}_s$  with $ \|\phi\|_{\m{X}_s} \leq \delta $, (\ref{x-1})  admits exactly one solution $w$
			in the space  $C([0,T]; \m{X}_s) $  satisfying the auxiliary condition
			\begin{equation} \label{x-2} w\in \m{Y}^T_s  \end{equation}
			where  $\m{Y}^T_s$  is an auxiliary metric space;
			\item[(b)] the solution $w$ depends continuously on its initial data $\phi$  in the sense
			that the mapping $\phi \to  u $ is continuous from $\{\phi :\|\phi\|_{\m{X}_s} \leq \delta \}$ to $C([0, T];\m{X}_s)$.
		\end{itemize}
	\end{definition}
	
	The well-posedness described by Definition \ref{def2} is local in character since the time $T$ depends on $\delta $. If  $T$ can be specified independently of $\delta $ in Definition \ref{def2}, then (\ref{x-1}) is said to be globally well-posed in the space $\m{X} _s$. On the other hand, the Cauchy problem (\ref{x-1})  is said to be {\it  (locally) uniformly well-posed}, $C^k $-well-posed ($k\geq 0$), or {\it analytically} well-posed in the space $\m{X}_s$  if the corresponding  solution map  is (locally) uniform continuous, $C^k$ or real analytic. 
	
	In this paper, we are looking for an answer to the following  problem.
	
	\noindent
	{\bf Problem}:  
	{\em For what values of  $s\in \m{R}$  is  the Cauchy problem  (\ref{ckdv, coef form})
		well-posed  in  the space $\mathcal{H}^{s}(\m{R})$?}
	

	\subsection{Literature review}
	\label{Subsec, literature}
	\quad
	It is beneficial and instructive to  the study of  the Cauchy problem (\ref{ckdv, coef form}) by first reviewing the well-posedness of the  Cauchy  problem of the single KdV equation
	
	\be\label{KdV eq}
	u_t+\ uu_x +u_{xxx} = 0, \qquad u(x,0)= u_0 (x) 
	\ee 
	posed either on  the whole  real  line $\m{R}$ or on a periodic domain  $\m{T}$.
	The study began  in   the late 1960s  with the work of Sj\"oberg  \cite{Sjo67,Sjo70}  and has come to a happy end with the  work of Killip and Visan\cite{KV19}.  
	Looking back, this  study, which has lasted more than half  a century,  can be divided into four stages with  four different  major approaches  developed  in the process.      
	
	In Stage 1,    (\ref{KdV eq}) was  most studied  using   traditionally PDE  and functional analysis techniques.  Sj\"oberg  \cite{Sjo67,Sjo70} and Temam \cite{Tem69}  (see \cite{Kam69, TMI70, TM71, ST76} and the references therein for some  other works followed)  obtained the existence and uniqueness of solutions of (\ref{KdV eq})  on $\m{T}$ 
	in the space $L^{\infty} (0, T; H^3 (\m{T})) $ (instead of in the space $C([0,T]; H^3 (\m{T}))$), 
	but without  showing  the continuity  of the associated solution map.
	The first well-posedness result was due to  Bona and Smith \cite{BS75} who showed that (\ref{KdV eq})  is (globally) well-posed in the space $H^k (\m{R})$  or $H^k (\m{T})$ for any integer $k\geq 2$ using  a cleverly designed  regularization scheme
	and classical energy estimate method. Then,  (\ref{KdV eq})  was  shown by Bona and Scott \cite{BS76}  to be (globally) well-posed in the space $H^s (\m{R})$  or $H^s (\m{T})$ for any real number $s\geq 2$   using Tartar's  nonlinear  interpolation  
	theory \cite{Tar72}. After this, as one of the applications  of the semigroup theory, which is a powerful general  theory dealing  with various  quasi-linear evolutionary PDEs,  Kato\cite{Kat75, Kat79, Kat81, Kat83}  showed that (\ref{KdV eq}) is  locally well-posed in  $H^s (\m{R})$  or $H^s (\m{T})$ for any $s> \frac32$. 
	
	In Stage 2,  as various smoothing properties  of dispersive wave equations   were discovered in 1980s (cf. \cite{Kat83, CS88, Sjo87, KPV91Indiana},  Kenig, Ponce and Vega\cite{KPV89, KPV91JAMS, KPV93CPAM}  were able to exploit  the various  dispersive smoothing 
	properties of the  linear 
	KdV equation to show  that  (\ref{KdV eq}) 
	is locally  well-posed in the space $H^s (\m{R})$ for any  $s>\frac{3}{4}$ by applying  the  contraction mapping principle in a carefully constructed  Banach space,  now known as   the Kenig-Ponce-Vega  (or KPV)  space.  As one of the key linear estimate fails when $s<\frac34$,    one can only show that (\ref{KdV eq}) is well-posed in $H^s (\m{R}) $ for $s>\frac34$  using this approach.
	
	In Stage 3,  Bourgain\cite{Bou93b}  introduced the  Fourier restriction spaces $X_{s,b}$  and  showed that the Cauchy problem (\ref{KdV eq}) is well-posed in  both spaces $H^s (\m{R})$ and $H^s (\m{T})$ for any $s\geq 0$  by  using the  contraction mapping principle in  $X_{s,b}$.    Then Kenig, Ponce and Vega\cite{KPV96} showed   that (\ref{KdV eq}) is locally well-posed in  $H^s (\m{R})$ for any $s> -\frac34$ and in $H^s (\m{T})$ for any $s\geq -\frac12$. The  local well-posedness of (\ref{KdV eq}) in the space $H^{-\frac34} (\m{R})$  was established by Christ, Colliander and Tao\cite{CCT03}.  The thresholds $-\frac{3}{4}$ for $H^s(\m{R})$  and  $-\frac12$ for $H^s(\m{T})$ are  sharp if one requires the solution map to be uniformly continuous, see \cite{CCT03}.  Moreover,  (\ref{KdV eq}) has been shown to be globally well-posed in $H^s (\m{R})$ for $s\geq -\frac34$ and in $H^s (\m{T})$ for $s\geq -\frac12$ (see \cite{CKSTT03, Guo09, Kis09}).
	
	In Stage 4,  Kappeler and Topalov  \cite{KT06} proved that (\ref{KdV eq})  is globally well-posed in the space $H^s (\m{T})$ for any $s\geq -1$ by developing a new approach  based on the inverse scattering method. Recently, Killip and Visan \cite{KV19} showed that  (\ref{KdV eq}) is globally well-posed   in the space $H^{-1} (\m{R})$  by introducing a new method of general applicability for the study of low-regularity well-posedness for integrable PDE.   As it has already been shown by Molinet \cite{Mol11, Mol12}  that   (\ref{KdV eq}) is ill-posed in both $H^s (\m{R})$ and $H^s (\m{T})$ for any $s< -1$,  the study of the well-posedness of (\ref{KdV eq})  has drawn a satisfactory conclusion.
	
	There is a difference between the well-posdenss  presented in Stages 1 and 4 and those  presented in Stages 2 and 3.  For the well-posedness obtained in Stage 1 and 4, 
	the solution of (\ref{KdV eq}) depends only continuously  on its initial value.  By  contrast,  for  the well-posedness established in Stage 2 and 3,  one can show   the solution of (\ref{KdV eq}) depends on its initial value analytically (cf.  \cite{Zha95JFA, Zha95DIE, Zha95SIMA}). Thus the Cauchy problem  (\ref{KdV eq}) is   analytically well-posed in $H^s (\m{R})$ for $s\geq -\frac34 $ and in $H^s (\m{T})$  for $s\geq -\frac12$, but  is only continuously well-posed in  $H^s (\m{R})$ for $-1\leq s<  -\frac34 $ and in $H^s (\m{T})$  for $-1\leq s < -\frac12$ .
	
	Naturally, following the  advances of  the study of the well-posedness of the  Cauchy problem (\ref{KdV eq}) for the single KdV equation,   there have been many works on the well-posedness of the Cauchy problem   (\ref{ckdv, coef form}) for the coupled KdV-KdV  systems. Here we provide a  brief summary  of the previous results on $\mcal{H}^{s}(\m{R})$. As a convenience of the notation, LWP and GWP will stand for local well-posedness and global well-posedness.
	\begin{itemize}
		\item  Majda-Biello system (\ref{M-B system}). 
		\begin{itemize}
			\item If $a_{2}=1$, the LWP in $\mcal{H}^{s}(\m{R})$ for any $s>-\frac{3}{4}$ follows immeidately from the single KdV theory. The GWP in $\mcal{H}^{s}$ for any $s>-\frac{3}{4}$ was justified by Oh\cite{Oh09b} via the I-method.
			
			\item If $a_{2}\in(0,4)\setminus\{1\}$, Oh\cite{Oh09a} proved that (\ref{M-B system}) is locally well-posed  in $\mcal{H}^{s}(\m{R})$ for $s\geq 0$ and  ill-posed when $s<0$ if the solution map  is required to be $C^{2}$.   The key ingredient in the proof for the LWP is the bilinear estimate under the Fourier restriction norm. Due to the $L^{2}$ conservation law of (\ref{M-B system}), its GWP in $\mcal{H}^{s}(\m{R})$ for $s\geq 0$ automatically holds.
		\end{itemize}
		
		\item Hirota-Satsuma system (\ref{H-S system}).
		\begin{itemize}
			
			\item Alvarez-Carvajal\cite{AC08} proved the LWP for (\ref{H-S system}) in $\mcal{H}^{s}(\m{R})$ for $s>\frac{3}{4}$ via the method in \cite{KPV91JAMS}.
			
			\item Feng\cite{Fen94} considered a slightly general system:
			\be\label{H-S system, general}
			\left\{\begin{array}{rll}
				u_{t}+ a_{1}u_{xxx} &=& -6a_{1}uu_{x}+c_{12}vv_{x},   \vspace{0.03in}\\
				v_{t}+v_{xxx} &=& c_{22}vv_{x}+d_{22}uv_{x}, \vspace{0.03in}\\
				\left. (u,v)\right |_{t=0} &=& (u_{0},v_{0}).
			\end{array}\right.\ee
			When $c_{22}=0$ and $d_{22}=-3$, (\ref{H-S system, general}) reduces to the original Hirota-Satusma system (\ref{H-S system}). Feng proved the LWP of (\ref{H-S system, general}) in $\mcal{H}^{s}(\m{R})$ for $s\geq 1$ under the assumption that $a_{1}\neq 1$ and $c_{12}d_{22}<0$. The GWP was also shown by the further restriction that $0<a_{1}<1$. 
		\end{itemize}
		
		\item Gear-Grimshaw system (\ref{G-G system}). 
		\begin{itemize}
			\item Assume $\sigma_{4}=0$ and $\rho_{2}\sigma_{3}^{2}\neq 1$. Bona-Ponce-Saut-Tom\cite{BPST92} proved the LWP of (\ref{G-G system}) in $\mcal{H}^{s}(\m{R})$ for $s\geq 1$. They also showed the GWP of (\ref{G-G system}) in $\mcal{H}^{s}(\m{R})$ for $s\geq 1$ under further  assumption that $\rho_{2}\sigma_{3}^{2}<1$.
			
			\item Later, further LWP and GWP results were proven by Ash-Cohen-Wang\cite{ACW96}, Linares-Panthee\cite{LP04} and Saut-Tzvetkov\cite{ST00}, where the best LWP result is proven in $\mcal{H}^{s}(\m{R})$ for $s>-\frac{3}{4}$. However, their argument essentially requires the matrix $A_{1}$ in (\ref{A1 in G-G system}) to be similar to the identity matrix, which means $\sigma_{3}=0$ and $\rho_{1}=1$. Equivalently, if considering the diagonalized system (\ref{ckdv, coef form}), their results are  only valid under the assumption that $a_{1}=a_{2}$ (see Remark 1.2  in \cite{Oh09a}  and Remark 3.1  in \cite{AC08}  for more detailed explanations).
		\end{itemize}
		\item General coupled KdV-KdV  systems
		\begin{itemize}
			\item Alvarez-Carvajal\cite{AC08} considered the diagonalized system (\ref{ckdv, coef form}) where $(b_{ij})=0$, $d_{11}=d_{12}$ and $d_{21}=d_{22}$, i.e.,  
			\be\label{G-G system, diag}
			\left\{\begin{array}{rcl}
				u_{t}+a_{1}u_{xxx} &=& c_{11}uu_{x}+c_{12}vv_{x}+d_{11}(uv)_{x},\vspace{0.03in}\\
				v_{t}+a_{2}v_{xxx} &=& c_{21}uu_{x}+c_{22}vv_{x}+d_{22}(uv)_{x}, \vspace{0.03in}\\
				\left. (u,v)\right |_{t=0} &=& (u_{0},v_{0}).
			\end{array}\right.\ee
			They proved that  (\ref{G-G system, diag})  is locally well-posed in $\mcal{H}^{s}(\m{R})$ for $s>-\frac{3}{4}$ if $a_{1}=-a_{2}\neq 0$. The key tool in their proof is the bilinear estimate under the Fourier restriction norm.    The question whether (\ref{G-G system, diag}) is  well-posed  in $\mcal{H}^{s}(\m{R})$  when $|a_{1}|\ne|a_{2}|$  is left open in \cite{AC08}. On the other hand,  Alvarez-Carvaja's result  in  \cite{AC08} actually does not apply to the  Gear-Grimshaw system (\ref{G-G system}) since $a_{1}=-a_{2}>0  $ implies $\rho_{1} =-1$ which is against the  assumption $\rho_{1} >0$.

		\end{itemize} 
		
	\end{itemize}
	
	\subsection{Main results on well-posedness}
	\label{Subsec, results on wp}
	
	\quad
	As we have seen from the literature review,  the  dispersion coefficients $a_{1}$ and $a_{2}$, and other coefficients $(b_{ij})$, $(c_{ij})$ and $(d_{ij})$, in the systems (\ref{ckdv,  coef form}) have significant impact  on  the well-posedness  results.  The following theorem is the main finding we have obtained so far.
	
	\begin{theorem}\label{Thm, lwp for ckdv}
		Let $a_{1},\,a_{2}\in\m{R}\setminus\{0\}$ and denote $r=\dfrac{a_{2}}{a_{1}}$. Then (\ref{ckdv, coef form}) is locally analytically  well-posed in $\mathcal{H}^{s}(\m{R})$ for any case in Table \ref{Table, lwp for ckdv}. 
		\begin{table}[!ht]
			\renewcommand\arraystretch{1.5}
			\begin{center}
				\begin{tabular}{|c|c|c|c|} \hline
					Case & $r=\frac{a_{2}}{a_{1}}$   & Coefficients $b_{ij}$, $c_{ij}$ and $d_{ij}$    & $s$    \\ \hline
					(1) & $r<0$ & \begin{tabular}{c} $(c_{ij})=0$, $d_{11}=d_{12}$ and $d_{21}=d_{22}$    \\ Otherwise \end{tabular} & \begin{tabular}{c} $s\geq -\frac{13}{12}$ \\ $s>-\frac{3}{4}$ \end{tabular}  \\\hline
					(2) & $0<r<\frac{1}{4}$ & \begin{tabular}{c} $c_{12}=d_{21}=d_{22}=0$     \\ Otherwise \end{tabular} & \begin{tabular}{c} $s>-\frac{3}{4}$ \\ $s\geq 0$ \end{tabular}  \\\hline
					(3) & $r=\frac{1}{4}$ & \begin{tabular}{c} $c_{21}=d_{11}=d_{12}=0$     \\ Otherwise \end{tabular} & \begin{tabular}{c} $s\geq 0$ \\ $s\geq\frac{3}{4}$ \end{tabular}  \\\hline
					(4) & $\frac{1}{4}<r<1$ & arbitrary & $s\geq 0$  \\\hline
					(5) & $r=1$ & \begin{tabular}{c} $b_{12}=b_{21}=0$, $d_{11}=d_{12}$ and $d_{21}=d_{22}$     \\ $b_{12}=b_{21}=0$, $d_{11}\neq d_{12}$ or $d_{21}\neq d_{22}$ \end{tabular} & \begin{tabular}{c} $s>-\frac{3}{4}$ \\ $s>0$ \end{tabular}  \\\hline
					(6) & $1<r<4$ & arbitrary & $s\geq 0$  \\\hline
					(7) & $r=4$ & \begin{tabular}{c} $c_{12}=d_{21}=d_{22}=0$     \\ Otherwise \end{tabular} & \begin{tabular}{c} $s\geq 0$ \\ $s\geq\frac{3}{4}$ \end{tabular}  \\\hline
					(8) & $r>4$ & \begin{tabular}{c} $c_{21}=d_{11}=d_{12}=0$     \\ Otherwise \end{tabular} & \begin{tabular}{c} $s>-\frac{3}{4}$ \\ $s\geq 0$ \end{tabular}  \\\hline
				\end{tabular}	 
			\end{center} 
			\caption{Main Results}
			\label{Table, lwp for ckdv}
		\end{table}
	\end{theorem}
	
	The well-posedness results presented in Theorem \ref{Thm, lwp for ckdv} are sharp in the sense that the key bilinear estimates used in their proofs are sharp (up to the endpoints), see Theorem \ref{Thm, d2, neg}--\ref{Thm, sharp bilin est, general}. 
	
	As applications,  we  apply  Theorem \ref{Thm, lwp for ckdv} to  a few  specializations of  (\ref{ckdv, coef form}). First,  we consider a special class of (\ref{ckdv, coef form}) of the following form
	\be\label{weak nonlin ckdv}
	\left\{\begin{array}{rcl}
		u_{t}+a_{1}u_{xxx} &=& d_{1}(uv)_{x},\vspace{0.03in}\\
		v_{t}+a_{2}v_{xxx} &=& d_{2} (uv)_{x}, \vspace{0.03in}\\
		\left. (u,v)\right |_{t=0} &=& (u_{0},v_{0}).
	\end{array}\right.\ee
	
	\begin{theorem}\label{Thm, weak nonlin ckdv}
		If $a_{1}a_{2}<0$, then the system (\ref{weak nonlin ckdv}) is locally analytically well-posed in $\mathcal{H}^{s}(\m{R})$ for $s\geq -\frac{13}{12}$.
	\end{theorem}
	
	The above theorem is surprising  since even the Cauchy problem  (\ref{KdV eq}) of the single KdV equation  is ill-posed in $H^s(\m{R}) $ for any $s<-1$.
	
	\begin{theorem}\label{Thm, M-B WP}
		The Majda-Biello system (\ref{M-B system}), where $a_2\neq 0$, is locally (resp. globally) analytically well-posed in $\mathcal{H}^{s}(\m{R})$ for any case in Table \ref{Table, lwp for M-B} (resp. Table \ref{Table, gwp for M-B}).
		\begin{table}[!ht]
			\begin{floatrow}
				\capbtabbox{
					\renewcommand\arraystretch{2}
					\begin{tabular}{|c|c|c|} \hline
						Case & Coefficient $a_2$  & $s$    \\ \hline
						(1) & $a_2\in(-\infty,0)\cup\{1\}\cup\{4,\infty\}$ & $s>-\frac34$  \\\hline
						(2) & $a_2\in(0,1)\cup(1,4)$ & $s\geq 0$  \\\hline
						(3) & $a_2=4$ & $s\geq\frac{3}{4}$ \\\hline
					\end{tabular}	 
				}
				{\caption{LWP Results}
					\label{Table, lwp for M-B}
				}
				\capbtabbox{
					\renewcommand\arraystretch{2}
					\begin{tabular}{|c|c|c|} \hline
						Case & Coefficient $a_2$  & $s$    \\ \hline
						(1) & $a_2=1$ & $s>-\frac34$  \\\hline
						(2) & $a_2\not\in\{1,4\}$ & $s\geq 0$  \\\hline
						(3) & $a_2=4$ & $s\geq1$ \\\hline
					\end{tabular}	 
				}
				{\caption{GWP Results}
					\label{Table, gwp for M-B}
				}
			\end{floatrow}
		\end{table}
	\end{theorem}
	
	Remark: in Theorem \ref{Thm, M-B WP}, Case (1) and (2) in Table \ref{Table, lwp for M-B} and  \ref{Table, gwp for M-B} have been known earlier in Oh\cite{Oh09a, Oh09b}.
	
	\begin{theorem}\label{Thm, H-S WP}
		The Hirota-Satsuma systems (\ref{H-S system}), where $a_1\neq 0$, is  locally (resp. globally) analytically well-posed in $\mathcal{H}^{s}(\m{R})$ for any case in Table \ref{Table, lwp for H-S} (resp. Table \ref{Table, gwp for H-S}).
		\begin{table}[!ht]
			\begin{floatrow}
				\capbtabbox{
					\renewcommand\arraystretch{2}
					\begin{tabular}{|c|c|c|} \hline
						Case & Coefficients $a_1$ and $c_{12}$  & $s$    \\ \hline
						(1) & $a_1\in(-\infty,0)\cup(0,\frac14)$ & $s>-\frac34$  \\\hline
						(2) & $a_1\in(\frac14,1)\cup(1,\infty)$ & $s\geq 0$  \\\hline
						(3) & $a_1=1$ & $s>0$ \\ \hline
						(4) & $a_1=\frac14$ & $s\geq\frac{3}{4}$ \\\hline
					\end{tabular}	 
				}
				{\caption{LWP Results}
					\label{Table, lwp for H-S}
				}
				\capbtabbox{
					\renewcommand\arraystretch{2}
					\begin{tabular}{|c|c|c|} \hline
						Case & Coefficients $a_1$ and $c_{12}$  & $s$    \\ \hline
						(1) & $a_1\not\in\{\frac14, 1\}$,\, $c_{12}>0$ & $s\geq 0$  \\\hline
						(2) & $a_1=\frac14$,\, $c_{12}>0$ & $s\geq 1$  \\\hline
					\end{tabular}	 
				}
				{\caption{GWP Results}
					\label{Table, gwp for H-S}
				}
			\end{floatrow}
		\end{table}
	\end{theorem}
	
	We finally turn to the Gear-Grimshaw  system (\ref{G-G system}) and introduce the condition (\ref{G-G, r=1/4}) for convenience.
	\be\label{G-G, r=1/4}
	\rho_2\sigma_3^2\leq \frac{9}{25} \quad\text{and}\quad \rho_{1}^2+\frac{25\rho_2\sigma_3^2-17}{4}\rho_1+1=0.\ee
	
	\begin{theorem}\label{Thm, G-G WP}
		The Gear-Grimshaw system (\ref{G-G system}), where $\rho_1,\,\rho_2> 0$, is locally (resp. globally) analytically well-posed in $\mathcal{H}^{s}(\m{R})$ for any case in Table \ref{Table, lwp for G-G} (resp. Table \ref{Table, gwp for G-G}).
		\begin{table}[!ht]
			\begin{floatrow}
				\capbtabbox{
					\renewcommand\arraystretch{2}
					\begin{tabular}{|c|c|c|} \hline
						Case & $\rho_1$, $\rho_2$ and $\sigma_{i} (1\leq i\leq 4)$  & $s$    \\ \hline
						(1) & $\sigma_3=0$, $\rho_1=1$ & $s>-\frac34$  \\\hline
						(2) & $\rho_2\sigma_3^2>1$ & $s>-\frac34$  \\\hline
						(3) & $\rho_2\sigma_3^2<1$, (\ref{G-G, r=1/4}) fails & $s\geq 0$ \\ \hline
						(4) & $\rho_2\sigma_3^2<1$, (\ref{G-G, r=1/4}) holds & $s\geq\frac{3}{4}$ \\\hline
					\end{tabular}	 
				}
				{\caption{LWP Results}
					\label{Table, lwp for G-G}
				}
				\capbtabbox{
					\renewcommand\arraystretch{2}
					\begin{tabular}{|c|c|c|} \hline
						Case & $\rho_1$, $\rho_2$ and $\sigma_{i} (1\leq i\leq 4)$  & $s$    \\ \hline
						(1) & $\rho_2\sigma_3^2\neq 1$, (\ref{G-G, r=1/4}) fails & $s\geq 0$  \\\hline
						(2) & $\rho_2\sigma_3^2\neq 1$, (\ref{G-G, r=1/4}) holds & $s\geq 1$  \\\hline
					\end{tabular}	 
				}
				{\caption{GWP Results}
					\label{Table, gwp for G-G}
				}
			\end{floatrow}
		\end{table}
	\end{theorem}
	
	It  should be pointed out that  Case (1) in Table \ref{Table, lwp for G-G} is trivial since it directly follows from the proof of the single KdV case.
	
	\newpage
	
	\subsection{Remarks}
	
	A few remarks are now in order.
	\begin{remark} \label{Re, gwp}
		While the results presented in  Section 1.3  provides a  rather thorough description of the analytically well-posedness in  $\mathcal{H}^s (\mathbb{R})$  for the systems (\ref{ckdv, coef form}), the study of the well-posedness of the Cauchy problem of (\ref{ckdv, coef form}) in  $\mathcal{H}^s (\mathbb{R})$ is far from over in comparison to the study of the KdV equation \eqref{KdV eq}.	We list below a few problems  among many to be investigated.
	\end{remark} 
	\begin{itemize}
		\item  Question 1.1: For the locally analytically well-posedness results of (\ref{ckdv, coef form}) listed in Table 1,  it  requires $s> -\frac34 $  in Cases (1), (2), (5) and (8).  {\em Can those results be strengthened to include $s=-\frac34$?}
		
		\item  Question 1.2:  The  locally analytically well-posedness results of the  systems (\ref{ckdv, coef form}) listed in Table  1 are sharp in the sense that the needed bilinear estimates, a key ingredient in the proofs,  fail if $s$ is less than the corresponding critical index $s^*$. 
		{\em Is the Cauchy problem of (\ref{ckdv, coef form})  analytically ill-posed in the space $\mathcal{H}^s (\m{R})$ for any $s$ which is less than the corresponding critical index $s^*$? }
		
		\item Question 1.3:  {\em  Can those locally analytically well-posedness results of the systems (\ref{ckdv, coef form}) listed in Table 1 be strengthened to be globally analytically well-posed results?}
		
	\end{itemize}

	\begin{remark}   As hinted by the study of the single KdV equation,  the answers to both Question 1.1 and Question 1.2  will most likely be  positive.  For Question 1.1,  some more subtly modified Bourgain spaces may need to be constructed,  see e.g. \cite{Guo09, Kis09}.  For Question1.2, some counter examples are needed to show that the solution map fails to be smooth if $s$ is less than the corresponding critical index $s^*$.  We leave this study to future works since the current paper is already long.
	\end{remark}

	\begin{remark}  For Question 1.3,  as long as one can  establish a priori global $\mathcal{ H}^s(\mathbb{R})$  estimates for solutions of the system (\ref{ckdv, coef form}), the GWP of (\ref{ckdv, coef form})  in  $\mathcal{H}^s(\mathbb{R})$ follows from  the corresponding LWP result.
		In particular,  when there are conserved energy at certain regularity level, the corresponding GWP can be easily verified. For example, we also include some GWP results in Theorem \ref{Thm, M-B WP}--\ref{Thm, G-G WP}.  But if the regularity considered in the well-posedness problem is lower than the level provided by the available conserved energy,  one may need to apply other methods, such as the I-method \cite{CKSTT03}, to establish the GWP. 
	\end{remark}
	
	
	\begin{remark}
		The single KdV equation has also been  intensively studied  from control point of views for its controllability and stabilizability (the interested readers are referred to \cite{contr-1, control-2, control-3, control-4, control-5, control-6, control-7, control-8} and the references therein for  an  overview of this subject). Various tools developed in the study of the welll-posedness of the single KdV equation  have played important roles in studying control theory of the KdV equation. By contrast, there are few studies of the  systems (\ref{ckdv, coef form}) from control points of view. We expect the results and the tools obtained and developed in the study of the well-posedness of the Cauchy problem of (\ref{ckdv, coef form}) will stimulate and play important roles in further studies of the control theory for the coupled KdV-KdV systems.
	\end{remark} 
	
	\subsection{Organization}
	
	The remaining of the paper is organized as follows. In Section \ref{Sec, Pre}, some linear estimates are recalled or proved as a preparation. In Section \ref{Sec, results on bilin est}, we present our main results on the bilinear estimates which are the key ingredients
	in the proof of the main well-posedness result: Theorem \ref{Thm, lwp for ckdv}. The proofs of these bilinear estimates will be postponed to Sections \ref{Sec, proof for bilin est} and \ref{Sec, sharp of bilin est}. In Section \ref{Sec, proof of wp}, we prove Theorem \ref{Thm, lwp for ckdv}, and its consequences, Theorem \ref{Thm, weak nonlin ckdv}--\ref{Thm, G-G WP}. Section \ref{Sec, proof for bilin est} is devoted to establish the various bilinear estimates, Theorem \ref{Thm, d2, neg} and \ref{Thm, bilin est, general}, presented in Section \ref{Sec, results on bilin est}. Finally, Section \ref{Sec, sharp of bilin est} is dedicated to justify Theorem \ref{Thm, sharp d2, neg} and \ref{Thm, sharp bilin est, general} which exposit the sharpness of the various bilinear estimates.
	
	\section{Preliminaries}
	\label{Sec, Pre}
	
	Let $\psi\in C_{0}^{\infty}(\m{R})$ be a bump function supported on $[-2,2]$ with $\psi=1$ on $[-1,1]$. We will use $C$ and $C_{i}(i\geq 1)$ to denote the constants. Moreover, $C=C(a,b\dots)$ means the constant C only depends on $a,b\dots$. We use $A\lesssim B$ to denote an estimate of the form $A\leq CB$. The notation $A\gtrsim B$ is used similarly. In addition, we will write $A\sim B$  if  $A\lesssim B$ and $B\lesssim A$. Finally, the notation $\la\cdot\ra$ means $=1+|\cdot|$.
	
	
	Consider the Cauchy problem of the   following linear KdV equation with  $\a,\b\in\m{R}$ and $\a\neq 0$.
	\be\label{linear eq}
	\left\{\begin{array}{ll}
		w_{t}+\alpha w_{xxx}+\beta w_{x}=0, & x\in\m{R},\, t\in\m{R},\\
		w(x,0)=w_{0}(x).
	\end{array}\right.\ee
	For any $w_0 \in H^s (\m{R})$, it admits a unique solution  $w\in C_b (\m{R}; H^s (\m{R}))$ for any $s\in \m{R}$, which can be written as 
	\be\label{semigroup op}
	w(x,t)= S^{\alpha,\beta}(t)w_{0}(x)=\int_{\m{R}}e^{i\xi x}e^{i\phi^{\alpha,\beta}(\xi)t}\,\widehat{w_{0}}(\xi)d\xi,\ee
	where $\phi^{\a,\b}(\xi) = \alpha \xi ^3 -\beta \xi $.

	\begin{lemma}\label{Lemma, lin est for KdV}
		For any $\a\neq 0$, $b>\frac{1}{2}$, $s,\,\b\in\m{R}$, there exists $C=C(b)$ such that
		\be\label{lin est for free KdV}
		\|\psi(t)S^{\a,\b}(t)w_{0}\|_{X^{\a,\b}_{s,b}}\leq C\| w_{0}\|_{H^{s}(\m{R})}\ee
		and
		\be\label{lin est for Duhamel int term}
		\Big\| \psi(t)\int_{0}^{t}S^{\a,\b}(t-t')F(t')dt' \Big\|_{X^{\a,\b}_{s,b}}\leq C\| F\|_{X^{\a,\b}_{s,b-1}}.\ee
	\end{lemma}
	\begin{proof}
		The proof follows exactly as Lemma 3.1 and Lemma 3.3 in \cite{KPV93Duke}.
	\end{proof}
	
	Strictly speaking, the constant in the above lemma also depends on $\psi$.  However,  we will not track the dependence of the constant on it  since $\psi$ is a fixed bump function throughout this paper.
	
	\begin{lemma}\label{Lemma, est for lin term}
		Let $\a_{1},\,\a_{2}\in \m{R}\setminus\{0\}$ with $\a_{1}\neq \a_{2}$. Then there exist $\eps=\eps(\a_{1},\a_{2})$ and $C=C(\a_{1},\a_{2})$ such that for any $s\in\m{R}$, $\frac{1}{2}<b\leq \frac{2}{3}$,  and for any $\b_{1}$ and $\b_{2}$ with $|\b_2-\b_{1}|\leq\eps$,
		\be\label{est for lin term}
		\|\p_{x}w\|_{X^{\a_{2},\b_{2}}_{s,b-1}}\leq C\,\|w\|_{X^{\a_{1},\b_{1}}_{s,b}}, \quad\forall\,w\in X^{\a_{1},\b_{1}}_{s,b}.\ee
	\end{lemma}
	\begin{proof}
		By duality and Plancherel identity, it is equivalent to prove for any $g\in X^{\a_{2},\b_{2}}_{-s,1-b}$,
		\[\bigg|\i_{\m{R}}\i_{\m{R}}\xi \wh{w}(\xi,\tau)\wh{g}(\xi,\tau)\,d\xi\,d\tau\bigg| \leq C\|w\|_{X^{\a_{1},\b_{1}}_{s,b}}\|g\|_{X^{\a_{2},\b_{2}}_{-s,1-b}}. \]
		Let 
		\[f_{1}(\xi,\tau)=\la\xi\ra^{s}\la L_{1}\ra^{b}\wh{w}(\xi,\tau)\quad\text{and}\quad f_{2}(\xi,\tau)=\la\xi\ra^{-s}\la L_{2}\ra^{1-b}\wh{g}(\xi,\tau),\]
		with  $L_{i}=\tau-\phi^{\a_{i},\b_{i}}(\xi)$ for $i=1,2$.  It reduces to show 
		\be\label{weighted l2 for lin}
		\bigg|\i_{\m{R}}\i_{\m{R}}\frac{\xi f_{1}(\xi,\tau)f_{2}(\xi,\tau)}{\la L_{1}\ra^{b}\la L_{2}\ra^{1-b}}\,d\xi\,d\tau\bigg| \leq C\prod_{i=1}^{2}\|f_{i}\|_{L^{2}(d\xi d\tau)}.\ee
		By Holder's inequality, it suffices to verify 
		\be\label{ineq for lin}
		\sup_{\xi,\tau\in\m{R}}\frac{|\xi|}{\la L_{1}\ra^{b}\la L_{2}\ra^{1-b}}\leq C.\ee
		If $|\xi|\leq 1$, then (\ref{ineq for lin}) holds for $C=1$. If  $|\xi|>1$, then it follows from $\frac{1}{2}<b\leq \frac{2}{3}$ that
		\[\la L_{1}\ra^{b}\la L_{2}\ra^{1-b}\geq \big(\la L_{1}\ra\la L_{2}\ra\big)^{\frac{1}{3}}\geq \la L_{1}-L_{2}\ra^{\frac{1}{3}}.\]
		Since $\a_{1}\neq \a_{2}$, then it is easy to see that when $|\b_2-\b_{1}|\leq\eps$ for a sufficiently small $\eps=\eps(\a_{1},\a_{2})$, we have
		\[\la L_{1}-L_{2}\ra = |(\a_{2}-\a_{1})\xi^{3}-(\b_{2}-\b_{1})\xi|+1\geq \frac{|\a_{2}-\a_{1}|}{2}\,|\xi|^{3}.\]
		Thus (\ref{ineq for lin}) also holds when $|\xi |>1$.
	\end{proof}
	
	\begin{proposition}\label{Prop, false of est for lin term}
		If $\a_{1}=\a_{2}\neq 0$, then for any $s,b,\b_{1},\b_{2}\in\m{R}$, there does not exist a constant $C=C(\a_1,\a_2,s,b,\b_1,\b_2)$ such that (\ref{est for lin term}) holds. 
	\end{proposition}
	\begin{proof}
		Let  $\a_{1}=\a_{2}:=\a$. If there exist $s,b,\b_{1},\b_{2}\in\m{R}$ such that (\ref{est for lin term}) holds for some  constant $C$, then (\ref{weighted l2 for lin}) needs to be true  for any $f_j \in L^2 (\m{R}\times \m{R})$, $j=1,2$. We will only prove the statement in the case when $ b\geq \frac12 $ since the situation when $ b<\frac12 $ is similar.  When $b\geq \frac{1}{2}$,  for any $N\geq 2$, define $f_{1}(\xi,\tau)=f_{2}(\xi,\tau)=\mb{1}_{E}(\xi,\tau)$ with 
		\[E=\{(\xi,\tau)\in\m{R}^{2}: N-1\leq \xi\leq N,\,|\tau-\a\xi^{3}+\b_{1}\xi|\leq 1\},\]
		then for any $(\xi,\tau)\in E$, $|L_{1}|\leq 1$ and 
		$|L_{2}|=|L_{1}+(\b_{2}-\b_{1})\xi|\ls N$. In addition, the area of $E$ is 2 by direct calculation. As a result, the right hand side of (\ref{weighted l2 for lin}) equals $2C$ while its left hand side has the following lower bound:
		\[\bigg|\i_{\m{R}}\i_{\m{R}}\frac{\xi f_{1}(\xi,\tau)f_{2}(\xi,\tau)}{\la L_{1}\ra^{b}\la L_{2}\ra^{1-b}}\,d\xi\,d\tau\bigg|\gs \frac{N}{N^{1-b}}= N^{b},\]
		which is impossible when $N\rightarrow\infty$.
		
	\end{proof}

	\section{Main results on bilinear estimates}
	\label{Sec, results on bilin est}
	\quad
	Our main well-posedness results in Theorem \ref{Thm, lwp for ckdv} will be proved using  the same  approach  as that developed by Bourgain \cite{Bou93a, Bou93b}, Kenig-Ponce -Vega \cite{KPV96} in establishing  analytical well-posedness of the Cauchy problem of (\ref{KdV eq}) in the space $H^s (\m{R})$ for $s>-\frac34$. The key ingredient  in the approach is the bilinear estimate under the Fourier restriction space (also called Bourgain space). Let us first introduce the definition of this space. For any $\a,\,\b\in\m{R}$ with $\a\neq 0$, denote the polymomial $\phi^{\a,\b}$ as
	\be\label{phase fn}
	\phi^{\alpha,\beta}(\xi)=\alpha \xi^{3}-\beta \xi.\ee
	For convenience, $\phi^{\a,0}$ will be denoted as $\phi^{\a}$. Then the Fourier restriction space is defined as follows.
	\begin{definition}
		For any $\a,\,\beta,\,s,\,b\in\m{R}$ with $\a\neq 0$, the Fourier restriction space $X^{\a,\b}_{s,b}$ is defined to be the completion of the Schwartz space $\mathscr{S}(\m{R}^{2})$ with respect to the norm
		\be\label{Fourier R-norm}
		\|w\|_{X^{\a,\b}_{s,b}}=\|\la \xi\ra ^{s}\la\tau-\phi^{\a,\b}(\xi)\ra ^{b}\wh{w}(\xi,\tau)\|_{L^{2}(d\xi d\tau)},\ee
		where $\la\cdot\ra=1+|\cdot|$, $\phi^{\a,\b}$ is given by (\ref{phase fn}), and $\wh{w}$ refers to the space-time Fourier transform of $w$. Moreover, $X^{\a,0}_{s,b}$ is simply denoted as $X^{\a}_{s,b}$. On the other hand, for any $T>0$,  $X^{\a,\b}_{s,b} ([0,T])$ denotes the
		restriction of $X^{\a,\b}_{s,b}$ on the domain $\m{R} \times [0,T]$   which is a Banach space when equipped with the usual quotient norm.
	\end{definition}
	The bilinear estimate which was first considered by Bourgain\cite{Bou93b} is the following one:
	\begin{equation} \label{classical bilin}
		\| \partial _x (w_1w_2) \| _{X^{1}_{s,b-1} } \leq C \|w_1\|_{X^{1}_{s,b} }  \|w_2\|_{X^{1}_{s,b} }, \quad\forall w_1,w_2.
	\end{equation}
	Bourgain proved (\ref{classical bilin}) for $s=0$ and $b=\frac12$ while the following lemma is due to Kening, Ponce and Vega.
	
	\begin{lemma} [Kenig-Ponce-Vega \cite{KPV96}] \label{kpv-lemma}
		The bilinear estimate (\ref{classical bilin}) holds for any $s> -\frac34$ and $b\in \big(\frac12, b_{0}(s)\big)$ with some $b_0(s)>\frac12$, but fails for any $b\in \m{R}$ if $s<-\frac34$.
	\end{lemma}

	In order to deal with the general KdV-KdV systems (\ref{ckdv, coef form}), four types of bilinear estimates need to be investigated. In (\ref{d1})-(\ref{nd2}), $(D)$ represents divergence form and (ND) refers to non-divergence form.
	\begin{eqnarray}
		&&\text{(D1):}\qquad \|\p_{x}(w_{1}w_{2})\|_{X^{\a_{2},\b_{2}}_{s,b-1}}\leq C \|w_{1}\|_{X^{\a_{1},\b_{1}}_{s,b}}\|w_{2}\|_{X^{\a_{1},\b_{1}}_{s,b}}, \quad\forall\,w_1,w_2.\label{d1} \\
		&&\text{(D2):}\qquad \|\p_{x}(w_{1}w_{2})\|_{X^{\a_{1},\b_{1}}_{s,b-1}}\leq C \|w_{1}\|_{X^{\a_{1},\b_{1}}_{s,b}}\|w_{2}\|_{X^{\a_{2},\b_{2}}_{s,b}}, \quad\forall\,w_1,w_2.\label{d2}
	\end{eqnarray}
	and 
	\begin{eqnarray}
		&&\text{(ND1):}\qquad
		\|(\p_{x}w_{1})w_{2}\|_{X^{\a_{1},\b_{1}}_{s,b-1}}\leq C \| w_{1}\|_{X^{\a_{1},\b_{1}}_{s,b}}\|w_{2}\|_{X^{\a_{2},\b_{2}}_{s,b}}, \quad\forall\,w_1,w_2.\label{nd1}\\
		&&\text{(ND2):}\qquad
		\|w_{1}(\p_{x}w_{2})\|_{X^{\a_{1},\b_{1}}_{s,b-1}}\leq C \| w_{1}\|_{X^{\a_{1},\b_{1}}_{s,b}}\|w_{2}\|_{X^{\a_{2},\b_{2}}_{s,b}}, \quad\forall\,w_1,w_2.\label{nd2}
	\end{eqnarray}
	Here, $(\a_{1},\b_{1})$ \big(or $(\a_{2},\b_{2})\big)$ stands for $(a_{1},b_{11})$ or $(a_{2},b_{22})$. (D1) is used to deal with the square terms $uu_{x}$ and $vv_{x}$ in (\ref{ckdv, coef form}). (D2) is responsible for the mixed divergence term $(uv)_{x}$ when $d_{11}=d_{12}$ or $d_{21}=d_{22}$ in (\ref{ckdv, coef form}). (ND1) and (ND2) are applied to treat the mixed non-divergence terms $u_{x}v$ and $uv_{x}$ when $d_{11}\neq d_{12}$ or $d_{21}\neq d_{22}$. On the other hand, (D1) is different from (D2) since $w_{1}$ and $w_{2}$ live in the same space $X^{\a_{1},\b_{1}}_{s,b}$ for (D1) but in different spaces for (D2). (ND1) is also slightly different from (ND2). Nevertheless, due to the relation $(w_{1}w_{2})_{x}=(\p_{x}w_{1})w_{2}+w_{1}(\p_{x}w_{2})$,  any results for (ND2) can be automatically obtained once the corresponding results are known for (D2) and (ND1). The main challenges of studying the bilinear estimates (\ref{d1})-- (\ref{nd2}) come from either the distinct dispersion coefficients $\alpha_1$ and $\alpha_2$ or the non-divergence form.
	
	\begin{theorem}\label{Thm, d2, neg}
		Let  $\a_{1} \a_{2} <0$. Assume $s$ and $b$ satisfy one of the following conditions.
		\begin{enumerate}[(1)]
			\item $-\frac{13}{12}\leq s\leq -1$ and $\frac{1}{4}-\frac{s}{3}\leq b\leq \frac{4}{3}+\frac{2s}{3}$;
			\item $-1< s<-\frac{3}{4}$ and $\frac{1}{4}-\frac{s}{3}\leq b\leq 1+\frac{s}{3}$;
			
			\item $s\geq -\frac{3}{4}$ and $\frac{1}{2}<b< \frac{3}{4}$.
		\end{enumerate}
		Then there exist $\eps=\eps(\a_{1},\a_{2})$ and $C=C(\a_{1},\a_{2},s,b)$ such that for any $|\b_2-\b_{1}|\leq\eps$, (\ref{d2}) holds.
	\end{theorem}
	
	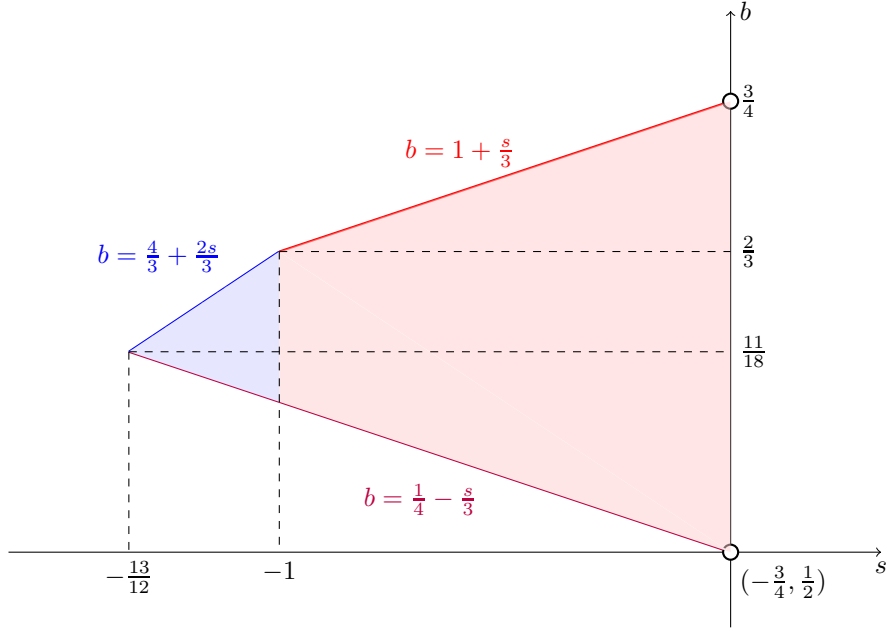
\begin{figure}[!ht]
		\centering
		\begin{tikzpicture}[scale=1.2]
			\draw[->] (-4.8,0) -- (1.4,0) node[anchor=north] {$s$};
			\draw[->] (0,-0.5) -- (0,3.6) node[anchor=west] {$b$};
			\draw (0,-0.2) node[right] {$(-\frac34,\frac12)$};
			\draw[thick, fill=white] (0,0) circle [radius=0.05];
			\draw[thick, fill=white] (0,3) circle [radius=0.05];
			\draw[thick, purple, domain=(-4:-0.05), samples=100] plot ({\x},{-(\x)/3});
			\draw (-2.9,0.35) node[right] {$\color{purple} b=\frac{1}{4}-\frac{s}{3}$};
			\draw[thick, blue, domain=(-4:-3), samples=100] plot ({\x},{2*(\x)/3+4});
			\draw (-4.1,1.65) node[above] {$\color{blue} b=\frac{4}{3}+\frac{2s}{3}$};
			\draw[thick, red, domain=(-3:-0.05), samples=100] plot ({\x},{(\x)/3+3});
			\draw (-2,2.5) node[above] {$\color{red} b=1+\frac{s}{3}$};
			\fill[red!20, opacity=0.5]  plot[domain=(-3:0), samples=100] ({\x},{(\x)/3+3})--(0,0);
			\fill[red!10]  plot[domain=(1:2), samples=100] ({-3}, {\x})--(0,0);
			\fill[blue!10]  plot[domain=(-4:-3), samples=100] ({\x},{2*(\x)/3+4})--(-3,1);
			\draw (-4,0) node[below] {$-\frac{13}{12}$};
			\draw [dashed] (-4,4/3)--(-4,0);
			\draw (-3,0) node[below] {$-1$};
			\draw [dashed] (-3,2)--(-3,0);
			\draw (0,4/3) node[right] {$\frac{11}{18}$};
			\draw [dashed] (-4,4/3)--(0,4/3);
			\draw (0,2) node[right] {$\frac{2}{3}$};
			\draw [dashed] (-3,2)--(0,2);
			\draw (0,3) node[right] {$\frac{3}{4}$};
		\end{tikzpicture}
		\caption{Range of $s$ and $b$ when $s<-\frac{3}{4}$}
		\label{Fig, range of s and b}
	\end{figure}
	
	For the convenience of the readers, we draw a picture of the range of $s$ and $b$ when $s< -\frac34$, see Figure \ref{Fig, range of s and b}. This range is sharp due to Theorem \ref{Thm, sharp d2, neg}.
	
	\begin{theorem}\label{Thm, sharp d2, neg}
		Let $\a_1\a_2 < 0$ and $\b_1=\b_2=\b$. Assume $s$ and $b$ satisfy one of the following conditions.
		\begin{enumerate}[(1)]
			\item $s<-\frac{13}{12}$ and $b\in\m{R}$;
			\item $-\frac{13}{12}\leq s\leq -1$ and $b\notin [\frac{1}{4}-\frac{s}{3},\, \frac{4}{3}+\frac{2s}{3}]$;
			\item $-1< s<-\frac{3}{4}$ and $b\notin [\frac{1}{4}-\frac{s}{3},\, 1+\frac{s}{3}]$.
		\end{enumerate}
		Then there does not exist any constant $C=C(\a_1,\a_2,\b,s,b)$ such that (\ref{d2}) holds.
	\end{theorem}
	
	The results presented in Theorem \ref{Thm, d2, neg} and Theorem \ref{Thm, sharp d2, neg} together are surprising  in comparison to the previous results on the bilinear estimate. 
	\begin{itemize}
		\item First, in the case of the single KdV equation (\ref{KdV eq}), the critical index for the corresponding bilinear estimate (\ref{classical bilin}) is $-\frac{3}{4}$. However, when $\a_{1}\a_{2}<0$, the critical index of the bilinear estiamte (\ref{d2}) of type (D2) can be as low as $-\frac{13}{12}$.
		
		\item Secondly, for the previous biliear estimates, $b$ is usually required to be close to $\frac{1}{2}$ as $s$ approaches to the critical threshold. However, for the bilinear estimate (\ref{d2}) with  $\a_{1}\a_{2}<0$ and $-\frac{13}{12}\leq s<-\frac{3}{4}$, $b$ needs to be away from $\frac{1}{2}$. In particular, when $s=-\frac{13}{12}$, $b$ needs to be exactly $\frac{11}{18}$.
	\end{itemize}
	
	\begin{theorem}\label{Thm, bilin est, general}
		Let $\a_1,\a_2\in\m{R}\setminus\{0\}$ and denote $r=\frac{\a_2}{\a_1}$. Assume $r$, $s$ and the bilinear estimate type belong to any case in Table \ref{Table, bilin est}. Then there exist $b_0=b_0(s)>\frac12$ and $\eps=\eps(\a_1,\a_2)$ such that for any $\frac12<b\leq b_0$ and for any $|\b_2-\b_1|\leq \eps$, the bilinear estimate holds with some constant $C=C(\a_1,\a_2,s,b)$.
		
		\begin{table}[!ht]
			\renewcommand\arraystretch{1.8}
			\begin{center}
				\begin{tabular}{|c|c|c|c|c|c|} \hline
					Type   & $r<0$    &  $0<r<\frac{1}{4}$ & $r=\frac{1}{4}$ & $r>\frac{1}{4}$, $r\neq 1$ &  $r=1$  \\ \hline
					(D1): (\ref{d1}) & $s>-\frac{3}{4}$ & $s>-\frac{3}{4}$  & $s\geq \frac{3}{4}$  &  $s\geq 0$  &  $s>-\frac{3}{4}$  \\\hline
					(D2): (\ref{d2}) &   &  $s>-\frac{3}{4}$ &  $s\geq \frac{3}{4}$ &  $s\geq 0$ &  $s>-\frac{3}{4}$  \\\hline
					(ND1): (\ref{nd1}) &  $s>-\frac{3}{4}$  &  $s>-\frac{3}{4}$  &  $s\geq \frac{3}{4}$ & $s\geq 0$  &  $s>0$  \\\hline
					(ND2): (\ref{nd2}) & $s>-\frac{3}{4}$  &  $s>-\frac{3}{4}$  &  $s\geq \frac{3}{4}$  & $s\geq 0$  &  $s>0$ \\\hline
				\end{tabular}	 
			\end{center} 
			\caption{Bilinear Estimates}
			\label{Table, bilin est}
		\end{table}
	\end{theorem}
	
	The indexes in Table \ref{Table, bilin est} are also sharp.
	
	\begin{theorem}\label{Thm, sharp bilin est, general}
		Let $\a_1,\a_2\in\m{R}\setminus\{0\}$ and denote $r=\frac{\a_2}{\a_1}$. Let $\b_1=\b_2=\b$. Assume $r$, $s$ and the bilinear estimate type belong to any case in Table \ref{Table, sharp bilin est}. Then for any $b\in\m{R}$, there does not exist a constant $C=C(\a_1,\a_2,\b,s,b)$ such that the bilinear estimate holds.
		
		\begin{table}[!ht]
			\renewcommand\arraystretch{1.8}
			\begin{center}
				\begin{tabular}{|c|c|c|c|c|c|} \hline
					Type   & $r<0$    &  $0<r<\frac{1}{4}$ & $r=\frac{1}{4}$ & $r>\frac{1}{4}$, $r\neq 1$ &  $r=1$  \\ \hline
					(D1): (\ref{d1}) & $s<-\frac{3}{4}$ & $s<-\frac{3}{4}$  & $s< \frac{3}{4}$  &  $s< 0$  &  $s<-\frac{3}{4}$  \\\hline
					(D2): (\ref{d2}) &   &  $s<-\frac{3}{4}$ &  $s< \frac{3}{4}$ &  $s< 0$ &  $s<-\frac{3}{4}$  \\\hline
					(ND1): (\ref{nd1}) &  $s<-\frac{3}{4}$  &  $s<-\frac{3}{4}$  &  $s< \frac{3}{4}$ & $s< 0$  &  $s<0$  \\\hline
					(ND2): (\ref{nd2}) & $s<-\frac{3}{4}$  &  $s<-\frac{3}{4}$  &  $s<\frac{3}{4}$  & $s< 0$  &  $s<0$ \\\hline
				\end{tabular}	 
			\end{center} 
			\caption{Sharpness of Bilinear Estimates}
			\label{Table, sharp bilin est}
		\end{table}
	\end{theorem}
	
	There are several things worth mentioning about Theorem \ref{Thm, bilin est, general} and \ref{Thm, sharp bilin est, general}. First, when $r<0$, the critical index for Type (D1) is $-\frac34$ which is much larger than that for Type (D2), see Theorem \ref{Thm, d2, neg}. Secondly, when $r=\frac14$, the critical index is $\frac34$ which is much larger than other cases for $r\neq \frac14$. Thirdly, when $r=1$, the critical index is $-\frac34$ for the divergence forms but is $0$ for the non-divergence forms. 
	
	\begin{remark}\label{Re, pre bilin result}
		Some results in Theorem \ref{Thm, bilin est, general} and \ref{Thm, sharp bilin est, general} have already been known (or can be proven similarly) in the previous literatures. More specifically, in Table \ref{Table, bilin est} and \ref{Table, sharp bilin est}, when $r=1$, Type (D1) and (D2) have been established in \cite{KPV96}; when $r>\frac14$ but $r\neq 1$, Type (D1) and (D2) have been justified in \cite{Oh09a}, and Type (ND1) and (ND2) can be proven similarly. But note that the notations in \cite{Oh09a} are slightly different from here. Actually, the roles of $\a_1$ and $\a_2$ are interchanged there. In Table 8 with $r=-1$, Type (D1) has appeared in \cite{AC08}.
	\end{remark}
	
	The proofs of Theorem \ref{Thm, d2, neg}--\ref{Thm, sharp bilin est, general} are very technical and tedious, so we postpone them to Section \ref{Sec, proof for bilin est} and \ref{Sec, sharp of bilin est}.
	
	\section{Proofs of the main results on well-posedness} 
	\label{Sec, proof of wp}
	
	\subsection{Proof of Theorem \ref{Thm, lwp for ckdv}}
	\label{Subsec, proof for main thm on wp}
	The proofs for the  local well-posedness  results   in this paper will use the scaling argument as in \cite{KPV96}. This argument reduces the proofs to the case when the initial data and the elements $b_{ij}$ in the matrix $B$ are sufficiently small. Define the functions $u^{\lam}$ and $v^{\lam}$ for $\lam\geq 1$ as follows:
	\be\label{cov for div form}
	\left\{\begin{aligned}
		u^{\lam}(x,t)=\lam^{-2}u(\lam^{-1}x,\lam^{-3}t), \\
		v^{\lam}(x,t)=\lam^{-2}v(\lam^{-1}x,\lam^{-3}t), 
	\end{aligned}\quad x\in\m{R}, \,t\in\m{R}.\right.\ee
	Then (\ref{ckdv, coef form}) is equivalent to the system below.
	
	\be\label{ckdv scaling, coef form}
	\left\{\begin{array}{rcl}
		u^{\lam}_{t}+a_{1}u^{\lam}_{xxx}+b_{11}^{\lam}u^{\lam}_{x} &=& -b_{12}^{\lam}v^{\lam}_{x}+c_{11}u^{\lam}u^{\lam}_{x}+c_{12}v^{\lam}v^{\lam}_{x}+d_{11}u^{\lam}_{x}v^{\lam}+d_{12}u^{\lam}v^{\lam}_{x},\vspace{0.03in}\\
		v^{\lam}_{t}+a_{2}v^{\lam}_{xxx}+b_{22}^{\lam}v^{\lam}_{x} &=& -b_{21}^{\lam}u^{\lam}_{x}+c_{21}u^{\lam}u^{\lam}_{x}+c_{22}v^{\lam}v^{\lam}_{x}+d_{21}u^{\lam}_{x}v^{\lam}+d_{22}u^{\lam}v^{\lam}_{x}, \vspace{0.03in}\\
		(u^{\lam},v^{\lam})(x,0) &=& (u^{\lam}_{0},v^{\lam}_{0})(x),
	\end{array}\right.\ee
	where $b^{\lam}_{ij}=\lam^{-2}b_{ij}$ and 
	\[\left\{
	\begin{aligned}
		u^{\lam}_{0}(x) =\lam^{-2}u_{0}(\lam^{-1}x), \\
		v^{\lam}_{0}(x) =\lam^{-2}v_{0}(\lam^{-1}x),
	\end{aligned} 
	\qquad x\in\m{R}.
	\right.\]
	Since $\lam\geq 1$ and $s\geq-\frac{13}{12}$, then 
	\[\left\{\begin{array}{ll}
		\|u^{\lam}_{0}\|_{H^{s}(\m{R})}\leq \lam^{-\frac{5}{12}}\|u_{0}\|_{H^{s}(\m{R})}, \\
		\|v^{\lam}_{0}\|_{H^{s}(\m{R})}\leq \lam^{-\frac{5}{12}}\|v_{0}\|_{H^{s}(\m{R})}.
	\end{array}\right.\]
	Consequently, as $\lam\rightarrow\infty$, 
	\be\label{small data}
	\max_{1\leq i,j\leq 2} |b^{\lam}_{ij}|\rightarrow 0\quad\text{and}\quad 
	\max_{1\leq i,j\leq 2}(|c_{ij}|+|d_{ij}|)\big(\|u^{\lam}_{0}\|_{H^{s}(\m{R})}+\|v^{\lam}_{0}\|_{H^{s}(\m{R})}\big)\rightarrow 0.\ee
	So in order to prove the local well-posedness of (\ref{ckdv, coef form}), it suffices to justify the statement below.

	\begin{proposition}\label{Prop, lwp for s-ckdv}
		Let $a_{1},\,a_{2}\in\m{R}\setminus\{0\}$ and denote $r=\frac{a_{2}}{a_{1}}$. Assume $r$, $s$ and the coefficients $b_{ij}$, $c_{ij}$ and $d_{ij}$ belong to  any case in Table \ref{Table, lwp for ckdv} of Theorem \ref{Thm, lwp for ckdv}. Let $T>0$ be given. Then there exists a constant $\eps=\eps(a_{1},a_{2},s, T)$ such that  if
		\be\label{smallness of coef} 
		\max_{1\leq i,j\leq 2}|b_{ij}|\leq \eps\,\quad\text{and}\quad \max_{1\leq i,j\leq 2}(|c_{ij}|+|d_{ij}|)\big(\|u_{0}\|_{H^{s}(\m{R})}+\|v_{0}\|_{H^{s}(\m{R})}\big)\leq\eps,\ee
		then (\ref{ckdv, coef form})  admits a unique solution $(u, v) \in C\big([0,T]; \mcal{H}^s (\m{R})\big)$ satisfying the auxiliary  condition
		\[ \|u\|_{X^{a_1, b_{11}}_{s,b} ([0,T])}  + \|v\|_{X^{a_2,b_{22}}_{s,b}([0,T]) }< +\infty \]
		with some $\frac12 < b\leq \frac23$.  Moreover,  the corresponding solution map is real analytic in the corresponding spaces.
	\end{proposition}
	
	\begin{proof}[{\bf Proof of Proposition \ref{Prop, lwp for s-ckdv}}]
		We  only prove Case (1) with $r<0$, $(c_{ij})=0$, $d_{11}=d_{12}:=d_1$, $d_{21}=d_{22}:=d_2$ and $s\geq -\frac{13}{12}$. Other cases can be proved similarly by using appropriate bilinear estimates presented  in Theorem \ref{Thm, bilin est, general}. In addition, without loss of generality, we assume $T=1$. Hence, (\ref{ckdv, coef form}) with the assumption (\ref{smallness of coef}) becomes 
		\be\label{ckdv, coef form, special}
		\left\{\begin{array}{rcl}
			u_t+a_{1}u_{xxx}+b_{11}u_x &=& -b_{12}v_x+d_{1}(uv)_{x},\vspace{0.03in}\\
			v_t+a_{2}v_{xxx}+b_{22}v_x &=& -b_{21}u_x+d_{2}(uv)_{x}, \vspace{0.03in}\\
			\left. (u,v)\right |_{t=0} &=& (u_{0},v_{0})\in\mcal{H}^{s}(\m{R}),
		\end{array}\right.\ee
		where $a_1a_2<0$, $s\geq -\frac{13}{12}$ and 
		\be\label{smallness of coef, special} 
		\max_{1\leq i,j\leq 2}|b_{ij}|\leq \eps\,\quad\text{and}\quad (|d_1|+|d_2|)\big(\|u_{0}\|_{H^{s}(\m{R})}+\|v_{0}\|_{H^{s}(\m{R})}\big)\leq\eps,\ee
		for some $\eps=\eps(a_1,a_2,s)$ to be determined.
		
		By virtue of the semigroup operator $S_{i}=S^{a_{i},b_{ii}}$ for $i=1,2$, the Cauchy problem (\ref{ckdv, coef form, special}) for $t\in[0,1]$ can be converted into the integral form
		\be\label{integ form}
		\left \{ \begin{aligned} 
			u(t) &= \psi(t)\Big(S_{1}(t)u_{0}+\int_{0}^{t}S_{1}(t-t')F_{1}(u,v)(t')\,dt'\Big),\\ 
			v(t) &= \psi(t)\Big(S_{2}(t)v_{0}+\int_{0}^{t}S_{2}(t-t')F_{2}(u,v)(t')\,dt'\Big), \end{aligned}\right.\ee
		where $\psi(t)$ is the bump function defined at the beginning of Section \ref{Sec, Pre} and
		\be\label{nonlin term}
		\left\{\begin{aligned}
			F_{1}(u,v)=-b_{12}v_{x}+d_{1}(uv)_{x},\\
			F_{2}(u,v)=-b_{21}u_{x}+d_{2}(uv)_{x}.
		\end{aligned}\right.\ee 
		This suggests to consider the map $\Phi(u,v)\triangleq \big(\Phi_{1}(u,v), \Phi_{2}(u,v)\big)$, where 
		\be\label{contra map}
		\left\{\begin{aligned}
			\Phi_{1}(u,v) &=\psi(t)\Big(S_{1}(t)u_{0}+\int_{0}^{t}S_{1}(t-t')F_{1}(u,v)(t')\,dt'\Big),\\
			\Phi_{2}(u,v) &=\psi(t)\Big(S_{2}(t)v_{0}+\int_{0}^{t}S_{2}(t-t')F_{2}(u,v)(t')\,dt'\Big).
		\end{aligned}\right. \ee
		The goal is to show $\Phi $ is a contraction mapping in a ball in an appropriate Banach space, which will imply that the fixed point of $\Phi$ is the desired solution to the Cauchy problem (\ref{ckdv, coef form, special}) for $0\leq t\leq 1$.
		
		For  convenience,  let $Y^{i}_{s,b}=X_{s,b}^{a_{i},b_{ii}}$, $i=1,2$,  and $\mcal{Y}_{s,b} = Y^1_{s,b}\times Y^2_{s,b} $ equipped with the norm  
		\[   \|(u,v)\|_{\mcal{Y}_{s,b} }:=\|u\|_{Y^{1}_{s,b}}+\|v\|_{Y^{2}_{s,b}}.\] 
		Define $M_{1}=\max\limits_{1\leq i,j\leq 2}|b_{ij}|$ and $M_{2}=\max\limits_{1\leq i\leq 2}|d_i|$. Then  assumption (\ref{smallness of coef, special}) becomes
		\be\label{small M}
		M_{1}\leq \eps \quad\text{and}\quad M_{2}\big(\|u_{0}\|_{H^{s}(\m{R})}+\|v_{0}\|_{H^{s}(\m{R})}\big)\leq \eps.\ee
		Define 
		\be\label{soln space}
		\mcal{B}_{s,b,C}(u_0, v_0)=\big\{(u,v)\in \mcal{Y}_{s,b}: \|(u,v)\|_{\mcal{Y}_{s,b}} \leq C(\|u_{0}\|_{H^{s}}+\|v_{0}\|_{H^{s}})\big\}.\ee
		
		In the following, we will choose suitable $\eps$, $b$ and $C$ such that $\Phi$ is a contraction mapping on $\mcal{B}_{s,b,C}(u_0, v_0)$. We will first show that $\Phi$ maps  the closed ball $\mcal{B}_{s,b,C} (u_0, v_0)$ into itself. For any  $(u,v)\in \mcal{B}_{s,b,C} (u_0, v_0)$,  by Lemma \ref{Lemma, lin est for KdV}, for any $b>\frac12$, there exists a constant $C_{1}=C_{1}(b)$ such that 
		\be\label{est on Phi}\begin{split}
			\|\Phi_{1}(u,v)\|_{Y^{1}_{s,b}} & \leq C_{1}\|u_0\|_{H^{s}}+C_{1}\|F_{1}(u,v)\|_{Y^{1}_{s,b-1}},\\
			\|\Phi_{2}(u,v)\|_{Y^{2}_{s,b}} & \leq C_{1}\|v_0\|_{H^{s}}+C_{1}\|F_{2}(u,v)\|_{Y^{2}_{s,b-1}}.
		\end{split}\ee
		Since $F_{1}(u,v)=-b_{12}v_{x}+d_{1}(uv)_{x}$, we will estimate $\|b_{12}v_{x}\|_{Y^{1}_{s,b-1}}$ and $\|d_{1}(uv)_{x}\|_{Y^{1}_{s,b-1}}$ separately  in order to bound $\|F_{1}(u,v)\|_{Y^{1}_{s,b-1}}$. Since $a_1\neq a_2$, it follows from Lemma 2.2 that for any $b\in(\frac12,\frac23]$, there exist $\eps_{1}=\eps_{1}(a_1,a_2)$ and $C_2=C_2(a_1,a_2)$ such that for any $|b_{22}-b_{11}|\leq \eps_1$, 
		\[\|b_{12}v_{x}\|_{Y^{1}_{s,b-1}}\leq C_{2}|b_{12}|\|v\|_{Y^{2}_{s,b}}\leq C_{2}M_{1}\|v\|_{Y^{2}_{s,b}}.\]
		On the other hand, by Theorem \ref{Thm, d2, neg}, there exist $b^{*}=b^{*}(s)\in(\frac12,\frac23]$, $\eps_{2}=\eps_{2}(a_1,a_2)$ and $C_{3}=C_{3}(a_1,a_2,s,b^{*})$ such that for any $|b_{22}-b_{11}|\leq \eps_{2}$,  
		\[\|d_{1}(uv)_{x}\|_{Y^{1}_{s,b^{*}-1}}\leq C_{3}|d_{1}|\|u\|_{Y^{1}_{s,b^{*}}}\|v\|_{Y^{2}_{s,b^{*}}}\leq C_{3}M_2\|u\|_{Y^{1}_{s,b^{*}}}\|v\|_{Y^{2}_{s,b^{*}}}.\]
		Thus, for this particular $b^{*}$, taking $\eps_3=\min\{\eps_{1},\eps_2\}$ and $C_{4}=\max\{C_{1},C_{2},C_{3}\}$, then for any $|b_{11}|+|b_{22}|\leq \eps_3$, 
		\be\label{nonlin est 1}
		\|F_{1}(u,v)\|_{Y^{1}_{s,b^{*}-1}}\leq C_{4}\Big(M_1\|v\|_{Y^{2}_{s,b^{*}}}+M_2\|u\|_{Y^{1}_{s,b^{*}}}\|v\|_{Y^{2}_{s,b^{*}}}\Big).\ee
		Analogously, it also holds
		\be\label{nonlin est 2}
		\|F_{2}(u,v)\|_{Y^{2}_{s,b^{*}-1}}\leq C_{4}\Big(M_1\|u\|_{Y^{1}_{s,b^{*}}}+M_2\|u\|_{Y^{1}_{s,b^{*}}}\|v\|_{Y^{2}_{s,b^{*}}}\Big).\ee
		Adding (\ref{est on Phi}), (\ref{nonlin est 1}) and (\ref{nonlin est 2}) together yields that 
		\be\label{apriori bdd for Phi}\|\Phi(u,v)\|_{\mcal{Y}_{s,b^*}}\leq C_{5}\Big(\|u_0\|_{H^s}+\|v_0\|_{H^s}+M_{1}\|(u,v)\|_{\mcal{Y}_{s,b^*}}+M_2\|u\|_{Y^{1}_{s,b^*}}\|v\|_{Y^{2}_{s,b^*}}\Big),\ee
		where the constant $C_{5}$ only depends on $a_1$, $a_2$, $s$ and $b^*$. Actually, since $b^*$ is determined by $s$, $C_{5}$ only depends on $a_1$, $a_2$ and $s$. Denote $E_{0}=\|u_{0}\|_{H^{s}}+\|v_{0}\|_{H^{s}}$ and define
		\be\label{m-1} C^{*}=8C_{5}.\ee
		Then it follows from (\ref{soln space}) that $\|(u,v)\|_{\mcal{Y}_{s,b^*}}\leq C^{*}E_{0}$ for any $(u,v)\in \mcal{B}_{s,b^*,C^*} (u_0, v_0)$. Hence, it follows from (\ref{apriori bdd for Phi}) that
		\[\|\Phi(u,v)\|_{\mcal{Y}_{s,b^*}} \leq C_{5}E_{0}+C_{5}M_{1}C^{*}E_{0}+C_{5}M_{2}(C^{*})^{2}E_{0}^{2}.\]
		Since $C^{*}=8C_{5}$,
		\[\|\Phi(u,v)\|_{\mcal{Y}_{s,b^*}} \leq C_{5}E_{0}+8C_{5}^{2}M_{1}E_{0}+64C_{5}^{3}M_{2}E_{0}^{2}.\]
		Now choose 
		\be\label{choice of eps}\eps^{*}=\min\Big\{\frac{\eps_{3}}{2},\,\frac{1}{16C_{5}},\,\frac{1}{128C_{5}^{2}}\Big\}.\ee
		Then for any $(u,v)\in \mcal{B}_{s,b^*,C^*} (u_0, v_0)$, it follows from (\ref{small M}) and (\ref{choice of eps}) that
		\[\|\Phi(u,v)\|_{\mcal{Y}_{s,b^*}}\leq 2C_{5}E_{0}= \frac{C^{*}E_{0}}{4},\]
		which implies $\Phi(u,v)\in \mcal{B}_{s,b^*,C^*} (u_0, v_0)$. 
		
		Next for any $(u_j, v_j)\in \mcal{B}_{s,b^*,C^*} (u_0, v_0)$, $j=1,2$, the same argument yields 
		\[ \| \Phi (u_1, v_1 )- \Phi (u_2 , v_2)\|_{\mcal{Y}_{s,b^*}} \leq \frac12 \| (u_1, v_1)-(u_2, v_2)\|_{\mcal{Y}_{s,b^*}}.\]
		We  have thus shown  that $\Phi$ is a contraction on $\mcal{B}_{s,b^*,C^*} (u_0, v_0)$.  Its fixed point is the desired solution of   the system (\ref{ckdv, coef form, special}) on a time interval of size 1. 
	\end{proof}
	
	\subsection{Proofs of Theorem \ref{Thm, weak nonlin ckdv} -- Theorem \ref{Thm, G-G WP}}
	\label{Subsec, proof for other thm on wp}
	Theorem \ref{Thm, weak nonlin ckdv} follows directly from Case (1) in Theorem \ref{Thm, lwp for ckdv}. 
	
	For the Majda-Biello system (\ref{M-B system}), it is a special case of (\ref{ckdv, coef form}) with the coefficients
	\be\label{M-B system, coef}\begin{split}
		&a_{1}=1;\quad (b_{ij})=0;\quad c_{11}=c_{21}=c_{22}=0,\, c_{12}=-1;\\
		& d_{11}=d_{12}=0,\, d_{21}=d_{22}=-1. \end{split}\ee
	So the LWP results in Theorem \ref{Thm, M-B WP} follow directly from Theorem \ref{Thm, lwp for ckdv}. Then according to these LWP results and the conserved energies (\ref{M-B energy}), the GWP results in Theorem \ref{Thm, M-B WP} are established (except when $a_2=1$ for which case the GWP was proved for any $s>-\frac34$ by Oh\cite{Oh09b} via the I-method). 
	\be\label{M-B energy}\begin{split}
		E_{1}(u,v) &= \int u^{2}+v^{2}\,dx,\\
		E_{2}(u,v) &= \int u_{x}^{2}+a_{2}v_{x}^{2}-uv^{2}\,dx.\end{split}\ee

	For the Hirota-Satsuma system (\ref{H-S system}), it is a special case of (\ref{ckdv, coef form}) with the coefficients
	\be\label{H-S system, coef}\begin{split} &a_{2}=1;\quad (b_{ij})=0; \quad c_{11}=-6a_{1},\,c_{21}=c_{22}=0;\\ & d_{11}=d_{12}=d_{21}=0,\,d_{22}=-3. \end{split}\ee
	So the LWP results in Theorem \ref{Thm, H-S WP} follow directly from Theorem \ref{Thm, lwp for ckdv}. Then according to these LWP results and the conserved energies (\ref{H-S energy}), the GWP results in Theorem \ref{Thm, H-S WP} are established.
	\be\label{H-S energy}\begin{split}
		E_{1}(u,v) &= \int u^{2}+\frac{c_{12}}{3}v^{2}\,dx,\\
		E_{2}(u,v) &= \int (1-a_1)u_{x}^{2}+c_{12}v_{x}^{2}-2(1-a_{1})u^{3}-c_{12}uv^{2}\,dx.
	\end{split}\ee

	For the Gear-Grimshaw system (\ref{G-G system}), we first write it into the vector form:
	\be\label{G-G system, vector}
	\left\{\begin{array}{rcl}
		\bp u_t\\v_t  \ep + A_{1}\bp u_{xxx}\\v _{xxx}\ep + A_{2}\bp u_x\\v_x \ep &=& A_{3}\bp uu_x\\vv_x \ep + A_{4}\bp u_{x}v\\uv_{x}\ep,  \vspace{0.1in}\\
		\left. (u,v)\right |_{t=0} &=& (u_{0},v_{0}),
	\end{array}\right.\ee
	where 
	\[A_{1}=\bp 1& \sigma_{3}\\ \frac{\rho_{2}\sigma_{3}}{\rho_{1}} & \frac{1}{\rho_{1}} \ep,
	\quad A_{2}=\bp 0 &0\\0&\frac{\sigma_4}{\rho_1} \ep,\quad A_{3}=\bp -1 & \sigma_1 \\ \frac{\rho_2\sigma_2}{\rho_1} & -\frac{1}{\rho_1}\ep,\quad A_{4}=\bp \sigma_2 & \sigma_2\\ \frac{\rho_2\sigma_1}{\rho_1} & \frac{\rho_2\sigma_1}{\rho_1}\ep.\]
	When $\rho_2\sigma_3^2\neq 1$, $A_{1}$ has two nonzero eigenvalues $\lam_{1}$ and $\lam_{2}$:
	\be\label{eigen of A_1, G-G}
	\lam_{1}=\frac{\rho_1+1}{2\rho_1}+\frac{\sqrt{(\rho_1-1)^2+4\rho_1\rho_2\sigma_3^2}}{2\rho_1},\qquad \lam_{2}=\frac{\rho_1+1}{2\rho_1}-\frac{\sqrt{(\rho_1-1)^2+4\rho_1\rho_2\sigma_3^2}}{2\rho_1}.\ee
	So there exists an invertible real-valued matrix $M$ such that $A_{1}=M\bp\lam_{1}& \\ & \lam_{2}\ep M^{-1}$. By regarding $M^{-1}\bp u\\v\ep$ as the new unknown functions (still denoted by $u$ and $v$), (\ref{G-G system, vector}) can be rewritten as
	\be\label{G-G system, vector, diag}
	\left\{\begin{array}{rcl}
		\bp u_t\\v_t  \ep + \bp\lam_{1}&\\ &\lam_{2}\ep\bp u_{xxx}\\v _{xxx}\ep + B\bp u_x\\v_x \ep &=& C\bp uu_x\\vv_x \ep + D\bp u_{x}v\\uv_{x}\ep, \vspace{0.1in}\\
		\left. (u,v)\right |_{t=0} &=& (u_{0},v_{0}),
	\end{array}\right.\ee
	where $d_{11}=d_{12}$, $d_{21}=d_{22}$. In addition, $B=\bp 0&0\\ 0&0\ep$ if $\sigma_4=0$. Define 
	\be\label{r for G-G}
	r=\frac{\lam_2}{\lam_1}.\ee
	Then it follows from (\ref{eigen of A_1, G-G}) that $r<1$. Moreover, since both $\rho_1$ and $\rho_2$ are positive numbers, we have $r<0 \Longleftrightarrow \rho_2\sigma_3^2>1$. Moreover, $r=\frac14$ if and only if (\ref{G-G, r=1/4}) holds, that is
	\[\rho_2\sigma_3^2\leq \frac{9}{25} \quad\text{and}\quad \rho_{1}^2+\frac{25\rho_2\sigma_3^2-17}{4}\rho_1+1=0.\]
	Based on the above observations, the LWP results in Theorem \ref{Thm, G-G WP} follow from Theorem \ref{Thm, lwp for ckdv}. Then according to these LWP results and the conserved energies (\ref{G-G energy}), the GWP results in Theorem \ref{Thm, G-G WP} are established.
	\be\label{G-G energy}\begin{split}
		E_{1}(u,v) &= \int \rho_{2}u^{2}+\rho_{1}v^{2}\,dx,\\
		E_{2}(u,v) &= \int \rho_{2}u_{x}^{2}+v_{x}^{2}+2\rho_{2}\sigma_{3}u_{x}v_{x}-\frac{\rho_2}{3}u^{3}+\rho_2\sigma_{2}u^{2}v+\rho_{2}\sigma_{1}uv^{2}-\frac{1}{3}v^{3}-\sigma_{4}v^{2}.
	\end{split}\ee

	\section{Proofs of the bilinear estimates}
	\label{Sec, proof for bilin est}
	\quad 
	The goal of this section is to prove Theorem \ref{Thm, d2, neg} and Theorem \ref{Thm, bilin est, general}.

	\subsection{Idea of the proofs}
	\label{Subsec, idea of proof}
	
	The main idea of treating the bilinear estimates of different types are similar, and is thus explained only  for the following divergence form with $\b_i=0\,(i=1,2,3)$. 
	\be\label{bilin est, div, general}
	\|\p_{x}(w_{1}w_{2})\|_{X^{\a_{3}}_{s,b-1}}\ls \| w_{1}\|_{X^{\a_{1}}_{s,b}}\| w_{2}\|_{X^{\a_{2}}_{s,b}},\quad\forall\,w_{1},w_2.\ee
	By duality and Plancherel identity, (\ref{bilin est, div, general}) is equivalent to (see e.g. \cite{Tao01})
	\be\label{weighted l2 form, div, general}
	\Bigg|\int\limits_{\sum\limits_{i=1}^{3}\xi_{i}=0}\int\limits_{\sum\limits_{i=1}^{3}\tau_{i}=0}\frac{\xi_{3}\la\xi_{3}\ra^{s}\prod\limits_{i=1}^{3}f_{i}(\xi_{i},\tau_{i})}{\la\xi_{1}\ra^{s}\la\xi_{2}\ra^{s}\la L_{1}\ra^{b}\la L_{2}\ra^{b}\la L_{3}\ra^{1-b}}\Bigg| \leq C\,\prod_{i=1}^{3}\|f_{i}\|_{L^{2}_{\xi\tau}}, \quad\forall\,\{f_{i}\}_{1\leq i\leq 3},\ee
	where 
	\[L_{i}=\tau_{i}-\phi^{\a_{i}}(\xi_{i})=\tau_{i}-\a_{i}\xi_{i}^{3},\quad 1\leq i\leq 3.\]
	In (\ref{weighted l2 form, div, general}), the loss of the spatial derivative in the bilinear estimate (\ref{bilin est, div, general}) is reflected in the term $\frac{\xi_{3}\la\xi_{3}\ra^{s}}{\la\xi_{1}\ra^{s}\la\xi_{2}\ra^{s}}$ and the gain of the time derivative is reflected in  the term $\la L_{1}\ra^{b}\la L_{2}\ra^{b}\la L_{3}\ra^{1-b}$. Then how to compensate the loss of the spatial derivative from the gain of the time derivative is the key issue. Denote 
	\[K_{1}=\frac{\xi_{3}\la\xi_{3}\ra^{s}}{\la\xi_{1}\ra^{s}\la\xi_{2}\ra^{s}}\quad\text{and}\quad K_{2}=\la L_{1}\ra^{b}\la L_{2}\ra^{b}\la L_{3}\ra^{1-b}.\]
	Then the main idea is to control $K_{1}$ by taking advantage of $K_{2}$. Since $\sum\limits_{i=1}^{3}\xi_{i}=0$, then $\la\xi_{3}\ra\leq \la\xi_{1}\ra\la\xi_{2}\ra$. As a result, $K_{1}$ is a decreasing function in $s$, which means the smaller $s$ is, the more likely the bilinear estimate will fail. So the question is  how to find the smallest $s$ such that the bilinear estimate holds. Noticing that $L_{i}$ contains the time variable $\tau_{i}$, so a single $L_{i}$ can barely have any contributions. Since $\sum\limits_{i=1}^{3}\tau_{i}=0$, then $\sum\limits_{i=1}^{3}L_{i}=-\sum\limits_{i=1}^{3}\a_{i}\xi_{i}^{3}$ is a function only in $\xi_{i}$, $1\leq i\leq 3$. Define 
	\[H(\xi_1,\xi_2,\xi_3):=\sum_{i=1}^{3}\a_{i}\xi_{i}^{3}.\]
	Then it is obvious that $K_{2}\gs |H|^{\min\{b,\,1-b\}}$, which may be used to control $K_{1}$. Thus, $H(\xi_{1},\xi_{2},\xi_{3})$ plays a fundamental role. In addition, $H$ measures to what extent the spatial frequencies $\xi_{1}$, $\xi_{2}$ and $\xi_{3}$ can resonate with each other. Because of this, $H$ is called the {\it resonance function} (see Page 856 in \cite{Tao01}). Unfortunately, $|H|$ is not always large, the situation may become complicated near the region where $H$ vanishes. We shall call the zero set of $H$ to be the {\it resonance set}. Usually, the worst situation occurs near the resonance set and this trouble is  called {\it resonant interactions} (see Page 856 in \cite{Tao01}). 
	
	In the following, we will investigate the resonance function and the resonance set in three typical situations (again $\{\b_i\}_{i=1}^{3}$ are assumed to be zero for simplicity).
	\begin{itemize}
		\item In the classical case when $\a_1=\a_2=\a_3$, the resonance function $H_{0}$ is in a very simple form:
		\[H_0(\xi_1,\xi_2,\xi_3)=3\a_1\xi_1\xi_2\xi_3.\]
		The resonance set consists of three hyperplanes: $\{\xi_i=0\}$, $i=1,2,3$. 
		
		\item For the bilinear estimate of Type (D1), the resonance function $H_{1}$ is 
		\[H_{1}(\xi_1,\xi_2,\xi_3)=\a_1\xi_1^3+\a_1\xi_2^3+\a_2\xi_3^3.\]
		By writing $\xi_2=-(\xi_1+\xi_3)$, 
		\[H_{1}(\xi_1,\xi_2,\xi_3)=\xi_{3}\Big[(\a_2-\a_1)\xi_3^2-3\a_1\xi_1\xi_3-3\a_1\xi_1^2\Big].\]
		So $\{\xi_3=0\}$ belongs to the resonance set. If $\xi_3\neq 0$, then $H_1$ can be rewritten as 
		\[H_{1}(\xi_1,\xi_2,\xi_3)=-3\a_1\xi_3^3h_{r}\Big(\frac{\xi_1}{\xi_3}\Big),\]
		where $r=\frac{\a_2}{\a_1}$ and 
		\[h_{r}(x):=x^2+x+\frac{1-r}{3}.\]
		So the resonance set is determined by the roots of $h_{r}$.
		
		\item Similarly, for the bilinear estimate of Type (D2), the resonance function $H_{2}$ is 
		\[H_{2}(\xi_1,\xi_2,\xi_3)=\a_1\xi_1^3+\a_2\xi_2^3+\a_1\xi_3^3.\]
		By writing $\xi_3=-(\xi_1+\xi_2)$, 
		\[H_{2}(\xi_1,\xi_2,\xi_3)=\xi_{2}\Big[(\a_2-\a_1)\xi_2^2-3\a_1\xi_1\xi_2-3\a_1\xi_1^2\Big].\]
		So $\{\xi_2=0\}$ belongs to the resonance set. If $\xi_2\neq 0$, then $H_2$ can be rewritten as 
		\[H_{2}(\xi_1,\xi_2,\xi_3)=-3\a_1\xi_2^3h_{r}\Big(\frac{\xi_1}{\xi_2}\Big).\]
		Again the resonance set is determined by the roots of $h_{r}$.
	\end{itemize}
	Due to the above observations, the function $h_{r}$ is crucial in determining the resonance set. The roots of $h_r$ have three possibilities.
	\begin{itemize}
		\item[(1)] If $r<\frac14$, then $h_{r}$ does not have any real roots.
		\item[(2)] If $r=\frac14$, then $h_{r}$ has one real root $-\frac12$ of multiplicity 2.
		\item[(3)] If $r>\frac14$, then $h_{r}$ has two distinct real roots.
	\end{itemize}
	As we have seen, the structure of $H_1$ is analogous to that of $H_{2}$. In addition, $H_0$ is just a special case of $H_2$ when $r=1$. So in the following, we will just focus on $H_{2}$ to discuss the effect of the resonance set on the threshold of $s$. 
	
	\begin{enumerate}[(i)]
		\item $r=1$. This agrees with the classical case and $h_{r}$ has two roots $x_{1r}=-1$ and $x_{2r}=0$. As we have seen that the resonance set of $H_0$ consists of three hyperplanes $\{\xi_{i}=0\}$, $1\leq i\leq 3$.  When $s<0$, by writing $\rho=-s$, then $\rho>0$ and 
		\[K_{1}\sim \la\xi_{1}\ra^{\rho}\la\xi_{2}\ra^{\rho}|\xi_{3}|^{1-\rho}.\]
		As a result, $K_{1}$ is also small near the resonance set $\{\xi_{i}=0\}$, $1\leq i\leq 3$, which means the resonant interactions do not cause too much trouble. This is why the sharp index for the bilinear estimate of the divergence form can be as low as $-\frac{3}{4}$ as shown in Lemma \ref{kpv-lemma}.

		\item $r<\frac{1}{4}$. In this case, there exists a positive constant $\delta_{r}$ such that $h_{r}(x)\geq \delta_{r}$ for any $x\in\m{R}$. Consequently, the resonance set is only a single hyperplane $\{\xi_{2}=0\}$. Moreover, $|\xi_{2}|\ll 1$ and $|\xi_{1}|\sim |\xi_{3}|$ near this hyperplane. As a result, $|K_{1}|\sim |\xi_{3}|$ does not depend on $s$ at all, which means the resonant interactions have no effect on $s$ in this case. So there is hope to obtain an even smaller threshold for $s$. Actually, for Type (D2) with $r<0$, $s$ can be as small as $-\frac{13}{12}$. 
		
		\item $r>\frac{1}{4}$ and $r\neq 1$.
		In this case, $h_{r}$ has two distinct nonzero roots $x_{1r}$ and $x_{2r}$. Therefore, 
		\[H_{2}(\xi_{1},\xi_{2},\xi_{3})=-3\a_1\xi_{2}^{3}\Big(\frac{\xi_1}{\xi_2}-x_{1r}\Big)\Big(\frac{\xi_1}{\xi_2}-x_{2r}\Big).\]
		The resonance set consists of three different hyperplanes: $\{\xi_{2}=0\}$, $\{\xi_{1}=x_{1r}\xi_{2}\}$ and $\{\xi_{1}=x_{2r}\xi_{2}\}$. If $s<0$, then near the hyperplane $\{\xi_{1}=x_{1r}\xi_{2}\}$ or $\{\xi_{1}=x_{2r}\xi_{2}\}$ with large $\xi_{2}$, the resonance function $H_{2}$ is small while $K_{1}$ is large. Thus, the bilinear estimate is likely to fail. Actually, the threshold for $s$ in this case is $s\geq 0$. This has already been pointed out by Oh\cite{Oh09a}.
		
		\item $r=\frac{1}{4}$. In this case, 
		\[H_{2}(\xi_{1},\xi_{2},\xi_{3})=-3\a_1\xi_{2}^{3}\Big(\frac{\xi_1}{\xi_2}+\frac12\Big)^{2}.\]
		The resonance set consists of two hyperplanes $\{\xi_{2}=0\}$ and $\{\xi_{1}=-\frac12\xi_{2}\}$. But the resonance interaction is significant near the hyperplane $\{\xi_{1}=-\frac12\xi_{2}\}$ due to the square power. Consequently, the situation is expected to be worse. Actually, the bilinear estimate is valid only for $s\geq \frac{3}{4}$. 
	\end{enumerate}

	In addition to the resonant interactions, there is another trouble coming from {\it coherent interactions}  (see \cite{Tao01}) when one has $\nabla \phi^{\a_{1}}(\xi_{1})=\nabla \phi^{\a_{2}}(\xi_{2})$, that is $\a_{1}\xi_{1}^{2}=\a_{2}\xi_{2}^{2}$. Geometrically, coherent interactions occur when the surfaces $\tau_{1}=\phi^{\a_{1}}(\xi_{1})$ and $\tau_{2}=\phi^{\a_{2}}(\xi_{2})$ fail to be transverse. For example, when $r<\frac14$, as we just discussed above, the resonance set of $H_{2}$ is a sinlge hyperplane: $\{\xi_{2}=0\}$ no matter $r<0$ or $0<r<\frac{1}{4}$. However, the critical indexes for $s$ are different in these two cases. 
	\begin{itemize}
		\item If $r<0$, then $\a_{1}\xi_{1}^{2}$ will not match $\a_{2}\xi_{2}^{2}$ regardless of the values of $\xi_{1}$ and $\xi_{2}$. So the coherent interactions do not occur in this case and the sharp index for $s$ is $-\frac{13}{12}$.
		
		\item If $r>0$, then $\a_{1}\xi_{1}^{2}=\a_{2}\xi_{2}^{2}$ when $\xi_{1}=\pm \sqrt{r}\xi_{2}$. So the coherent interactions occur along the hyperplanes $\{\xi_{1}=\pm \sqrt{r}\xi_{2}\}$. It turns out that the critical index for this case is $-\frac{3}{4}$.
	\end{itemize}
	
	The above arguments revealed the difficulties for the bilinear estimate of the divergence form. These difficulties play the similar role in the nondivergence case. But the nondivergence form can bring additional trouble. Let us compare (D2) and (ND1) with $r=1$ and $\b_{i}=0\,(i=1,2)$. In this case, the resonance functions for (D2) and (ND1) are the same, both of them are equal to $H_0$. However, the terms $K_{1}$ and $\widetilde{K}_{1}$ coming from the loss of the spatial derivative for (D1) and (ND1) are different. More precisely, 
	\[K_{1}=\frac{\xi_{3}\la\xi_{3}\ra^{s}}{\la\xi_{1}\ra^{s}\la\xi_{2}\ra^{s}}\quad\text{and}\quad \widetilde{K}_{1}=\frac{\xi_{1}\la\xi_{3}\ra^{s}}{\la\xi_{1}\ra^{s}\la\xi_{2}\ra^{s}}.\]
	Consider $s=-\frac{3}{4}+$, then 
	\[|K_{1}|\sim \frac{|\xi_3|\la\xi_{1}\ra^{\frac34 -}\la\xi_{2}\ra^{\frac34 -}}{\la\xi_{3}\ra^{\frac34 -}}\quad\text{and}\quad |\widetilde{K}_{1}|\sim\frac{|\xi_1|\la\xi_{1}\ra^{\frac34 -}\la\xi_{2}\ra^{\frac34 -}}{\la\xi_{3}\ra^{\frac34 -}}.\]
	Previously, the worst region for (D2) is when $|\xi_1|\sim |\xi_2|\gg |\xi_3|$ and this forces $s$ to be greater than $-3/4$. But in this region, it is easily seen that $\widetilde{K}_{1}$ is even much larger than $K_{1}$. So there is no hope to control $\widetilde{K}_{1}$ as well when $s$ is near $-3/4$. Actually, it will be shown that the critical index for (ND1) is 0.
	
	In summary, there are three main troubles in establishing the bilinear estimates (\ref{d1})-(\ref{nd2}).
	\begin{itemize}
		\item[(T1)]: {\it resonant interactions};
		\item[(T2)]: {\it coherent interaction};
		\item[(T3)]: the nondivergence form in the region $|\xi_{1}|\sim |\xi_{2}|\gg |\xi_{3}|$.
	\end{itemize}
	Generally speaking, (T1) is the most significant trouble and (T2) and (T3) are of the same level of influence. In most cases, these troubles do not occur at the same place, then the strategy is simply to divide the region suitably and  deal with one trouble in each region. However, if more than one trouble occur at the same place, then the situation is expected to be worse.
	In the following, we provide  Table \ref{Table, bilin est, Trouble and CI} to present the main trouble and the critical indexes for $s$ in each case for the bilinear estimates (\ref{d1})-(\ref{nd2}).  The sign  ``+"  indicates the situation when two troubles occur at the same place.
	\begin{table}[!ht]
		\renewcommand\arraystretch{1.6}
		\begin{center}
			\begin{tabular}{|c|c|c|c|c|c|} \hline
				& $r<0$    &  $0<r<\frac{1}{4}$ & $r=\frac{1}{4}$ & $r>\frac{1}{4}$, $r\neq 1$ &  $r=1$  \\ \hline
				(D1): (\ref{d1}) & \begin{tabular}{c} (T2) \\ $-\frac{3}{4}$ \end{tabular} & \begin{tabular}{c} (T2) \\ $-\frac{3}{4}$ \end{tabular} & \begin{tabular}{c} (T1)+(T2) \\ $\frac{3}{4}$ \end{tabular} & \begin{tabular}{c} (T1) \\ $ 0$ \end{tabular} & \begin{tabular}{c} (T2) \\ $-\frac{3}{4}$ \end{tabular} \\\hline
				(D2): (\ref{d2}) & \begin{tabular}{c} None \\ $ -\frac{13}{12}$ \end{tabular} & \begin{tabular}{c} (T2) \\ $-\frac{3}{4}$ \end{tabular} & \begin{tabular}{c} (T1)+(T2) \\ $ \frac{3}{4}$ \end{tabular} & \begin{tabular}{c} (T1) \\ $ 0$ \end{tabular} & \begin{tabular}{c} (T2) \\ $-\frac{3}{4}$ \end{tabular} \\\hline
				(ND1): (\ref{nd1}) & \begin{tabular}{c} (T3) \\ $-\frac{3}{4}$ \end{tabular} & \begin{tabular}{c} (T2) or (T3) \\ $-\frac{3}{4}$ \end{tabular} & \begin{tabular}{c} (T1)+(T2) \\ $ \frac{3}{4}$ \end{tabular} & \begin{tabular}{c} (T1) \\ $ 0$ \end{tabular} & \begin{tabular}{c} (T2)+(T3) \\ $0$ \end{tabular} \\\hline
				(ND2): (\ref{nd2}) & \begin{tabular}{c} (T3) \\ $-\frac{3}{4}$ \end{tabular} & \begin{tabular}{c} (T2) or (T3) \\ $-\frac{3}{4}$ \end{tabular} & \begin{tabular}{c} (T1)+(T2) \\ $ \frac{3}{4}$ \end{tabular}  & \begin{tabular}{c} (T1) \\ $ 0$ \end{tabular} & \begin{tabular}{c} (T2)+(T3) \\ $0$ \end{tabular}\\\hline
			\end{tabular}	 
		\end{center} 
		\caption{Troubles and Critical Indexes ($r=\frac{\a_2}{\a_1}$)}
		\label{Table, bilin est, Trouble and CI}
	\end{table}

	\subsection{Auxiliary lemmas}
	\label{Subsec, Aux lemma}
	
	\begin{lemma}\label{Lemma, int in tau}
		Let $\rho_{1}>1$ and $0\leq\rho_{2}\leq\rho_{1}$ be given. There exists  a constant $C=C(\rho_{1},\rho_{2})$ such that for any $\a,\b\in\m{R}$,
		\be\label{int in tau}
		\int_{-\infty}^{\infty}\frac{dx}{\la x-\a \ra^{\rho_{1}} \la -x-\b \ra^{\rho_{2}}}\leq \frac{C}{\la\a+\b\ra^{\rho_{2}}}.\ee
	\end{lemma}
	\noindent
	The proof for this lemma is standard and therefore omitted, we just want to remark that $\la\a+\b\ra=\la(x-\a)+(-x-\b)\ra$, this observation will be used in the estimate (\ref{use of lemma: int in tau}).

	\begin{lemma}\label{Lemma, bdd int}
		If $\rho>\frac{1}{2}$, then there exists $C=C(\rho)$ such that for any $\sigma_{i}\in\m{R},\,0\leq i\leq 2$, with $\sigma_{2}\neq 0$,
		\be\label{bdd int for quad}
		\int_{-\infty}^{\infty}\frac{dx}{\la \sigma_{2}x^{2}+\sigma_{1}x+\sigma_{0}\ra^{\rho}}\leq \frac{C}{|\sigma_{2}|^{1/2}}.\ee
		Similarly, if $\rho>\frac{1}{3}$, then there exists $C=C(\rho)$ such that for any $\sigma_{i}\in\m{R},\,0\leq i\leq 3$, with $\sigma_{3}\neq 0$,
		\be\label{bdd int for cubic}
		\int_{-\infty}^{\infty}\frac{dx}{\la \sigma_{3}x^{3}+\sigma_{2}x^{2}+\sigma_{1}x+\sigma_{0}\ra^{\rho}}\leq \frac{C}{|\sigma_{3}|^{1/3}}.\ee
	\end{lemma}
	
	\begin{proof}
		We refer the reader to the proof of Lemma 2.5 in \cite{BOP97} where (\ref{bdd int for cubic}) was  proved. The similar argument can also be applied to  obtain (\ref{bdd int for quad}).
	\end{proof}
	
	If the power $\rho$ in Lemma \ref{Lemma, bdd int} is greater than 1, then stronger conclusions hold.
	
	\begin{lemma}\label{Lemma, int for quad}
		Let $\rho>1$ be given. There exists  a constant $C=C(\rho)$ such that for any $\sigma_{i}\in\m{R},\,0\leq i\leq 2$, with $\sigma_{2}\neq 0$, 
		\be\label{int for quad}
		\int_{-\infty}^{\infty}\frac{dx}{\la \sigma_{2}x^{2}+\sigma_{1}x+\sigma_{0}\ra^{\rho}}\leq C\,|\sigma_{2}|^{-\frac{1}{2}}\Big\la \sigma_{0}-\frac{\sigma_{1}^{2}}{4\sigma_{2}}\Big\ra^{-\frac{1}{2}}.\ee
	\end{lemma}
	\begin{proof}
		It suffices to consider the case when $\sigma_{2}>0$. By rewriting 
		\[\sigma_{2}x^{2}+\sigma_{1}x+\sigma_{0}=\sigma_{2}\Big(x+\frac{\sigma_{1}}{2\sigma_{2}}\Big)^{2}+\sigma_{0}-\frac{\sigma_{1}^{2}}{4\sigma_{2}}\]
		and doing a change of variable $y=\sqrt{\sigma_{2}}\big(x+\frac{\sigma_{1}}{2\sigma_{2}}\big)$, it reduces to show for any $\a\in\m{R}$, 
		\[\int_{-\infty}^{\infty}\frac{dy}{\la y^{2}+\a\ra^{\rho}}\leq C\la\a\ra^{-\frac{1}{2}},\]
		for which, the verification is straightforward and left to the readers.
	\end{proof}
	
	\begin{lemma}\label{Lemma, int for cubic}
		Let $\rho>1$ be given. There exists a constant $C=C(\rho)$ such that  for any $\sigma_{i}\in\m{R},\,0\leq i\leq 2$,
		\be\label{int for cubic}
		\int_{-\infty}^{\infty}\frac{dx}{\la x^3+\sigma_{2}x^{2}+\sigma_{1}x+\sigma_{0}\ra^{\rho}}\leq C\big\la 3\sigma_{1}-\sigma_{2}^{2}\big\ra^{-\frac14}.\ee
	\end{lemma}
	
	\begin{proof}
		By the change of variable $y=x+\frac{\sigma_2}{3}$, 
		\[\int_{-\infty}^{\infty}\frac{dx}{\la x^3+\sigma_{2}x^{2}+\sigma_{1}x+\sigma_{0}\ra^{\rho}}=\int_{-\infty}^{\infty}\frac{dy}{\la y^{3}+b_{1}y+b_{0}\ra^{\rho}},\]
		where 
		\[b_{1}=\sigma_1-\frac{1}{3}\sigma_{2}^2,\quad b_{0}=\frac{2}{27}\sigma_{2}^{3}-\frac{1}{3}\sigma_1\sigma_2+\sigma_0.\]
		Thus, (\ref{int for cubic}) reduces to justify
		\be\label{int in cubic, simp}
		\int_{-\infty}^{\infty}\frac{dy}{\la y^{3}+b_{1}y+b_{0}\ra^{\rho}}\leq C\la b_1\ra^{-\frac14}\ee
		for some constant $C$ which only depends on $\rho$. If $|b_1|\ls 1$, then (\ref{int in cubic, simp}) follows from (\ref{bdd int for cubic}) in Lemma \ref{Lemma, bdd int}. If $|b_1|\gg 1$, we define $g(y)=y^{3}+b_1y+b_0$ and find $g'(y)=3y^2+b_1$.
		If $|g'(y)|\geq |b_1|^{1/4}$, then 
		\[\int\frac{dy}{\la g(y)\ra^{\rho}}\leq \int\frac{1}{|b_1|^{1/4}}\frac{|g'(y)|}{\la g(y)\ra^{\rho}}\,dy\leq C\la b_1\ra^{-\frac14}.\]
		If $|g'(y)|\leq |b_1|^{1/4}$, then the measure of the set of these $y$ values is at most $O\big(|b_{1}|^{-\frac14}\big)$, so the integral of $\la g(y)\ra^{-\rho}$ on this set is also bounded by $C\la b_1\ra^{-1/4}$.
	\end{proof}
	

	For the proof of the bilinear estimate, it is usually beneficial to transfer it to an estimate of some weighted convolution of $L^{2}$ functions as pointed out in \cite{Tao01, CKSTT03}. The next lemma is one of such an example for the general bilinear estimate whose proof is standard by using duality and Plancherel theorem. For the convenience of notation, 
	we denote $\vec{\xi}=(\xi_{1},\xi_{2},\xi_{3})$ and $\vec{\tau}=(\tau_{1},\tau_{2},\tau_{3})$ to be the vectors in $\m{R}^{3}$ and define
	\be\label{int domain}
	A:=\Big\{(\vec{\xi},\vec{\tau})\in\m{R}^{6}:\sum_{i=1}^{3}\xi_{i}=\sum_{i=1}^{3}\tau_{i}=0\Big\}.\ee 
	
	\begin{lemma}\label{Lemma, bilin to weighted l2}
		Given $s$, $b$ and $(\a_{i},\b_{i}) 1\leq i\leq 3$,  the bilinear estimate
		\[\|\p_{x}(w_{1}w_{2})\|_{X^{\a_{3},\b_{3}}_{s,b-1}}\leq C\,\|w_{1}\|_{X^{\a_{1},\b_{1}}_{s,b}}\,\|w_{2}\|_{X^{\a_{2},\b_{2}}_{s,b}}, \quad\forall\, \{w_{i}\}_{i=1,2},\]
		is equivalent to
		\be\label{weighted l2 form, nd1}
		\int\limits_{A}\frac{\xi_{3}\la\xi_{3}\ra^{s}\prod\limits_{i=1}^{3}f_{i}(\xi_{i},\tau_{i})}{\la\xi_{1}\ra^{s}\la\xi_{2}\ra^{s}\la L_{1}\ra^{b}\la L_{2}\ra^{b}\la L_{3}\ra^{1-b}} \leq C\,\prod_{i=1}^{3}\|f_{i}\|_{L^{2}_{\xi\tau}}, \quad\forall\,\{f_{i}\}_{1\leq i\leq 3},\ee
		where 
		\be\label{def of L}
		L_{i}=\tau_{i}-\phi^{\a_{i},\b_{i}}(\xi_{i}),\quad i=1,2,3.\ee
	\end{lemma}
	
	\subsection{Resonance functions and the characteristic quadratic function}
	\label{Subsec, res fcn and char quad fcn}
	
	Based on the discussion in Section \ref{Subsec, idea of proof}, the resonance function plays an essential role in establishing bilinear estimates. Now we follow \cite{Tao01} to give a formal definition to this function in the most general form.
	
	\begin{definition}[\cite{Tao01}]\label{Def, res fcn}
		Let $\big((a_1,\b_1),(\a_2,\b_2),(\a_3,\b_3)\big)$ be a triple in $(\m{R}^{*}\times\m{R})^{3}$. Define the {\it resonance function} $H$ associated to this triple by
		\be\label{res fcn}
		H(\xi_1,\xi_2,\xi_3)=\sum_{i=1}^{3}\phi^{\a_i,\b_i}(\xi_i),\quad \forall\,\sum_{i=1}^{3}\xi_{i}=0.\ee
		The {\it resonance set} of $H$ is defined to be the zero set of $H$, that is
		\be\label{res set}
		\Big\{(\xi_1,\xi_2,\xi_3)\in\m{R}^3:\sum_{i=1}^{3}\xi_i=0,\,H(\xi_1,\xi_2,\xi_3)=0\Big\}.\ee
	\end{definition}
	
	In particular, we introduce the notations of $H_{0}$, $H_{1}$ and $H_{2}$. 
	\begin{itemize}
		\item[(1)] The resonance function associated to the triple $\big((a_1,\b_1),(\a_1,\b_1),(\a_1,\b_1)\big)$ is denoted as $H_0$:
		\be\label{H_0}
		H_{0}(\xi_1,\xi_2,\xi_3)=\sum_{i=1}^{3}\phi^{\a_1,\b_1}(\xi_{i}),\quad \forall\,\sum_{i=1}^{3}\xi_{i}=0.\ee
		This applies to the classical case (\ref{classical bilin}) or the case $r=1$ in Table \ref{Table, bilin est}. By direct calculation, 
		\[H_{0}(\xi_1,\xi_2,\xi_3)=3\a_1\xi_1\xi_2\xi_3.\]
		
		\item[(2)] The resonance function associated to the triple $\big((a_1,\b_1),(\a_1,\b_1),(\a_2,\b_2)\big)$ is  denoted as $H_1$:
		\be\label{H_1}
		H_{1}(\xi_1,\xi_2,\xi_3)=\phi^{\a_1,\b_1}(\xi_1)+\phi^{\a_1,\b_1}(\xi_2)+\phi^{\a_2,\b_2}(\xi_3),\quad \forall\,\sum_{i=1}^{3}\xi_{i}=0.\ee
		This applies to the bilinear estimate of Type (D1). By direct calculation and writing $\xi_2=-(\xi_1+\xi_3)$, 
		\be\label{exp of H_1}
		H_{1}(\xi_1,\xi_2,\xi_3)=\xi_{3}\Big[(\a_2-\a_1)\xi_3^2-3\a_1\xi_1\xi_3-3\a_1\xi_1^2\Big]+(\b_1-\b_2)\xi_3.\ee
		If $\xi_3=0$, then $H_1=0$. If $\xi_3\neq 0$, then $H_1$ can be rewritten as 
		\be\label{simp exp of H_1}
		H_{1}(\xi_1,\xi_2,\xi_3)=-3\a_1\xi_3^3h_{r}\Big(\frac{\xi_1}{\xi_3}\Big)+(\b_1-\b_2)\xi_3,\ee
		where $r=\frac{\a_2}{\a_1}$ and 
		\be\label{def of h_r}
		h_{r}(x):=x^2+x+\frac{1-r}{3}.\ee

		\item[(3)] The resonance function $H_{2}$ associated to the triple $\big((a_1,\b_1),(\a_2,\b_2),(\a_1,\b_1)\big)$ is denoted as $H_2$:
		\be\label{H_2}
		H_{2}(\xi_1,\xi_2,\xi_3)=\phi^{\a_1,\b_1}(\xi_1)+\phi^{\a_2,\b_2}(\xi_2)+\phi^{\a_1,\b_1}(\xi_3),\quad \forall\,\sum_{i=1}^{3}\xi_{i}=0.\ee
		This applies to the bilinear estimates of Type (D2), (ND1) and (ND2). By direct calculation and writing $\xi_3=-(\xi_1+\xi_2)$, 
		\be\label{exp of H_2}
		H_{2}(\xi_1,\xi_2,\xi_3)=\xi_{2}\Big[(\a_2-\a_1)\xi_2^2-3\a_1\xi_1\xi_2-3\a_1\xi_1^2\Big]+(\b_1-\b_2)\xi_2.\ee
		If $\xi_2=0$, then $H_2=0$. If $\xi_2\neq 0$, then $H_2$ can be rewritten as 
		\be\label{simp exp of H_2}
		H_{2}(\xi_1,\xi_2,\xi_3)=-3\a_1\xi_2^3h_{r}\Big(\frac{\xi_1}{\xi_2}\Big)+(\b_1-\b_2)\xi_2,\ee
		where $r=\frac{\a_2}{\a_1}$ and $h_{r}$ is as defined in (\ref{def of h_r}).
	\end{itemize}
	According to the above computation, the quadratic function $h_{r}$ is essential to determine the behavior of $H_{1}$ and $H_{2}$, thus, it is a characterization of the coupled KdV-KdV systems.
	
	\begin{definition}
		The quadratic function $h_{r}$ in (\ref{def of h_r}) is called the {\it characteristic quadratic function} associated to the coupled KdV-KdV systems (\ref{ckdv, coef form}).
	\end{definition}

	\subsection{Proof of Theorem \ref{Thm, d2, neg}}
	\label{Subsec, proof of bilin, neg}
	
	For the convenience of the proof, we introduce some notations below. For any $r\in\m{R}$, we define $h_{r}$ as in (\ref{def of h_r}) and define $p_{r}:\m{R}\to\m{R}$ by 
	\be\label{def of p_r}
	p_{r}(x)=x^2+2x+1-r.\ee
	For fixed $\a_1,\a_2,\b_1,\b_2$ and for any $\xi,\tau\in\m{R}$, define $P_{\xi,\tau}$ and $Q_{\xi,\tau}$ from $\m{R}$ to $\m{R}$ as
	\begin{eqnarray}
		P_{\xi,\tau}(x) &=& (\a_1-\a_2)x^3+3\a_1\xi x^2+(3\a_1\xi^2+\b_2-\b_1)x+\phi^{\a_1,\b_1}(\xi)-\tau  \label{def of P}, \\
		Q_{\xi,\tau}(x) &=& 3\a_1\xi x^2+3\a_1\xi^2 x+\phi^{\a_1,\b_1}(\xi)-\tau \label{def of Q},
	\end{eqnarray}
	where $\phi^{\a_1,\b_1}(\xi)=\a_1\xi^3-\b_1\xi$. In the case when $(\a_1,\b_1)=(\a_2,\b_2)$, $P_{\xi,\tau}$ reduces to $Q_{\xi,\tau}$.
	
	\bigskip
	
	\begin{proof}[Proof of Theorem \ref{Thm, d2, neg}]
		We will provide details for Case (1) and then briefly mention Case (2) and Case (3).

		\noindent {\bf Proof of Case (1).}
		
		For  $-\frac{13}{12}\leq s\leq -1$ and $\frac{1}{4}-\frac{s}{3}\leq b\leq \frac{4}{3}+\frac{2s}{3}$,  let $\rho=-s$. Then 
		\be\label{d2, neg, cond on b}
		1\leq \rho\leq \frac{13}{12}, \qquad \frac{1}{4}+\frac{\rho}{3}\leq b\leq \frac{4}{3}-\frac{2\rho}{3}.\ee
		According to Lemma \ref{Lemma, bilin to weighted l2}, it suffices to prove 
		\be\label{d2, neg, weighted l2 form}
		\int\limits_{A}\frac{|\xi_{3}|\la\xi_{1}\ra^{\rho}\la\xi_{2}\ra^{\rho}\prod\limits_{i=1}^{3}|f_{i}(\xi_{i},\tau_{i})|}{\la\xi_{3}\ra^{\rho}\la L_{1}\ra^{b}\la L_{2}\ra^{b}\la L_{3}\ra^{1-b}} \leq C\,\prod_{i=1}^{3}\|f_{i}\|_{L^{2}_{\xi\tau}}, \quad\forall\,\{f_{i}\}_{1\leq i\leq 3},\ee
		where $A$ is as defined in (\ref{int domain}) and 
		\be\label{d2, neg, L}
		L_{1}=\tau_1-\phi^{\a_1,\b_1}(\xi_1),\quad L_{2}=\tau_2-\phi^{\a_2,\b_2}(\xi_2), \quad L_{3}=\tau_3-\phi^{\a_1,\b_1}(\xi_3).\ee
		The resonance function $H_2$ is as defined in (\ref{H_2}).
		
		Since $r:=\frac{\a_2}{\a_1}<0$, the function $h_{r}$, as defined in (\ref{def of h_r}), has no real roots, so there exists $\delta_1=\delta_1(\a_1,\a_2)$ such that 
		\be\label{d2, neg, lbd for h_r}
		h_{r}(x)\geq \delta_{1}(1+x^2),\quad\forall\,x\in\m{R}.\ee
		Then according to (\ref{simp exp of H_2}), there exists $\delta_{2}=\delta_{2}(\a_1,\a_2)$ such that 
		\[|H_2(\xi_1,\xi_2,\xi_3)|\geq \delta_{2}|\xi_2|(\xi_1^2+\xi_2^2)-|\b_1-\b_2||\xi_2|.\]
		Now if $|\b_2-\b_1|\leq \eps_1$ with sufficiently small $\eps_1$ depending only on $\delta_2$, 
		\[\la H_2(\xi_1,\xi_2,\xi_3)\ra\geq \frac{\delta_2}{2}|\xi_2|(\xi_1^2+\xi_2^2).\]
		Since $\sum\limits_{i=1}^{3}\xi_i=0$, the above estimate implies that 
		\be\label{d2, neg, lbd for H_2}\la H_2(\xi_1,\xi_2,\xi_3)\ra\gs |\xi_2|\sum_{i=1}^{3}\xi_i^2.\ee
		Define $\text{MAX}=\max\{\la L_{1}\ra, \la L_2\ra, \la L_3\ra\}$. Then it follows from $H_{2}=-\sum\limits_{i=1}^{3}L_{i}$ that $\text{MAX}\geq \frac{1}{3}\la H_2\ra$. Therefore, 
		
		\be\label{d2, neg, lbd for MAX}
		\text{MAX}\gs |\xi_2|\sum_{i=1}^{3}\xi_i^2.\ee
		Decompose the region  $A$  as $\bigcup\limits_{i=0}^{3}A_{i}$, where
		\be\label{d2, neg, region division}
		\begin{split}
			&A_{0}=\{(\vec{\xi},\vec{\tau})\in A:|\xi_{1}|\leq 1\,\,\text{or}\,\,|\xi_{2}|\leq 1\}, \vspace{0.05in}\\
			&A_{i}=\{(\vec{\xi},\vec{\tau})\in A:|\xi_{1}|> 1,\,|\xi_{2}|>1\,\,\text{and}\,\,\la L_{i}\ra=\text{MAX}\},\quad 1\leq i\leq 3.
		\end{split}\ee
		
		\noindent
		{\bf Contribution on $A_{0}$}: 
		
		\quad Since $\la\xi_1\ra\la\xi_2\ra\ls \la\xi_3\ra$ when $|\xi_1|\leq 1$ or $|\xi_2|\leq 1$,
		\begin{eqnarray}
			\int\frac{|\xi_{3}|\la\xi_{1}\ra^{\rho}\la\xi_{2}\ra^{\rho}\prod\limits_{i=1}^{3}|f_{i}(\xi_{i},\tau_{i})|}{\la\xi_{3}\ra^{\rho}\la L_{1}\ra^{b}\la L_{2}\ra^{b}\la L_{3}\ra^{1-b}} &\ls &
			\int\frac{|\xi_{3}|\prod\limits_{i=1}^{3}|f_{i}(\xi_{i},\tau_{i})|}{\la L_{1}\ra^{b}\la L_{2}\ra^{b}\la L_{3}\ra^{1-b}} \nonumber\\
			&=&
			\iint\frac{|\xi_{3}\|f_{3}|}{\la L_{3}\ra^{1-b}}\bigg(\iint\frac{|f_{1}f_{2}|}{\la L_{1}\ra^{b}\la L_{2}\ra^{b}}\,d\tau_{2}d\xi_{2}\bigg)\,d\tau_{3}\,d\xi_{3}. \label{d2, neg, A0, region simp}
		\end{eqnarray}
		In order to bound the above integral by $C\prod\limits_{i=1}^{3}\|f_{i}\|_{L^{2}_{\xi\tau}}$, it suffices to show
		\be\label{d2, neg, A0, CS}
		\sup_{\xi_{3},\tau_{3}}\frac{|\xi_{3}|}{\la L_{3}\ra^{1-b}}\bigg(\iint\frac{d\tau_{2}d\xi_{2}}{\la L_{1}\ra^{2b}\la L_{2}\ra^{2b}}\bigg)^{\frac{1}{2}}\leq C\ee
		due to  the same argument as in \cite{KPV96} via the Cauchy-Schwartz inequality. Next, for any fixed $\xi_3$ and $\tau_{3}$, we will estimate 
		\[\iint\frac{d\tau_{2}d\xi_{2}}{\la L_{1}\ra^{2b}\la L_{2}\ra^{2b}}.\] 
		Since $\tau_{1}=-\tau_2-\tau_3$ and $\xi_{1}=-\xi_2-\xi_3$, $L_{1}$ can be written as 
		\[L_{1}=-\tau_2-\tau_3-\phi^{\a_1,\b_1}(-\xi_2-\xi_3).\]
		Meanwhile, recalling $L_{2}=\tau_2-\phi^{\a_2,\b_2}(\xi_2)$, it then follows from Lemma \ref{Lemma, int in tau} that 
		\be\label{use of lemma: int in tau}
		\int\frac{d\tau_{2}}{\la L_{1}\ra^{2b}\la L_{2}\ra^{2b}}\ls \frac{1}{\la L_{1}+L_{2}\ra^{2b}}. \ee
		So (\ref{d2, neg, A0, CS}) is reduced to 
		\[\sup_{\xi_{3},\tau_{3}}\frac{|\xi_{3}|}{\la L_{3}\ra^{1-b}}\bigg(\int\frac{d\xi_{2}}{\la L_{1}+L_{2}\ra^{2b}}\bigg)^{\frac{1}{2}}\leq C,\]
		or equivalently, 
		\be\label{d2, neg, A0, L1+L2}
		\sup_{\xi_{3},\tau_{3}}\frac{|\xi_{3}|^2}{\la L_{3}\ra^{2(1-b)}}\int\frac{d\xi_{2}}{\la L_{1}+L_{2}\ra^{2b}}\leq C.\ee
		By direct calculation, we find
		\be\label{L1+L2, H2}
		L_{1}+L_{2}=P_{\xi_3,\tau_3}(\xi_2),\ee
		where $P_{\xi_{3},\tau_{3}}$ is as defined in (\ref{def of P}) with $(\xi,\tau)$ being replaced by $(\xi_3,\tau_3)$. Hence, (\ref{d2, neg, A0, L1+L2}) is further reduced to 
		\be\label{d2, neg, A0, single int}
		\sup_{\xi_{3},\tau_{3}}\frac{|\xi_{3}|^2}{\la L_{3}\ra^{2(1-b)}}\int \frac{d\xi_2}{\la P_{\xi_3,\tau_3}(\xi_2)\ra^{2b}}\leq C.\ee
		There are two situations.
		\begin{itemize}
			\item $|\xi_{3}|\leq 1$. In this situation, it suffices to prove $\int \frac{d\xi_2}{\la P_{\xi_3,\tau_3}(\xi_2)\ra^{2b}}$ is bounded. Since $P_{\xi_{3},\tau_{3}}(\xi_2)$ is a cubic function in $\xi_2$, the boundedness of this integral follows from Lemma \ref{Lemma, bdd int}.
			
			\item $|\xi_{3}|\geq 1$. In this situation, 
			\be\label{H2, deri of L1+L2}
			P_{\xi_3,\tau_3}'(\xi_2)=3(\a_1-\a_2)\xi_2^2+6\a_1\xi_3\xi_2+3\a_1\xi_3^2+\b_2-\b_1.\ee
			When $\xi_2\neq 0$, 
			\be\label{H2, deri of L1+L2, simp}
			P_{\xi_3,\tau_3}'(\xi_2)=3\a_1\xi_2^2p_{r}\Big(\frac{\xi_3}{\xi_2}\Big)+\b_2-\b_1,\ee
			where $p_r$ is as defined in (\ref{def of p_r}). Since $r<0$, $p_{r}$ does not have any real roots. Therefore, there exists $\delta_{3}=\delta_{3}(\a_1,\a_2)$ such that 
			\[p_{r}(x)\geq \delta_{3}(1+x^2), \quad\forall\,x\in\m{R}.\]
			As a result, there exists $\delta_{4}=\delta_{4}(\a_1,\a_2)$ such that 
			\[|P_{\xi_3,\tau_3}'(\xi_2)|\geq \delta_{4}(\xi_2^2+\xi_3^2)-|\b_2-\b_1|.\]
			Since $|\xi_3|\geq 1$, when $|\b_2-\b_1|$ is sufficiently small, 
			\be\label{d2, neg, lbd for deri of L1+L2}
			|P_{\xi_3,\tau_3}'(\xi_2)|\gs \xi_2^2+\xi_3^2.\ee
			
			Hence, 
			\[\int \frac{d\xi_2}{\la P_{\xi_3,\tau_3}(\xi_2)\ra^{2b}}\ls \frac{1}{|\xi_{3}|^{2}}\i_{\m{R}}\frac{|P_{\xi_3,\tau_3}'(\xi_2)|}{\la P_{\xi_3,\tau_3}(\xi_2)\ra^{2b}}\,d\xi_{2}\ls \frac{1}{|\xi_{3}|^{2}},\]
			which also justifies (\ref{d2, neg, A0, single int}).
		\end{itemize}
		
		\smallskip
		\noindent {\bf Contribution on $A_{3}$}: 
		
		\quad Since  $\la\xi_{i}\ra\sim |\xi_{i}|$  when $|\xi_{i}|>1$ for $i=1,2$, 
		\be\label{d2, neg, A3, region simp}
		\int\frac{|\xi_{3}|\la\xi_{1}\ra^{\rho}\la\xi_{2}\ra^{\rho}\prod\limits_{i=1}^{3}|f_{i}(\xi_{i},\tau_{i})|}{\la\xi_{3}\ra^{\rho}\la L_{1}\ra^{b}\la L_{2}\ra^{b}\la L_{3}\ra^{1-b}} \ls 
		\iint\frac{|\xi_{3}\|f_{3}|}{\la\xi_{3}\ra^{\rho}\la L_{3}\ra^{1-b}}\bigg(\iint\frac{|\xi_{1}\xi_{2}|^{\rho}|f_{1}f_{2}|}{\la L_{1}\ra^{b}\la L_{2}\ra^{b}}\,d\tau_{2}d\xi_{2}\bigg)\,d\tau_{3}\,d\xi_{3}. \ee
		In order to bound the above integral by $C\prod\limits_{i=1}^{3}\|f_{i}\|_{L^{2}_{\xi\tau}}$, similar to the derivation from (\ref{d2, neg, A0, region simp}) to (\ref{d2, neg, A0, single int}), it suffices to show 
		\be\label{d2, neg, A3, single int}
		\sup_{\xi_{3},\tau_{3}}\frac{|\xi_{3}|^2}{\la\xi_{3}\ra ^{2\rho}\la L_{3}\ra ^{2(1-b)}}\int\frac{|\xi_{1}\xi_{2}|^{2\rho}}{\la P_{\xi_3,\tau_3}(\xi_{2})\ra ^{2b}}\,d\xi_{2}\leq C,\ee
		where $P_{\xi_3,\tau_3}(\xi_{2})$ is the same as (\ref{L1+L2, H2}). Then by analogous derivation from (\ref{H2, deri of L1+L2}) to (\ref{d2, neg, lbd for deri of L1+L2}), it also holds $|P_{\xi_3,\tau_3}'(\xi_2)|\gs \xi_2^2+\xi_3^2$. Moreover, since $\sum\limits_{i=1}^{3}\xi_i=0$, 
		\be\label{d2, neg, A3, lbd for deri of L1+L2}
		|P_{\xi_3,\tau_3}'(\xi_2)|\gs \sum_{i=1}^{3}\xi_{i}^{2}\gs |\xi_1\xi_2|.\ee
		As a result, (\ref{d2, neg, A3, single int}) is reduced to
		\[\sup_{\xi_{3},\tau_{3}}\int\frac{|\xi_{3}|^{2}|\xi_{1}\xi_{2}|^{2\rho-1}}{\la\xi_{3}\ra^{2\rho}\la L_{3}\ra ^{2(1-b)}}\frac{|P_{\xi_3,\tau_3}'(\xi_2)|}{\la |P_{\xi_3,\tau_3}(\xi_2)|\ra ^{2b}}\,d\xi_{2}\leq C.\]
		Since $\la L_{3}\ra=\text{MAX}$ on $A_{3}$ and 
		\[\int\frac{|P_{\xi_3,\tau_3}'(\xi_2)|}{\la |P_{\xi_3,\tau_3}(\xi_2)|\ra ^{2b}}\,d\xi_{2}<\infty,\]
		it suffices to show
		\be\label{d2, neg, A3, final est}
		\frac{|\xi_{3}|^{2}|\xi_{1}\xi_{2}|^{2\rho-1}}{\la\xi_{3}\ra^{2\rho}(\text{MAX})^{2(1-b)}}\leq C.\ee
		To this end,  note  that  $|\xi_{3}|\leq \la\xi_{3}\ra^{\rho}$ as $\rho\geq 1$. Moreover,  it follows from (\ref{d2, neg, lbd for MAX}) that 
		$\la \text{MAX}\ra\gs  |\xi_{1}\xi_{2}|^{\frac{3}{2}}$. Consequently,  
		\begin{eqnarray*}
			\frac{|\xi_{3}|^{2}|\xi_{1}\xi_{2}|^{2\rho-1}}{\la\xi_{3}\ra^{2\rho}(\text{MAX})^{2(1-b)}} &\ls & \frac{|\xi_{1}\xi_{2}|^{2\rho-1}}{|\xi_{1}\xi_{2}|^{3(1-b)}}=|\xi_1\xi_2|^{3b+2\rho-4}.
		\end{eqnarray*}
		Noticing the restriction (\ref{d2, neg, cond on b}) implies $3b+2\rho-4\leq 0$, so $|\xi_1\xi_2|^{3b+2\rho-4}\leq 1$.
		
		\smallskip
		\noindent {\bf Contribution on $A_{1}$}: 
		
		\quad Since $\la L_{1}\ra=\text{MAX}$ on $A_{1}$, 
		\[\frac{1}{\la L_{1}\ra^{b\la }L_{3}\ra^{1-b}}\leq \frac{1}{\la L_{1}\ra^{1-b}\la L_{3}\ra^{b}}.\]
		Therefore,
		\begin{eqnarray*}
			\int\frac{|\xi_{3}|\la\xi_{1}\ra^{\rho}\la\xi_{2}\ra^{\rho}\prod\limits_{i=1}^{3}|f_{i}(\xi_{i},\tau_{i})|}{\la\xi_{3}\ra^{\rho}\la L_{1}\ra^{b}\la L_{2}\ra^{b}\la L_{3}\ra^{1-b}} &\ls &
			\int\frac{|\xi_{3}\|\xi_{1}\xi_{2}|^{\rho}\prod\limits_{i=1}^{3}|f_{i}(\xi_{i},\tau_{i})|}{\la \xi_{3}\ra^{\rho}\la L_{1}\ra^{1-b}\la L_{2}\ra^{b}\la L_{3}\ra^{b}}\\
			&=&
			\iint\frac{|f_{1}\|\xi_{1}|^{\rho}}{\la L_{1}\ra^{1-b}}\bigg(\iint\frac{|\xi_{2}|^{\rho}|\xi_{3}||f_{2}f_{3}|}{\la\xi_{3}\ra^{\rho}\la L_{2}\ra^{b}\la L_{3}\ra^{b}}\,d\tau_{2}d\xi_{2}\bigg)\,d\tau_{1}\,d\xi_{1}.
		\end{eqnarray*}
		Then similar to the derivation from (\ref{d2, neg, A0, region simp}) to (\ref{d2, neg, A0, L1+L2}), it suffices to show 
		\be\label{d2, neg, A1, L2+L3}
		\sup_{\xi_1,\tau_1}\frac{|\xi_{1}|^{2\rho}}{\la L_{1}\ra^{2(1-b)}}\int\frac{|\xi_{2}|^{2\rho}|\xi_{3}|^{2}}{\la\xi_{3}\ra^{2\rho}\la L_2+L_3\ra^{2b}}\,d\xi_{2}\leq C.\ee
		For any fixed $(\xi_1,\tau_1)$, writing $\tau_{3}=-\tau_2-\tau_{1}$ and $\xi_3=-\xi_2-\xi_1$, then by direct calculation, we find
		\be\label{L2+L3, H2}
		L_{2}+L_{3}=P_{\xi_1,\tau_1}(\xi_2),\ee
		where $P_{\xi_{1},\tau_{1}}$ is as defined in (\ref{def of P}) with $(\xi,\tau)$ being replaced by $(\xi_1,\tau_1)$. Hence, (\ref{d2, neg, A1, L2+L3}) is further reduced to 
		\be\label{d2, neg, A1, single int}
		\sup_{\xi_1,\tau_1}\frac{|\xi_{1}|^{2\rho}}{\la L_{1}\ra^{2(1-b)}}\int\frac{|\xi_{2}|^{2\rho}|\xi_{3}|^{2}}{\la\xi_{3}\ra^{2\rho}\la P_{\xi_1,\tau_1}(\xi_2)\ra^{2b}}\,d\xi_{2}\leq C.\ee
		Then by analogous derivation from (\ref{H2, deri of L1+L2}) to (\ref{d2, neg, lbd for deri of L1+L2}), for sufficiently small $|\b_2-\b_1|$, we have 
		\[|P_{\xi_1,\tau_1}'(\xi_2)|\gs \xi_1^2+\xi_2^2\geq |\xi_1\xi_2|.\]
		Based on this estimate, the rest argument is similar to that for the region $A_3$ after (\ref{d2, neg, A3, lbd for deri of L1+L2}). 
		
		\smallskip
		\noindent {\bf Contribution on $A_{2}$}: First, we  decompose $A_{2}$ into  three parts: $A_{2}=\bigcup\limits_{i=1}^{3}A_{2i}$ with 
		\be\label{d2, neg, decom of A2}
		\left\{\begin{array}{l}
			A_{21}=\{(\vec{\xi},\vec{\tau})\in A_{2}: |\xi_{1}|< \frac{1}{3}|\xi_{2}|\},\vspace{0.05in}\\
			A_{22}=\{(\vec{\xi},\vec{\tau})\in A_{2}: \frac{1}{3}|\xi_{2}|\leq |\xi_{1}|\leq \frac{2}{3}|\xi_{2}|\}, \vspace{0.05in}\\
			A_{23}=\{(\vec{\xi},\vec{\tau})\in A_{2}: |\xi_{1}|> \frac{2}{3}|\xi_{2}|\}.
		\end{array}\right.\ee
		\begin{itemize}
			\item On $A_{21}$ or $A_{23}$, since $\la L_{2}\ra=\text{MAX}$,
			\[\frac{1}{\la L_{2}\ra^{b\la }L_{3}\ra^{1-b}}\leq \frac{1}{\la L_{2}\ra^{1-b}\la L_{3}\ra^{b}}.\]
			Thus,
			\begin{eqnarray*}
				\int\frac{|\xi_{3}|\la\xi_{1}\ra^{\rho}\la\xi_{2}\ra^{\rho}\prod\limits_{i=1}^{3}|f_{i}(\xi_{i},\tau_{i})|}{\la\xi_{3}\ra^{\rho}\la L_{1}\ra^{b}\la L_{2}\ra^{b}\la L_{3}\ra^{1-b}} &\ls &
				\int\frac{|\xi_{3}\|\xi_{1}\xi_{2}|^{\rho}\prod\limits_{i=1}^{3}|f_{i}(\xi_{i},\tau_{i})|}{\la \xi_{3}\ra^{\rho}\la L_{1}\ra^{b}\la L_{2}\ra^{1-b}\la L_{3}\ra^{b}}\\
				&=&
				\iint\frac{|f_{2}\|\xi_{2}|^{\rho}}{\la L_{2}\ra^{1-b}}\bigg(\iint\frac{|\xi_{1}|^{\rho}|\xi_{3}||f_{1}f_{3}|}{\la\xi_{3}\ra^{\rho}\la L_{1}\ra^{b}\la L_{3}\ra^{b}}\,d\tau_{1}d\xi_{1}\bigg)\,d\tau_{2}\,d\xi_{2}.
			\end{eqnarray*}
			Then similar to the derivation from (\ref{d2, neg, A0, region simp}) to (\ref{d2, neg, A0, L1+L2}), it suffices to show 
			\be\label{d2, neg, A2, L1+L3}
			\sup_{\xi_2,\tau_2}\frac{|\xi_{2}|^{2\rho}}{\la L_{2}\ra^{2(1-b)}}\int\frac{|\xi_{1}|^{2\rho}|\xi_{3}|^{2}}{\la\xi_{3}\ra^{2\rho}\la L_1+L_3\ra^{2b}}\,d\xi_{1}\leq C.\ee
			For any fixed $(\xi_2,\tau_2)$, writing $\tau_{3}=-\tau_1-\tau_{2}$ and $\xi_3=-\xi_1-\xi_2$, then by direct calculation, we find
			\be\label{L1+L3, H2}
			L_{1}+L_{3}=Q_{\xi_2,\tau_2}(\xi_1),\ee
			where $Q_{\xi_{2},\tau_{2}}$ is as defined in (\ref{def of Q}) with $(\xi,\tau)$ being replaced by $(\xi_2,\tau_2)$. Hence, (\ref{d2, neg, A2, L1+L3}) is further reduced to 
			\be\label{d2, neg, A2, single int}
			\sup_{\xi_2,\tau_2}\frac{|\xi_{2}|^{2\rho}}{\la L_{2}\ra^{2(1-b)}}\int\frac{|\xi_{1}|^{2\rho}|\xi_{3}|^{2}}{\la\xi_{3}\ra^{2\rho}\la Q_{\xi_2,\tau_2}(\xi_1) \ra^{2b}}\,d\xi_{1}\leq C.\ee
			Again by direct calculation, 
			\be\label{d2, neg, deri of L1+L3} 
			Q_{\xi_2,\tau_2}'(\xi_1)=6\a_1\xi_2\xi_1+3\a_1\xi_2^2
			=3\a_1\xi_2(2\xi_1+\xi_2).\ee
			According to the definition of $A_{21}$ and $A_{23}$ in (\ref{d2, neg, decom of A2}), either $|\xi_1|<\frac13|\xi_2|$ or $|\xi_1|>\frac23|\xi_2|$, so it follows from (\ref{d2, neg, deri of L1+L3}) that 
			\[|Q_{\xi_2,\tau_2}'(\xi_1)|\gs |\xi_1\xi_2|.\]
			Based on this estimate, the rest argument is similar to that for the region $A_3$ after (\ref{d2, neg, A3, lbd for deri of L1+L2}). 
			
			\item  On $A_{22}$,  we have  $|\xi_{1}|\sim |\xi_{2}|\sim |\xi_{3}|$, so
			\begin{eqnarray*}
				\int\frac{|\xi_{3}|\la\xi_{1}\ra^{\rho}\la\xi_{2}\ra^{\rho}\prod\limits_{i=1}^{3}|f_{i}(\xi_{i},\tau_{i})|}{\la\xi_{3}\ra^{\rho}\la L_{1}\ra^{b}\la L_{2}\ra^{b}\la L_{3}\ra^{1-b}} &\ls &
				\int\frac{|\xi_{2}|^{1+\rho}\prod\limits_{i=1}^{3}|f_{i}(\xi_{i},\tau_{i})|}{\la L_{1}\ra^{b}\la L_{2}\ra^{b}\la L_{3}\ra^{1-b}}\\
				&=&
				\iint\frac{|f_{2}\|\xi_{2}|^{1+\rho}}{\la L_{2}\ra^{b}}\bigg(\iint\frac{|f_{1}f_{3}|}{\la L_{1}\ra^{b}\la L_{3}\ra^{1-b}}\,d\tau_{1}d\xi_{1}\bigg)\,d\tau_{2}\,d\xi_{2}.
			\end{eqnarray*}
			Then similar to the derivation from (\ref{d2, neg, A0, region simp}) to (\ref{d2, neg, A0, L1+L2}), it suffices to show 
			\[\sup_{\xi_2,\tau_2}\frac{|\xi_{2}|^{2(1+\rho)}}{\la L_{2}\ra^{2b}}\int\frac{d\xi_{1}}{\la L_1+L_3\ra^{2(1-b)}}\leq C.\]
			That is to prove 
			\be\label{d2, neg, A22, L1+L3}
			\sup_{\xi_2,\tau_2}\frac{|\xi_{2}|^{2(1+\rho)}}{\la L_{2}\ra^{2b}}\int\frac{d\xi_{1}}{\la Q_{\xi_2,\tau_2}(\xi_{1})\ra^{2(1-b)}}\leq C,\ee
			where $Q_{\xi_2,\tau_2}(\xi_{1})=L_1+L_3$ is as defined in (\ref{L1+L3, H2}). By (\ref{L1+L3, H2}) and (\ref{def of Q}), 
			\[Q_{\xi_2,\tau_2}(\xi_{1})
			=3\a_1\xi_2\xi_1^2+3\a_1\xi_2^2\xi_1+\phi^{\a_1,\b_1}(\xi_2)-\tau_2.\]
			In other words, $Q_{\xi_2,\tau_2}(\xi_{1})$ is a quadratic function in $\xi_1$ with the leading coefficient $3\a_1\xi_2$. Since (\ref{d2, neg, cond on b}) implies that $2(1-b)>\frac12$, then it follows from Lemma \ref{Lemma, bdd int} that 
			\[\int\frac{d\xi_{1}}{\la Q_{\xi_2,\tau_2}(\xi_{1})\ra^{2(1-b)}}\ls |\xi_2|^{-\frac12}.\]
			Therefore, (\ref{d2, neg, A22, L1+L3}) reduces to 
			\be\label{d2, neg, A22, final est}
			\sup_{\xi_2,\tau_2}\frac{|\xi_{2}|^{\frac{3}{2}+2\rho}}{\la L_{2}\ra^{2b}}\leq C.\ee
			
			Since $\la L_{2}\ra=\text{MAX}$ on $A_{22}$, it follows from (\ref{d2, neg, lbd for MAX}) that $\la L_{2}\ra\gs |\xi_{2}|^{3}$. Hence, 
			\[\frac{|\xi_{2}|^{\frac{3}{2}+2\rho}}{\la L_{2}\ra^{2b}}\ls |\xi_{2}|^{\frac{3}{2}+2\rho-6b}.\] 
			Finally, due to the restriction $b\geq \frac14+\frac{\rho}{3}$ in (\ref{d2, neg, cond on b}), we have   $\frac{3}{2}+2\rho-6b\leq 0$ and $|\xi_{2}|^{\frac{3}{2}+2\rho-6b}\leq 1$.
		\end{itemize}
		
		\bigskip
		
		\noindent {\bf Proof of Case (2).}
		
		Let $\rho=-s$. By the assumption in this case, 
		\be\label{d2, neg 2, cond on b}
		\frac{3}{4}<\rho<1, \qquad \frac{1}{4}+\frac{\rho}{3}\leq b\leq 1-\frac{\rho}{3}.\ee
		As  in the proof for Case (1), we first  decompose  $A=\bigcup\limits_{i=0}^{3}A_{i}$ as in (\ref{d2, neg, region division}). 
		\begin{itemize}
			\item On $A_{0}$, the proof is the same as that for Case (1).
			\item  On $A_{3}$, it again reduces to prove (\ref{d2, neg, A3, final est}). Since $\rho<1$, it suffices to show
			\be\label{d2, neg 2, A3, final est}
			\frac{|\xi_{3}|^{2-2\rho}|\xi_{1}\xi_{2}|^{2\rho-1}}{(\text{MAX})^{2(1-b)}}\leq C.\ee 
			By  (\ref{d2, neg, lbd for MAX}),
			\[\text{MAX}\gs \max\{|\xi_{2}\xi_{3}^{2}|,\,|\xi_{1}^{2}\xi_{2}|,\,|\xi_{2}|^{3}\}\geq 1.\]
			Then it follows from $\frac{3}{4}<\rho<1$ that 
			\[|\xi_{3}|^{2-2\rho}|\xi_{1}\xi_{2}|^{2\rho-1}=|\xi_{2}\xi_{3}^{2}|^{1-\rho}|\xi_{1}^{2}\xi_{2}|^{\rho-\frac{1}{2}}|\xi_{2}^{3}|^{\frac{4\rho-3}{6}}\ls (\text{MAX})^{2\rho/3}.\]
			Finally, (\ref{d2, neg 2, A3, final est}) holds since (\ref{d2, neg 2, cond on b}) implies $\frac{2\rho}{3}\leq 2(1-b)$.
			
			\item On $A_{1}$, similarly, it reduces to prove (\ref{d2, neg 2, A3, final est}) which can be justified exactly the same as above.
			
			\item On $A_{2}$, we also decompose $A_{2}$ as (\ref{d2, neg, decom of A2}). The arguments on $A_{21}$ and $A_{23}$ are again reduced to prove (\ref{d2, neg 2, A3, final est}). The argument on $A_{22}$ is the same as that for Case (1) thanks to the condition $b\geq \frac14+\frac{\rho}{3}$ in (\ref{d2, neg 2, cond on b}).
		\end{itemize}
		
		\bigskip
		
		\noindent {\bf Proof of Case (3).}
		
		Since $\la\xi_1\ra\la\xi_2\ra\geq \la\xi_3\ra$, the left hand side of (\ref{d2, neg, weighted l2 form}) is an increasing function in $\rho$. So it suffices to consider the case when $s=-\frac{3}{4}$. Then it can be justified in the same way as that for Case (2).
		
	\end{proof}
	
	\subsection{Proof of Theorem \ref{Thm, bilin est, general}}
	\label{Subsec, proof of bilin, general}
	
	First, we want to point out several cases in Table \ref{Table, bilin est} which have been known or can be proved similarly.
	\begin{itemize}
		\item When $r=1$, Type (D1) and (D2) with $s>-\frac34$ were established in \cite{KPV96}. 
		
		\item When $r>\frac14$ but $r\neq 1$, Type (D1) and (D2) with $s\geq 0$ have been justified in \cite{Oh09a}. 
		The situations for Type (ND1) and (ND2) can be treated similarly.
		
		\item When $r=-1$, Type (D1) was proved in \cite{AC08}.
	\end{itemize}
	In all of the above results, it is assumed that $\b_1=\b_2=0$. But as we have seen from the proof of Theorem \ref{Thm, d2, neg}, even if $\b_1$ or $\b_2$ is not equal to $0$, they will not affect the conclusion as long as $|\b_2-\b_1|$ is small. 
	
	For the rest cases in Table \ref{Table, bilin est}, we will only provide proofs for the following typical ones.
	\begin{enumerate}[(1)]
		\item Among the cases when $r<0$ or $0<r<\frac14$, we will only prove Type (ND1) with $0<r<\frac14$, see Section \ref{Subsubsec, ND1, 0<r<1/4}. There are two reasons. Firstly, the cases when $0<r<\frac14$ is generally more difficult than the cases when $r<0$. Secondly, the non-divergence cases is more challenging than the divergence cases.
		
		\item When $r=\frac14$, the justifications for all four types are analogous, so we will still only focus on Type (ND1), see Section \ref{Subsubsec, ND1, r=1/4}.

		\item When $r=1$, Type (D1) and (D2) have been known and Type (ND1) and (ND2) are similar, so we will again only deal with Type (ND1), see Section \ref{Subsubsec, ND1, r=1}.
	\end{enumerate}
	
	As discussed above, only Type (ND1) will be investigated, so we list some common notations which will be used in Sections \ref{Subsubsec, ND1, 0<r<1/4}--\ref{Subsubsec, ND1, r=1}. First, we define the set $A$ as (\ref{int domain}), that is
	\[A:=\Big\{(\vec{\xi},\vec{\tau})\in\m{R}^{6}:\sum_{i=1}^{3}\xi_{i}=\sum_{i=1}^{3}\tau_{i}=0\Big\}.\]
	Then for any $(\vec{\xi},\vec{\tau})\in A$, we denote
	\[L_{1}=\tau_1-\phi^{\a_1,\b_1}(\xi_1),\quad L_{2}=\tau_2-\phi^{\a_2,\b_2}(\xi_2), \quad L_{3}=\tau_3-\phi^{\a_1,\b_1}(\xi_3).\]
	The resonance function is $H_{2}$ as defined in (\ref{H_2}). That is 
	\[H_{2}(\xi_1,\xi_2,\xi_3)=\phi^{\a_1,\b_1}(\xi_1)+\phi^{\a_2,\b_2}(\xi_2)+\phi^{\a_1,\b_1}(\xi_3)=-\sum_{i=1}^{3}L_{i}.\]
	In addition, we write $\text{MAX}=\max\{\la L_1\ra, \la L_2\ra, \la L_3\ra\}$. It is obvious that $\text{MAX}\gs |H_{2}(\xi_1,\xi_2,\xi_3)|$. Finally, we denote the functions $h_r$, $p_{r}$, $P_{\xi,\tau}$ and $Q_{\xi,\tau}$ as in (\ref{def of h_r}), (\ref{def of p_r}), (\ref{def of P}) and (\ref{def of Q}) respectively.
	
	\subsubsection{Type (ND1) with $0<r<\frac14$ and $s>-\frac34$}
	\label{Subsubsec, ND1, 0<r<1/4}
	
	Let  $\rho=-s$. Then $\rho<\frac34$. Similar to the argument as in the proof of Lemma \ref{Lemma, bilin to weighted l2}, one only needs  to show
	\be\label{nd1, 0<r<1/4, weighted l2 form}
	\int\limits_{A}\,\frac{|\xi_{1}|\la\xi_{1}\ra^{\rho}\la\xi_{2}\ra^{\rho}\prod\limits_{i=1}^{3}|f_{i}(\xi_{i},\tau_{i})|}{\la\xi_{3}\ra^{\rho}\la L_{1}\ra^{b}\la L_{2}\ra^{b}\la L_{3}\ra^{1-b}} \leq C\,\prod_{i=1}^{3}\|f_{i}\|_{L^{2}_{\xi\tau}}, \quad\forall\,\{f_{i}\}_{1\leq i\leq 3}.\ee
	Since $\frac{\la\xi_{1}\ra\la\xi_{2}\ra}{\la\xi_{3}\ra}\geq 1$, it suffices to consider the case when $\frac{9}{16}\leq \rho<\frac{3}{4}$. Assume 
	\be\label{nd1, 0<r<1/4, cond on b}
	\frac{1}{2}<b\leq \frac{3}{4}-\frac{\rho}{3}:=b_0.\ee
	Since $0<r<\frac{1}{4}$, the function $h_{r}$ has no real roots. Then by the similar argument from (\ref{d2, neg, lbd for h_r}) to (\ref{d2, neg, lbd for MAX}) in Section \ref{Subsec, proof of bilin, neg}, there exist $\eps$ and $\delta$, which only depend on $\a_1$ and $\a_2$, such that whenever $|\b_2-\b_1|\leq \eps$, it holds
	\be\label{nd1, 0<r<1/4, lbd for MAX}
	\text{MAX}\geq \delta |\xi_2|\sum_{i=1}^{3}\xi_i^2.\ee
	Decompose  the region $A=\bigcup\limits_{i=0}^{3}A_{i}$ as in (\ref{d2, neg, region division}), that is
	\be\label{region division, nd1, 0<r<1/4}
	\begin{split}
		&A_{0}=\{(\vec{\xi},\vec{\tau})\in A:|\xi_{1}|\leq 1\,\,\text{or}\,\,|\xi_{2}|\leq 1\}, \\
		&A_{i}=\{(\vec{\xi},\vec{\tau})\in A:|\xi_{1}|> 1,\,|\xi_{2}|>1\,\,\text{and}\,\,\la L_{i}\ra=\text{MAX}\},\quad 1\leq i\leq 3.
	\end{split}\ee
	Among the above regions $ \{A_i\}_{i=0}^{3} $, the most challenging region is $ A_3 $, so we will only show how we estimate on this region next.
	
	\noindent {\bf Contribution on $A_{3}$}: 
	
	Similar to the derivation for (\ref{d2, neg, A3, single int}), it suffices to show 
	\be\label{nd1, 0<r<1/4, A3, single int}
	\sup_{\xi_{3},\tau_{3}}\frac{1}{\la\xi_{3}\ra ^{2\rho}\la L_{3}\ra ^{2(1-b)}}\int\frac{|\xi_{1}|^{2(1+\rho)}|\xi_{2}|^{2\rho}}{\la P_{\xi_3,\tau_3}(\xi_2)\ra ^{2b}}\,d\xi_{2}\leq C.\ee
	Since $\xi_{2}\neq 0$, then it follows from (\ref{H2, deri of L1+L2, simp}) that
	\[P_{\xi_3,\tau_3}'(\xi_2)=3\a_1\xi_2^2p_{r}\Big(\frac{\xi_3}{\xi_2}\Big)+\b_2-\b_1,\]
	where $p_r$ is as defined in (\ref{def of p_r}). Since $0<r<\frac14$, $p_{r}$ has two roots $x_{1r}=-1-\sqrt{r}$ and $x_{2r}=-1+\sqrt{r}$ which satisfy 
	\[-2<x_{1r}<-1<x_{2r}<0.\]
	So there exists a positive constant $\sigma_{r}$, depending only on $r$, such that 
	\[[x_{1r}-2\sigma_{r},x_{1r}+2\sigma_{r}]\subset (-2,-1) \quad\text{and}\quad [x_{2r}-2\sigma_{r},x_{2r}+2\sigma_{r}]\subset (-1,0).\]
	The region  $A_{3}$  is accordingly decomposed   further as $A_{3}=\bigcup\limits_{i=1}^{3}A_{3i}$, where
	\[\left\{\begin{array}{l}
		A_{31}=\big\{(\vec{\xi},\vec{\tau})\in A_{3}: \big|\frac{\xi_{3}}{\xi_{2}}-x_{1r}\big|\geq \sigma_{r}\,\,\text{and}\,\,\big|\frac{\xi_{3}}{\xi_{2}}-x_{2r}\big|\geq \sigma_{r}\big\}, \vspace{0.1in}\\
		A_{32}=\big\{(\vec{\xi},\vec{\tau})\in A_{3}: \big|\frac{\xi_{3}}{\xi_{2}}-x_{1r}\big|<\sigma_{r} \big\}, \vspace{0.1in}\\
		A_{33}=\big\{(\vec{\xi},\vec{\tau})\in A_{3}: \big|\frac{\xi_{3}}{\xi_{2}}-x_{2r}\big|<\sigma_{r} \big\}.
	\end{array}\right.\]
	\begin{itemize}
		\item On $A_{31}$, since $\frac{\xi_3}{\xi_2}$ is away from the roots of $p_{r}$, there exists $\delta$, depending only on $r$, such that 
		\[p_{r}\Big(\frac{\xi_3}{\xi_2}\Big)\geq \delta \Big[1+\Big(\frac{\xi_3}{\xi_2}\Big)^2\Big].\]
		Hence, 
		\[|P_{\xi_3,\tau_3}'(\xi_2)|\geq 3|\a_{1}|\delta(\xi_2^2+\xi_3^2)-|\b_2-\b_1|.\]
		When $|\b_2-\b_1|$ is sufficiently small, 
		\be\label{nd1, 0<r<1/4, A31, lbd for deri of L1+L2}
		|P_{\xi_3,\tau_3}'(\xi_2)|\gs \xi_2^2+\xi_3^2\gs \xi_1^2.\ee
		Then 
		\[\text{LHS of (\ref{nd1, 0<r<1/4, A3, single int})}\ls \sup_{\xi_{3},\tau_{3}}\i_{\m{R}}\frac{|\xi_{1}\xi_{2}|^{2\rho}}{\la\xi_{3}\ra ^{2\rho}\la L_{3}\ra ^{2(1-b)}}\,\frac{|P_{\xi_3,\tau_3}'(\xi_2)|}{\la P_{\xi_3,\tau_3}(\xi_2)\ra ^{2b}}\,d\xi_{2}.\]
		In order to prove the boundedness of the above integral, it suffices to show 
		\be\label{nd1, 0<r<1/4, A31, final est}
		\frac{|\xi_{1}\xi_{2}|^{2\rho}}{\la\xi_{3}\ra ^{2\rho}\la L_{3}\ra ^{2(1-b)}}\leq C.\ee
		Since $\la L_{3}\ra=\text{MAX}$, it follows from (\ref{nd1, 0<r<1/4, lbd for MAX}) that $\la L_{3}\ra\gs |\xi_{2}|(\xi_1^2+\xi_2^2)\gs |\xi_1\xi_2|^{\frac32}$. Finally, due to the restriction $\rho<\frac34$ and the choice (\ref{nd1, 0<r<1/4, cond on b}) for $b$, we have $3(1-b)\geq 2\rho$. Therefore, 
		\[\la L_{3}\ra ^{2(1-b)} \gs |\xi_1\xi_2|^{3(1-b)} \gs
		|\xi_1\xi_2|^{2\rho},\]
		which implies (\ref{nd1, 0<r<1/4, A31, final est}).

		\item On $A_{32}$, it is easily seen that 
		$|\xi_1|\sim |\xi_2|\sim |\xi_3|$. Then  
		$\la L_{3}\ra=\text{MAX}\gs |\xi_2|\sum\limits_{i=1}^{3}\xi_i^2\gs |\xi_{3}|^{3}$. Therefore, 
		\begin{eqnarray}
			\text{LHS of (\ref{nd1, 0<r<1/4, A3, single int})} &\sim & \sup_{\xi_{3},\tau_{3}}\frac{|\xi_{3}|^{4\rho+2}}{\la\xi_{3}\ra ^{2\rho}\la L_{3}\ra ^{2(1-b)}}\int\frac{d\xi_{2}}{\la P_{\xi_3,\tau_3}(\xi_2)\ra ^{2b}} \nonumber\\
			& \ls  & \sup_{\xi_{3},\tau_{3}}\frac{|\xi_{3}|^{4\rho+2}}{\la\xi_{3}\ra ^{2\rho}\la \xi_{3}\ra ^{6(1-b)}}\int\frac{d\xi_{2}}{\la P_{\xi_3,\tau_3}(\xi_2)\ra ^{2b}}. \label{nd1, 0<r<1/4, A32, single int}
		\end{eqnarray}
		Since 
		\[P_{\xi_3,\tau_3}(\xi_2)=(\a_1-\a_2)\xi_2^3+3\a_1\xi_3\xi_2^2+(3\a_1\xi_3^2+\b_2-\b_1)\xi_2+\phi^{\a_1,\b_1}(\xi_3)-\tau_3,\]
		then by dividing the leading coefficient $\a_1-\a_2$, we have
		\[\la P_{\xi_3,\tau_3}(\xi_2)\ra\sim \la \xi_2^3+\sigma_{2}\xi_2^2+\sigma_1\xi_2+\sigma_0\ra,\]
		where 
		\[\sigma_{2}=\frac{3\a_1\xi_3}{\a_1-\a_2},\quad \sigma_{1}=\frac{3\a_1\xi_3^2+\b_2-\b_1}{\a_1-\a_2},\quad \sigma_0=\frac{\phi^{\a_1,\b_1}(\xi_3)-\tau_3}{\a_1-\a_2}.\]
		Consequently, it follows from Lemma \ref{Lemma, int for cubic} and direct calculation that 
		\begin{eqnarray}
			\int\frac{d\xi_{2}}{\la P_{\xi_3,\tau_3}(\xi_2)\ra ^{2b}} &\ls & \big\la 3\sigma_{1}-\sigma_{2}^{2}\big\ra^{-\frac14} \nonumber\\
			&\sim & \la -9\a_1\a_2\xi_3^2+3(\a_1-\a_2)(\b_2-\b_1)\ra^{-\frac14}. \label{nd1, 0<r<1/4, A32, int in cubic}
		\end{eqnarray}
		Since $|\xi_{3}|\sim |\xi_{1}|\gs 1$, when $|\b_2-\b_1|$ is sufficiently small, it follows from (\ref{nd1, 0<r<1/4, A32, int in cubic}) that 
		\[\int\frac{d\xi_{2}}{\la P_{\xi_3,\tau_3}(\xi_2)\ra ^{2b}} \ls \la \xi_3^2\ra^{-\frac14}\sim |\xi_{3}|^{-\frac12}.\]
		Hence, it follows from (\ref{nd1, 0<r<1/4, A32, single int}) that 
		\[\text{LHS of (\ref{nd1, 0<r<1/4, A3, single int})} \ls  \sup_{\xi_{3}\gs 1}\frac{|\xi_{3}|^{4\rho+2}}{\la\xi_{3}\ra ^{2\rho}\la \xi_{3}\ra ^{6(1-b)}}|\xi_3|^{-\frac12} \ls \sup_{\xi_{3}\gs 1} |\xi_{3}|^{6b+2\rho-\frac92}\leq C,\]
		where the last inequality is due to $6b+2\rho-\frac92\leq 0$ \big(see (\ref{nd1, 0<r<1/4, cond on b})\big).
		
		\item On $A_{33}$, the argument is similar to that for $A_{32}$.
	\end{itemize}

	\subsubsection{Type (ND1) with $r=\frac14$ and $s\geq \frac34$}
	\label{Subsubsec, ND1, r=1/4}
	
	Similar to the argument as in the proof of Lemma \ref{Lemma, bilin to weighted l2}, it suffices to prove 
	\[\int\limits_{A}\frac{|\xi_{1}|\la\xi_{3}\ra^{s}\prod\limits_{i=1}^{3}|f_{i}(\xi_{i},\tau_{i})|}{\la\xi_{1}\ra^{s}\la\xi_{2}\ra^{s}\la L_{1}\ra^{b}\la L_{2}\ra^{b}\la L_{3}\ra^{1-b}} \leq C\,\prod_{i=1}^{3}\|f_{i}\|_{L^{2}_{\xi\tau}}, \quad\forall\,\{f_{i}\}_{1\leq i\leq 3}.\]
	As $\frac{\la\xi_{3}\ra}{\la\xi_{1}\ra\la\xi_{2}\ra}\leq 1$,  we only need consider the case of  $s=\frac{3}{4}$, i.e., 
	\be\label{nd1, r=1/4, weighted l2 form}
	\int\limits_{A}\frac{|\xi_{1}|\la\xi_{3}\ra^{\frac{3}{4}}\prod\limits_{i=1}^{3}|f_{i}(\xi_{i},\tau_{i})|}{\la\xi_{1}\ra^{\frac{3}{4}}\la\xi_{2}\ra^{\frac{3}{4}}\la L_{1}\ra^{b}\la L_{2}\ra^{b}\la L_{3}\ra^{1-b}} \leq C\,\prod_{i=1}^{3}\|f_{i}\|_{L^{2}_{\xi\tau}}, \quad\forall\,\{f_{i}\}_{1\leq i\leq 3}.\ee
	Assume  $b\in\big(\frac{1}{2},b_0\big]$ with  $b_{0}=1$. Similar as before, it reduces to show 
	\be\label{nd1, r=1/4, single int}
	\sup_{\xi_{3},\tau_{3}}\frac{\la\xi_3\ra^{\frac32}}{\la L_{3}\ra^{2(1-b)}}\int\frac{|\xi_{1}|^{2}}{\la\xi_1\ra^{\frac32}\la\xi_2\ra^{\frac32}\la P_{\xi_3,\tau_3}(\xi_{2})\ra^{2b}}\,d\xi_{2}\leq C,\ee
	where $P_{\xi_3,\tau_3}(\xi_{2})=L_{1}+L_{2}$ is the same as (\ref{L1+L2, H2}) but with $r=\frac14$. More precisely, 
	\be\label{nd1, r=1/4, L1+L2}
	P_{\xi_3,\tau_3}(\xi_{2})=\frac34\a_1\xi_2^3+3\a_1\xi_3\xi_2^2+(3\a_1\xi_3^2+\b_2-\b_1)\xi_2+\phi^{\a_1,\b_1}(\xi_3)-\tau_3.\ee
	Taking derivative with respect to $\xi_2$, then 
	\[P'_{\xi_3,\tau_3}(\xi_{2})=\frac94\a_1\xi_2^2+6\a_1\xi_3\xi_2+3\a_1\xi_3^2+\b_2-\b_1.\]
	When $\xi_2\neq 0$, it can be rewritten as 
	\be\label{nd1, r=1/4, deri of L1+L2}
	P'_{\xi_3,\tau_3}(\xi_{2})=3\a_1\xi_2^2 p\Big(\frac{\xi_3}{\xi_2}\Big)+\b_2-\b_1,\ee
	where the function $p$ is just the function $p_{r}$, as defined in (\ref{def of p_r}), with $r=\frac14$. That is
	\[p(x)=x^2+2x+\frac34.\]
	Since the function $p$ has two roots $-\frac32$ and $-\frac12$, we further decompose the domain $A$ as 
	$A=\bigcup\limits_{i=0}^{3}B_{i}$ with
	\begin{align*}
		&B_{0}=\{(\vec{\xi},\vec{\tau})\in A:|\xi_{1}|\leq 1\}, \vspace{0.05in}\\
		&B_{1}=\Big\{(\vec{\xi},\vec{\tau})\in A:|\xi_{1}|>1,\, \Big |\frac{\xi_{3}}{\xi_{2}}+\frac32\Big|\geq \frac{1}{10}\,\,\text{and}\,\,\Big |\frac{\xi_{3}}{\xi_{2}}+\frac12\Big|\geq \frac{1}{10}\Big\}, \vspace{0.25in}\\
		&B_{2}=\Big\{(\vec{\xi},\vec{\tau})\in A:|\xi_{1}|>1,\, \Big |\frac{\xi_{3}}{\xi_{2}}+\frac32\Big|< \frac{1}{10}\Big\}, \vspace{0.05in}\\
		&B_{3}=\Big\{(\vec{\xi},\vec{\tau})\in A:|\xi_{1}|>1,\,\Big |\frac{\xi_{3}}{\xi_{2}}+\frac12\Big|<\frac{1}{10}\Big\}.
	\end{align*}
	
	Among these regions $ \{B_i\}_{i=0}^{3} $, the most difficult analysis occurs on $ B_2 $ (or equivalently on $ B_3 $), so next we will just focus on $ B_2 $. It is easily seen that $|\xi_1|\sim |\xi_2|\sim |\xi_3|$ on $B_{2}$, so 
	\be\label{nd1, r=1/4, final est}
	\text{LHS of (\ref{nd1, r=1/4, single int})}\ls \sup_{\xi_{3},\tau_{3}}\frac{|\xi_3|^{\frac12}}{\la L_{3}\ra^{2(1-b)}}\int\frac{1}{\la P_{\xi_3,\tau_3}(\xi_{2})\ra^{2b}}\,d\xi_{2}.\ee
	By dividing the leading coefficient $\frac{3}{4}\a_1$ in (\ref{nd1, r=1/4, L1+L2}),  we get
	\[\la P_{\xi_3,\tau_3}(\xi_{2})\ra \sim\big\la\xi_2^3+\sigma_2\xi_2^2+\sigma_1\xi_2+\sigma_0\big\ra,\]
	where 
	\[\sigma_2=4\xi_3, \quad \sigma_1=4\xi_3^2+\frac{4(\b_2-\b_1)}{3\a_1}, \quad \sigma_0=\frac{4}{3\a_1}\Big(\phi^{\a_1,\b_1}(\xi_3)-\tau_3\Big).\]
	Then it follows from Lemma \ref{Lemma, int for cubic} that 
	\begin{eqnarray*}
		\int\frac{1}{\la P_{\xi_3,\tau_3}(\xi_{2})\ra^{2b}}\,d\xi_{2} &\ls &\la 3\sigma_1-\sigma_2^2\ra^{-\frac14}\\
		&=& \Big\la -4\xi_3^2+\frac{4(\b_2-\b_1)}{\a_1}\Big\ra^{-\frac14}.
	\end{eqnarray*}
	Since $|\xi_3|\sim |\xi_1|\geq 1$, when $|\b_2-\b_1|$ is sufficiently small, 
	$\big\la-4\xi_3^2+\frac{4(\b_2-\b_1)}{\a_1}\big\ra\sim \xi_3^2$. Consequently, 
	\[\int\frac{1}{\la P_{\xi_3,\tau_3}(\xi_{2})\ra^{2b}}\,d\xi_{2}\ls |\xi_3|^{-\frac12},\]
	which implies the boundedness of the right hand side of (\ref{nd1, r=1/4, final est}).
	
	\subsubsection{Type (ND1) with $r=1$ and $s>0$}
	\label{Subsubsec, ND1, r=1}
	
	Similar as before, it suffices to show
	\be\label{nd1, r=1, weighted l2 form}
	\int\limits_{A}\frac{|\xi_{1}|\la\xi_{3}\ra^{s}\prod\limits_{i=1}^{3}|f_{i}(\xi_{i},\tau_{i})|}{\la\xi_{1}\ra^{s}\la\xi_{2}\ra^{s}\la L_{1}\ra^{b}\la L_{2}\ra^{b}\la L_{3}\ra^{1-b}} \leq C\,\prod_{i=1}^{3}\|f_{i}\|_{L^{2}_{\xi\tau}}, \quad\forall\,\{f_{i}\}_{1\leq i\leq 3},\ee
	and we only need to consider the case when $0<s\leq \frac{3}{4}$.   Let  $b_{0}=\frac{1}{2}+\frac{s}{3}$ and assume  $b\in\big(\frac{1}{2},b_{0}\big]$. Hence, 
	\be\label{nd1, r=1, cond on b}
	\frac{1}{2}<b\leq \frac{1}{2}+\frac{s}{3}\leq \frac{3}{4}.\ee
	Decompose  the region $A$ as $A=\bigcup\limits_{i=0}^{2} B_{i}$, where
	\begin{align*}
		&B_{0}=\{(\vec{\xi},\vec{\tau})\in A:|\xi_{1}|\leq 1\},\\
		&B_{1}=\Big\{(\vec{\xi},\vec{\tau})\in A:|\xi_{1}|>1,\, |\xi_{3}|\geq \frac{1}{4}|\xi_{1}|\Big\},\\
		&B_{2}=\Big\{(\vec{\xi},\vec{\tau})\in A:|\xi_{1}|>1,\, |\xi_{3}|< \frac{1}{4}|\xi_{1}|\Big\}.
	\end{align*}
	Among the above regions, the most difficult analysis occurs on $ B_2 $, so next we will just focus on this part.
	
	\noindent {\bf Contribution on $B_{2}$:} 
	
	Since $|\xi_{3}|< \frac{1}{4}|\xi_{1}|$, then
	\be\label{nd1, r=1, xi_2 approx xi_1}
	\frac{4}{5}|\xi_{2}|\leq |\xi_{1}|\leq \frac{4}{3}|\xi_{2}|\ee
	and 
	\[\int\frac{|\xi_{1}|\la\xi_{3}\ra^{s}\prod\limits_{i=1}^{3}|f_{i}(\xi_{i},\tau_{i})|}{\la\xi_{1}\ra^{s}\la\xi_{2}\ra^{s}\la L_{1}\ra^{b}\la L_{2}\ra^{b}\la L_{3}\ra^{1-b}} 
	\ls 
	\iint\frac{|\xi_{2}|^{1-s}|f_{2}|}{\la L_{2}\ra^{b}}\bigg(\iint\frac{|f_{1}f_{3}|\,d\tau_{1}d\xi_{1}}{\la L_{1}\ra^{b}\la L_{3}\ra^{1-b}}\bigg)\,d\tau_{2}\,d\xi_{2}.\]
	Thus, similar to the derivation for (\ref{d2, neg, A22, L1+L3}) in Section \ref{Subsec, proof of bilin, neg}, it suffices to prove 
	\be\label{nd1, r=1, B2, single int}
	\sup_{\xi_{2},\tau_{2}}\frac{|\xi_{2}|^{2(1-s)}}{\la L_{2}\ra^{2b}}\int\frac{d\xi_{1}}{\la Q_{\xi_2,\tau_2}(\xi_{1})\ra^{2(1-b)}}\leq C,\ee
	where $Q_{\xi_2,\tau_2}(\xi_{1})=L_1+L_3$ is as defined in (\ref{L1+L3, H2}). More specifically, 
	\be\label{nd1, r=1, Q}
	Q_{\xi_2,\tau_2}(\xi_{1})
	=3\a_1\xi_2\xi_1^2+3\a_1\xi_2^2\xi_1+\phi^{\a_1,\b_1}(\xi_2)-\tau_2.\ee
	Taking derivative respect to $\xi_1$,
	\be\label{nd1, r=1, Q'}
	Q'_{\xi_2,\tau_2}(\xi_{1})=3\a_1\xi_2(2\xi_1+\xi_2).\ee
	Due to (\ref{nd1, r=1, xi_2 approx xi_1}), we have 
	\be\label{nd1, r=1, lbd for Q'}
	|Q'_{\xi_2,\tau_2}(\xi_{1})|\gs |\xi_{2}|^{2}.\ee
	Moreover, since $r=1$, then $\a_1=\a_2$ and  
	\[\big|\phi^{\a_1,\b_1}(\xi_2)-\tau_2\big|=\big|\phi^{\a_2,\b_2}(\xi_2)-\tau_2+(\b_2-\b_1)\xi_2 \big|=\big|-L_{2}+(\b_2-\b_1)\xi_2\big|.\]
	When $|\b_2-\b_1|$ is sufficiently small, 
	\[|-L_{2}+(\b_2-\b_1)\xi_2|\leq |L_2|+|\xi_2|^{3}. \]
	As a result, it follows from (\ref{nd1, r=1, Q}) and (\ref{nd1, r=1, xi_2 approx xi_1}) that 
	\be\label{nd1, r=1, ubd for Q}
	\big|Q_{\xi_2,\tau_2}(\xi_{1})\big|\leq C |\xi_{2}|^{3}+|L_{2}|.\ee
	Then by (\ref{nd1, r=1, ubd for Q}) and (\ref{nd1, r=1, lbd for Q'}), we obtain
	\begin{eqnarray*}
		\int\frac{d\xi_{1}}{\la Q_{\xi_2,\tau_2}(\xi_{1})\ra^{2(1-b)}} &\leq & 
		\i_{\{|Q_{\xi_2,\tau_2}(\xi_{1})|\leq C|\xi_{2}|^{3}+|L_{2}|\}}
		\frac{1}{|Q'_{\xi_2,\tau_2}(\xi_{1})|}\frac{|Q'_{\xi_2,\tau_2}(\xi_{1})|}{\la Q_{\xi_2,\tau_2}(\xi_{1})\ra^{2(1-b)}}\,d\xi_{1}\\
		&\ls & \frac{1}{|\xi_{2}|^{2}}\int_{0}^{C|\xi_{2}|^{3}+|L_{2}|}\frac{dy}{\la y\ra^{2(1-b)}}\\
		&\ls & \frac{|\xi_{2}|^{3(2b-1)}+\la L_{2}\ra^{2b-1}}{|\xi_{2}|^{2}}.
	\end{eqnarray*}
	Hence,
	\begin{eqnarray}
		\text{LHS of (\ref{nd1, r=1, B2, single int})} &\ls & \sup_{\xi_2,\tau_2}\frac{|\xi_{2}|^{2(1-s)}\big(|\xi_{2}|^{3(2b-1)}+\la L_{2}\ra^{2b-1}\big)}{\la L_{2}\ra^{2b}|\xi_{2}|^2} \notag\\
		&=& \sup_{\xi_2,\tau_2}\frac{|\xi_2|^{6b-2s-1}+\la L_2\ra^{2b-1}|\xi_2|^{2(1-s)}}{\la L_{2}\ra^{2b}|\xi_{2}|^2}. \label{nd1, r=1, B2, final est}
	\end{eqnarray}
	Since $|\xi_2|\gs 1$ and (\ref{nd1, r=1, cond on b}) implies $6b-2s-1\leq 2$, the boundedness of (\ref{nd1, r=1, B2, final est}) is justified.

	\section{Sharpness of bilinear estimates}
	\label{Sec, sharp of bilin est}
	
	In this section we  prove Theorem \ref{Thm, sharp d2, neg} and \ref{Thm, sharp bilin est, general} which establish the sharpness of all the bilinear estimates in Theorem \ref{Thm, d2, neg} and \ref{Thm, bilin est, general}. We first fix some notations.  First, we define $A$  as in (\ref{int domain}), that is 
	\[A=\Big\{(\vec{\xi},\vec{\tau})\in\m{R}^{6}:\sum_{i=1}^{3}\xi_{i}=\sum_{i=1}^{3}\tau_{i}=0\Big\}.\]
	Secondly,  for any set $E\in \m{R}^2$,  we denote its Lebesgue  measure by $|E|$.   The following is a simple result which will be used frequently in this section.
	
	\begin{lemma}\label{Lemma, area of convolution}
		Let $E_{i}\subset\m{R}^{2}\,(1\leq i\leq 3)$ be bounded regions such that $E_{1}+E_{2}\subseteq -E_{3}$,  i.e., 
		\be\label{comp region}
		-(\xi_{1}+\xi_{2}, \tau_{1}+\tau_{2})\in E_{3},\quad \forall\, (\xi_{i},\tau_{i})\in E_{i},\,i=1,2.\ee
		Then 
		\[\int\limits_{A}\prod_{i=1}^{3}\mb{1}_{E_{i}}(\xi_{i},\tau_{i})= |E_{1}\|E_{2}|.\]
	\end{lemma}
	The proof of this lemma follows from (\ref{comp region}) by rewriting the left hand side of the above equation as 
	\[\iint_{E_{1}}\bigg(\iint_{E_{2}}\mb{1}_{E_{3}}\big(-(\xi_{1}+\xi_{2}),-(\tau_{1}+\tau_{2})\big) \,d\xi_{2}d\tau_{2}\bigg)\,d\xi_{1}d\tau_{1}.\]
	
	\subsection{Proof of Theorem \ref{Thm, sharp d2, neg}}
	\label{Subsec, proof of sharp of bilin est, neg}
	
	\noindent {\bf Proof of Case (1).}
	
	Fix $\a_1,\a_2,\b\in\m{R}$ with $\a_1\a_2<0$. Suppose there exist $s<-\frac{13}{12}$, $b\in\m{R}$ and $C=C(\a_1,\a_2,\b,s,b)$ such that  the bilinear estimate (\ref{d2}) holds. Then it follows from Lemma \ref{Lemma, bilin to weighted l2} that 
	\be\label{weighted l2 form, d2, neg, fail}
	\int\limits_{A}\frac{\xi_{3}\la\xi_{3}\ra^{s}\prod\limits_{i=1}^{3}f_{i}(\xi_{i},\tau_{i})}{\la\xi_{1}\ra^{s}\la\xi_{2}\ra^{s}\la L_{1}\ra^{b}\la L_{2}\ra^{b}\la L_{3}\ra^{1-b}} \leq C\,\prod_{i=1}^{3}\|f_{i}\|_{L^{2}_{\xi\tau}}, \quad\forall\, \{f_i\}_{1\leq i\leq 3},  \ee
	where 
	\be\label{def of L, H2, fail}
	L_{1}=\tau_{1}-\phi^{\a_1,\b}(\xi_1),\quad L_{2}=\tau_{2}-\phi^{\a_2,\b}(\xi_2),\quad L_{3}=\tau_{3}-\phi^{\a_1,\b}(\xi_3).\ee
	Let $r=\frac{\a_2}{\a_1}$. Then $r<0$. The resonance function is $H_{2}$ as calculated in (\ref{exp of H_2}) with $\b_1=\b_2=\b$, that is
	\begin{eqnarray}
		H_2(\xi_{1},\xi_{2},\xi_{3}) &=& \xi_{2}\Big[(\a_2-\a_1)\xi_2^2-3\a_1\xi_1\xi_2-3\a_1\xi_1^2\Big] \nonumber\\
		&=& -3\a_1\xi_2\Big(\frac{1-r}{3}\xi_2^2+\xi_1\xi_2+\xi_1^2\Big) \label{H2, sharp, formula}.
	\end{eqnarray}
	So $|H_2|\sim |\xi_2|(\xi_1^2+\xi_2^2)$ due to the fact that $r<0$.
	\begin{itemize}
		\item {\bf Claim A}:  {\em If   (\ref{weighted l2 form, d2, neg, fail}) holds, then 
			\be\label{fail of d2-neg-cond 1 on b}
			b\leq \frac{4+2s}{3}.\ee }
		
		For any large number $N>0$, let 
		\begin{align*}
			B_{1}&=\big\{(\xi_{1},\tau_{1}):N-1\leq \xi_{1}\leq N,\quad \big |\tau_{1}-\phi^{\a_1,\b}(\xi_1)\big |\leq 1\big\},\\
			B_{2}&=\big\{(\xi_{2},\tau_{2}):-N-2\leq \xi_{2}\leq -N-1,\quad \big|\tau_{2}-\phi^{\a_2,\b}(\xi_2)\big|\leq 1\big\}.
		\end{align*}
		For any $(\xi_{1},\tau_{1})\in B_{1}$ and $(\xi_{2},\tau_{2})\in B_{2}$, $(\xi_{3},\tau_{3})=-(\xi_{1}+\xi_{2},\tau_{1}+\tau_{2})$ satisfies $1\leq \xi_{3}\leq 3$. Since $|\xi_{1}-N|\leq 1$ and $|\xi_{2}+N|\leq 2$, 
		\begin{eqnarray}
			\phi^{\a_1,\b}(\xi_1)+\phi^{\a_2,\b}(\xi_2)&=& \a_{1}(\xi_{1}-N+N)^{3}+\a_{2}(\xi_{2}+N-N)^{3}-\b(\xi_1+\xi_{2}) \nonumber\\
			&=& (\a_{1}-\a_2)N^{3}+O(N^{2}). \label{fail-d2-key}
		\end{eqnarray}
		Moreover, since $|\tau_{1}-\phi^{\a_1,\b}(\xi_1)|\leq 1$ and $|\tau_{2}-\phi^{\a_2,\b}(\xi_2)|\leq 1$, it follows from (\ref{fail-d2-key}) that 
		\[|\tau_{3}+(\a_{1}-\a_2)N^{3}|=O(N^{2}).\]
		
		Thus, for  a suitably large constant $C_{1}$, the set 
		\[B_{3}:=\{(\xi_{3},\tau_{3}):1\leq \xi_{3}\leq 3,\quad |\tau_{3}+(\a_{1}-\a_{2})N^{3}|\leq C_{1}N^{2}\}\]
		satisfies $B_{1}+B_{2}\subseteq -B_{3}$. In addition, $|B_{1}|=|B_{2}|=2$ and $|B_{3}|\sim N^{2}$. Choosing $f_{i}=\mb{1}_{B_{i}}$ ( $1\leq i\leq 3$) in  (\ref{weighted l2 form, d2, neg, fail}) yields 
		\be\label{fail of d2-neg-1}
		C\,\prod_{i=1}^{3}|B_{i}|^{\frac{1}{2}}\geq \int\limits_{A}\frac{\xi_{3}\la\xi_{3}\ra^{s}\prod\limits_{i=1}^{3}\mb{1}_{B_{i}}(\xi_{i},\tau_{i})}{\la\xi_{1}\ra^{s}\la\xi_{2}\ra^{s}\la L_{1}\ra^{b}\la L_{2}\ra^{b}\la L_{3}\ra^{1-b}}.\ee
		For any $(\xi_{i},\tau_{i})\in B_{i}$, $1\leq i\leq 3$, it holds that
		\[|L_{1}|\leq 1,\quad |L_{2}|\leq 1, \quad |H_{2}(\xi_{1},\xi_{2},\xi_{3})|\sim N^{3}.\]
		So $|L_{3}|=|H_2+L_{1}+L_{2}|\sim N^{3}$. It then follows from (\ref{fail of d2-neg-1}) and Lemma \ref{Lemma, area of convolution} that 
		\[N \gs \frac{1}{N^{2s}N^{3(1-b)}}\int\limits_{A}\prod_{i=1}^{3}\mb{1}_{B_{i}}(\xi_{i},\tau_{i})= \frac{|B_{1}\|B_{2}|}{N^{2s}N^{3(1-b)}}\sim \frac{1}{N^{2s}N^{3(1-b)}}.\]
		which implies (\ref{fail of d2-neg-cond 1 on b}).
		
		\item {\bf Claim B:} {\em If   (\ref{weighted l2 form, d2, neg, fail}) holds, then 
			\be\label{fail of d2-neg-cond 2 on b}
			b\geq \frac{1}{4}-\frac{s}{3}.\ee 
		}
		\medskip
		Similarly, for large number $N$, let 
		\begin{align*}
			B_{1}&:=\big\{(\xi_{1},\tau_{1}):N-N^{-\frac{1}{2}}\leq \xi_{1}\leq N,\quad \big|\tau_{1}-\phi^{\a_1,\b}(\xi_1)\big|\leq 1\big\},\\
			B_{3}&:=\big\{(\xi_{3},\tau_{3}):N-N^{-\frac{1}{2}}\leq \xi_{3}\leq N,\quad \big|\tau_{3}-\phi^{\a_1,\b}(\xi_3)\big|\leq 1\big\}.
		\end{align*}
		For any $(\xi_{1},\tau_{1})\in B_{1}$ and $(\xi_{3},\tau_{3})\in B_{3}$,  $(\xi_{2},\tau_{2})=-(\xi_{1}+\xi_{3},\tau_{1}+\tau_{3})$ satisfies
		\[ \mbox{$-2N\leq \xi_{2}\leq -2N+2N^{-\frac{1}{2}}$. }\]
		As 
		\begin{eqnarray*}
			\phi^{\a_1,\b}(\xi_1)+\phi^{\a_1,\b}(\xi_3) &=& \a_{1}\xi_{1}^{3}+\a_{1}\xi_{3}^{3}-\b(\xi_1+\xi_3) \\
			&=& \a_{1}(\xi_{1}+\xi_{3})\bigg[\frac{(\xi_{1}+\xi_{3})^{2}}{4}+\frac{3(\xi_{1}-\xi_{3})^{2}}{4}\bigg]+\b\xi_2 \\
			&=& -\frac{\a_{1}\xi_{2}^{3}}{4}+O(1)+\b\xi_2,
		\end{eqnarray*}
		it follows from $\big|\tau_{1}-\phi^{\a_1,\b}(\xi_1)\big|\leq 1$ and $\big|\tau_{3}-\phi^{\a_1,\b}(\xi_3)\big|\leq 1$ that
		\[\Big|\tau_{2}-\frac{\a_{1}}{4}\xi_{2}^{3}+\b\xi_2\Big|=O(1).\]
		Thus,  for a suitably large constant $C_{2}$, the set 
		\[B_{2}:=\big\{(\xi_{2},\tau_{2}):-2N\leq \xi_{2}\leq -2N+2N^{-\frac{1}{2}},\quad \big|\tau_{2}-\frac{\a_{1}}{4}\xi_{2}^{3}+\b\xi_2\big|\leq C_{2}\big\}\]
		satisfies $B_{1}+B_{3}\subseteq -B_{2}$. Moreover, $|B_{1}|\sim |B_{3}|\sim |B_{2}|\sim N^{-\frac{1}{2}}$. Choosing $f_{i}=\mb{1}_{B_{i}}$ ($1\leq i\leq 3$) in  (\ref{weighted l2 form, d2, neg, fail})  yields
		\be\label{fail of d2-neg-2}
		C\,\prod_{i=1}^{3}|B_{i}|^{\frac{1}{2}}\geq \int\limits_{A}\frac{\xi_{3}\la\xi_{3}\ra^{s}\prod\limits_{i=1}^{3}\mb{1}_{B_{i}}(\xi_{i},\tau_{i})}{\la\xi_{1}\ra^{s}\la\xi_{2}\ra^{s}\la L_{1}\ra^{b}\la L_{2}\ra^{b}\la L_{3}\ra^{1-b}}.\ee
		As for any $(\xi_{i},\tau_{i})\in B_{i}$, $1\leq i\leq 3$,  
		\[|L_{1}|\leq 1,\quad |L_{3}|\leq 1, \quad |H_2(\xi_{1},\xi_{2},\xi_{3})|\sim N^{3},\]
		we have $|L_{2}|=|H_2+L_{1}+L_{3}|\sim N^{3}$. It then follows from (\ref{fail of d2-neg-2}) and Lemma \ref{Lemma, area of convolution} that
		\[N^{-\frac{3}{4}} \gs \frac{N^{1+s}}{N^{2s}N^{3b}}\int\limits_{A}\prod_{i=1}^{3}\mb{1}_{B_{i}}(\xi_{i},\tau_{i})= \frac{N^{1+s}|B_{1}\|B_{3}|}{N^{2s}N^{3b}}\sim \frac{1}{N^{s}N^{3b}},\]
		which implies (\ref{fail of d2-neg-cond 2 on b}).
	\end{itemize}
	
	Combining (\ref{fail of d2-neg-cond 1 on b}) and (\ref{fail of d2-neg-cond 2 on b}) together yields $s\geq -\frac{13}{12}$, which contradicts to the assumption $s<-\frac{13}{12}$.
	
	\bigskip
	\noindent {\bf Proof of Case (2).} 
	
	If  $-\frac{13}{12}\leq s\leq -1$,  the  same arguments as in the proofs of Claim A and Claim  B show that  (\ref{fail of d2-neg-cond 1 on b}) and (\ref{fail of d2-neg-cond 2 on b}) are necessary conditions on $b$  if  the bilinear estimate  (\ref{d2}) holds.
	
	\bigskip
	
	\noindent {\bf Proof of Case (3).} 
	
	Let $-1<s<-\frac{3}{4}$.  The  same argument  as in the proof of Claim B shows 
	$b\geq \frac{1}{4}-\frac{s}{3}$.
	To obtain the desired upper bound for $b$,  let
	\begin{align*}
		B_{1}&:=\big\{(\xi_{1},\tau_{1}):N-1\leq \xi_{1}\leq N,\quad \big|\tau_{1}-\phi^{\a_1,\b}(\xi_1)\big|\leq 1\big\},\\
		B_{2}&:=\big\{(\xi_{2},\tau_{2}):N-1\leq \xi_{2}\leq N,\quad \big|\tau_{2}-\phi^{\a_2,\b}(\xi_2)\big|\leq 1\big\}.
	\end{align*}
	Then similar to the procedure in the proof of Claim A in Case (1), there exists a suitably large constant $C_{3}$ such that the set 
	\[B_{3}:=\{(\xi_{3},\tau_{3}):-2N\leq \xi_{3}\leq -2N+2,\quad |\tau_{3}+(\a_{1}+\a_{2})N^{3}|\leq C_{3}N^{2}\}\]
	has the property $B_{1}+B_{2}\subseteq -B_{3}$. In addition, $|B_{1}|=|B_{2}|=2$ and $|B_{3}|\sim N^{2}$. Choosing $f_{i}=-\mb{1}_{B_{i}}$  ($1\leq i\leq 3$) in (\ref{weighted l2 form, d2, neg, fail}) yields  
	\be\label{fail of div 2-neg-3}
	C\,\prod_{i=1}^{3}|B_{i}|^{\frac{1}{2}}\geq \int\limits_{A}\frac{-\xi_{3}\la\xi_{3}\ra^{s}\prod\limits_{i=1}^{3}\mb{1}_{B_{i}}(\xi_{i},\tau_{i})}{\la\xi_{1}\ra^{s}\la\xi_{2}\ra^{s}\la L_{1}\ra^{b}\la L_{2}\ra^{b}\la L_{3}\ra^{1-b}}.\ee
	For any $(\xi_{i},\tau_{i})\in B_{i}$, $1\leq i\leq 3$, 
	\[|L_{1}|\leq 1,\quad |L_{2}|\leq 1, \quad |H_2(\xi_{1},\xi_{2},\xi_{3})|\sim N^{3}, \quad |L_{3}|\sim N^3.\]
	It then  follows from (\ref{fail of div 2-neg-3}) and Lemma \ref{Lemma, area of convolution} that  $b\leq 1+\frac{s}{3}$.

	\subsection{Proof of Theorem \ref{Thm, sharp bilin est, general}}
	\label{Subsec, proof of sharp of bilin est, general}
	
	First, we want to point out several cases in Table \ref{Table, sharp bilin est} which have been known or can be proved similarly. When $r=1$, the bilinear estimates of Type (D1) and (D2) have been known to fail if $s<-\frac34$, see \cite{KPV96}. When $r>\frac14$ but $r\neq 1$, the bilinear estimates of Type (D1) and (D2) do not hold for $s<0$, see  \cite{Oh09a}. The situations for Type (ND1) and (ND2) can be treated similarly.
	
	For the rest cases in Table \ref{Table, sharp bilin est}, we will only prove the failure of the bilinear estimates in the following five cases since other cases are similar.  {\bf Case (1)}: Type (D1) with $r<0$ and $s<-\frac34$.  {\bf Case (2)}: Type (D2) with $0<r<\frac14$ and $s<-\frac34$.  {\bf Case (3)}: Type (ND1) with $r<0$ and $s<-\frac34$.  {\bf Case (4)}: Type (ND1) with $r=\frac14$ and $s<\frac34$.  {\bf Case (5)}: Type (ND1) with $r=1$ and $s<0$.
	
	Moreover, the general strategy for all the above cases is very similar to that in the proof of Theorem \ref{Thm, sharp d2, neg} as shown above. The key ingredient is to construct suitable $ \{B_i\}_{i=1}^3 $. So in the following, we will only write out the sets $ \{B_i\}_{i=1}^{3} $ that works for the argument, but omit the detailed computations which can be easily carried out.
	
	\bigskip
	\noindent {\bf Proof of Case (1).} 
	For any large number $N>0$, define
	\begin{align*}
		B_{1}&=\big\{(\xi_{1},\tau_{1}):N-1\leq \xi_{1}\leq N,\quad \frac{1}{2}N\leq \big|\tau_{1}-\phi^{\a_1,\b}(\xi_1)\big|\leq N\big\},\\
		B_{2}&=\big\{(\xi_{2},\tau_{2}):-N-2\leq \xi_{2}\leq -N-1,\quad \frac{1}{2}N\leq \big|\tau_{2}-\phi^{\a_1,\b}(\xi_2)\big|\leq N\big\}.
	\end{align*}
	Then we choose a suitably large constant $C_1$ such that the set   
	\[B_{3}:=\{(\xi_{3},\tau_{3}):1\leq \xi_{3}\leq 3,\quad |\tau_{3}-3\a_{1}N^{2}\xi_{3}|\leq C_{1}N\}\]
	satisfies $B_{1}+B_{2}\subseteq -B_{3}$. 
	
	\bigskip
	\noindent {\bf Proof of Case (2).} 
	For large number  $N>0$, define
	\begin{align*}
		B_{1}&:=\big\{(\xi_{1},\tau_{1}):N-N^{-\frac{1}{2}}\leq \xi_{1}\leq N,\quad \big|\tau_{1}-\phi^{\a_1,\b}(\xi_1)\big|\leq 1\big\},\\
		B_{2}&:=\big\{(\xi_{2},\tau_{2}):r^{-\frac{1}{2}}N-N^{-\frac{1}{2}}\leq \xi_{2}\leq r^{-\frac{1}{2}}N,\quad \big|\tau_{2}-\phi^{\a_2,\b}(\xi_2)\big|\leq 1\big\}.
	\end{align*}
	Then we choose a suitably large constant $C_1$ such that the set   
	\[B_{3}:=\Big\{(\xi_{3},\tau_{3}):-\big(1+r^{-\frac{1}{2}}\big)N\leq \xi_{3}\leq -\big(1+r^{-\frac{1}{2}}\big)N+2N^{-\frac{1}{2}},\, \big|\tau_{3}-2\a_{1}\big(1+r^{-\frac12}\big)N^{3}+(\b-3\a_{1}N^{2})\xi_{3}\big|\leq C_{1}\Big\}\]
	satisfies $B_{1}+B_{2}\subseteq -B_{3}$.

	\bigskip
	\noindent {\bf Proof of Case (3).} 
	For large number $N>0$,  define 
	\begin{align*}
		B_{1}&:=\big\{(\xi_{1},\tau_{1}):N-N^{-\frac{1}{2}}\leq \xi_{1}\leq N,\quad \frac{1}{2}N^{\frac{3}{2}}\leq \big|\tau_{1}-\phi^{\a_1,\b}(\xi_1)\big|\leq N^{\frac{3}{2}}\big\},\\
		B_{2}&:=\big\{(\xi_{2},\tau_{2}):-N-N^{-\frac{1}{2}}\leq \xi_{2}\leq -N,\quad \frac{1}{2}N^{\frac{3}{2}}\leq \big|\tau_{2}-\phi^{\a_2,\b}(\xi_2)\big|\leq N^{\frac{3}{2}}\big\}.
	\end{align*}
	Then we choose a suitably large constant $C_1$ such that the set   
	\[B_{3}:=\{(\xi_{3},\tau_{3}):0\leq \xi_{3}\leq 2N^{-\frac{1}{2}},\quad |\tau_{3}+(\a_{1}-\a_{2})N^{3}|\leq C_{1}N^{\frac{3}{2}}\}\]
	satisfies $B_{1}+B_{2}\subseteq -B_{3}$. 
	
	\bigskip
	\noindent {\bf Proof of Case (4).} 
	For large $N>0$, define 
	\begin{align*}
		B_{1}&=\big\{(\xi_{1},\tau_{1}):N-N^{-\frac{1}{2}}\leq \xi_{1}\leq N,\quad \big|\tau_{1}-\phi^{\a_1,\b}(\xi_1)\big|\leq 1\big\},\\
		B_{2}&=\big\{(\xi_{2},\tau_{2}):-2N-N^{-\frac{1}{2}}\leq \xi_{2}\leq -2N,\quad \big|\tau_{2}-\phi^{\a_2,\b}(\xi_2)\big|\leq 1\big\}.
	\end{align*}
	Then we choose a suitably large constant $C_1$ such that the set   
	\[B_{3}:=\{(\xi_{3},\tau_{3}):N\leq \xi_{3}\leq N+2N^{-\frac{1}{2}},\quad |\tau_{3}+2\a_{1}N^{3}+(\b-3\a_{1}N^{2})\xi_{3}|\leq C_{1}\}\]
	satisfies $B_{1}+B_{2}\subseteq -B_{3}$. 
	
	\bigskip
	\noindent {\bf Proof of Case (5).} 
	For large number $N>0$, define 
	\begin{align*}
		B_{1}&:=\big\{(\xi_{1},\tau_{1}):N-N^{-2}\leq \xi_{1}\leq N,\quad \big|\tau_{1}-\phi^{\a,\b}(\xi_1)\big|\leq 1\big\},\\
		B_{2}&:=\big\{(\xi_{2},\tau_{2}):-N-2N^{-2}\leq \xi_{2}\leq -N-N^{-2},\quad \big|\tau_{2}-\phi^{\a,\b}(\xi_2)\big|\leq 1\big\}.
	\end{align*}
	Then we choose a suitably large constant $C_1$ such that the set  
	\[B_{3}=\{(\xi_{3},\tau_{3}):N^{-2}\leq \xi_{3}\leq 3N^{-2},\quad |\tau_{3}+\b\xi_3|\leq C_{1}\}\]
	has the property $B_{1}+B_{2}\subseteq -B_{3}$. 
	
	\section*{Acknowledgements}
	The authors appreciate the anonymous referees for their careful review and helpful suggestions.


{\small
		\begin{thebibliography}{10}
			
			\bibitem{AC08}
			B.~Alvarez and X.~Carvajal.
			\newblock On the local well-posedness for some systems of coupled {K}d{V}
			equations.
			\newblock {\em Nonlinear Anal.}, 69(2):692--715, 2008.
			
			\bibitem{ACW96}
			J.~M. Ash, J.~Cohen, and G.~Wang.
			\newblock On strongly interacting internal solitary waves.
			\newblock {\em J. Fourier Anal. Appl.}, 2(5):507--517, 1996.
			
			\bibitem{BOP97}
			D.~Bekiranov, T.~Ogawa, and G.~Ponce.
			\newblock Weak solvability and well-posedness of a coupled
			{S}chr\"{o}dinger-{K}orteweg de {V}ries equation for capillary-gravity wave
			interactions.
			\newblock {\em Proc. Amer. Math. Soc.}, 125(10):2907--2919, 1997.
			
			\bibitem{BS76}
			J.~Bona and R.~Scott.
			\newblock Solutions of the {K}orteweg-de {V}ries equation in fractional order
			{S}obolev spaces.
			\newblock {\em Duke Math. J.}, 43(1):87--99, 1976.
			
			\bibitem{BPST92}
			J.~L. Bona, G.~Ponce, J.-C. Saut, and M.~M. Tom.
			\newblock A model system for strong interaction between internal solitary
			waves.
			\newblock {\em Comm. Math. Phys.}, 143(2):287--313, 1992.
			
			\bibitem{BS75}
			J.~L. Bona and R.~Smith.
			\newblock The initial-value problem for the {K}orteweg-de {V}ries equation.
			\newblock {\em Philos. Trans. Roy. Soc. London Ser. A}, 278(1287):555--601,
			1975.
			
			
			\bibitem{Bou93a}
			J.~Bourgain.
			\newblock Fourier transform restriction phenomena for certain lattice subsets
			and applications to nonlinear evolution equations. {I}. {S}chr\"{o}dinger
			equations.
			\newblock {\em Geom. Funct. Anal.}, 3(2):107--156, 1993.
			
			\bibitem{Bou93b}
			J.~Bourgain.
			\newblock Fourier transform restriction phenomena for certain lattice subsets
			and applications to nonlinear evolution equations. {II}. {T}he
			{K}d{V}-equation.
			\newblock {\em Geom. Funct. Anal.}, 3(3):209--262, 1993.
			
			
			\bibitem{control-4} E. Cerpa and E. Crépeau, Boundary controllability for the nonlinear
			Korteweg-de Vries equation on any critical domain. Ann. I.H. Poincaré – AN
			26:457--475, 2009.
			
			\bibitem{CCT03}
			Michael Christ, James Colliander, and Terrence Tao.
			\newblock Asymptotics, frequency modulation, and low regularity ill-posedness
			for canonical defocusing equations.
			\newblock {\em Amer. J. Math.}, 125(6):1235--1293, 2003.
			
			\bibitem{CKSTT03}
			J.~Colliander, M.~Keel, G.~Staffilani, H.~Takaoka, and T.~Tao.
			\newblock Sharp global well-posedness for {K}d{V} and modified {K}d{V} on
			{$\Bbb R$} and {$\Bbb T$}.
			\newblock {\em J. Amer. Math. Soc.}, 16(3):705--749, 2003.
			
			\bibitem{CS88}
			P.~Constantin and J.-C. Saut.
			\newblock Local smoothing properties of dispersive equations.
			\newblock {\em J. Amer. Math. Soc.}, 1(2):413--439, 1988.
			
			\bibitem{control-5} J.-M.  Coron and E. Crépeau,  Exact boundary controllability of a nonlinear
			KdV equation with a critical length. J. Eur. Math. Soc. 6:367--398, 2004.
			
			\bibitem{Fen94}
			X.~Feng.
			\newblock Global well-posedness of the initial value problem for the
			{H}irota-{S}atsuma system.
			\newblock {\em Manuscripta Math.}, 84(3-4):361--378, 1994.
			
			\bibitem{GG84}
			J.~A. Gear and R.~Grimshaw.
			\newblock Weak and strong interactions between internal solitary waves.
			\newblock {\em Stud. Appl. Math.}, 70(3):235--258, 1984.
			
			\bibitem{Guo09}
			Z.~Guo.
			\newblock Global well-posedness of {K}orteweg-de {V}ries equation in
			{$H^{-3/4}(\Bbb R)$}.
			\newblock {\em J. Math. Pures Appl. (9)}, 91(6):583--597, 2009.
			
			\bibitem{HS81}
			R.~Hirota and J.~Satsuma.
			\newblock Soliton solutions of a coupled {K}orteweg-de {V}ries equation.
			\newblock {\em Phys. Lett. A}, 85(8-9):407--408, 1981.
			
			\bibitem{Kam69}
			Y.~Kametaka.
			\newblock Korteweg -de vries equation, i, ii, iii, iv.
			\newblock {\em Proc. Japan Acad.}, 45:552--555; 556--558; 656--660; 661--665,
			1969.
			
			\bibitem{KT06}
			T.~Kappeler and P.~Topalov.
			\newblock Global wellposedness of {K}d{V} in {$H^{-1}(\Bbb T,\Bbb R)$}.
			\newblock {\em Duke Math. J.}, 135(2):327--360, 2006.
			
			\bibitem{Kat75}
			T.~Kato.
			\newblock Quasi-linear equations of evolution, with applications to partial
			differential equations.
			\newblock pages 25--70. Lecture Notes in Math., Vol. 448, 1975.
			
			\bibitem{Kat79}
			T.~Kato.
			\newblock On the {K}orteweg-de\thinspace {V}ries equation.
			\newblock {\em Manuscripta Math.}, 28(1-3):89--99, 1979.
			
			\bibitem{Kat81}
			T.~Kato.
			\newblock The {C}auchy problem for the {K}orteweg-de {V}ries equation.
			\newblock In {\em Nonlinear partial differential equations and their
				applications.}, volume~53 of {\em Pitman Research Notes in Math.}, pages
			293--307. 1981.
			
			\bibitem{Kat83}
			T.~Kato.
			\newblock On the {C}auchy problem for the (generalized) {K}orteweg-de {V}ries
			equation.
			\newblock In {\em Studies in applied mathematics}, volume~8 of {\em Adv. Math.
				Suppl. Stud.}, pages 93--128. Academic Press, New York, 1983.
			
			\bibitem{KPV89}
			C.~E. Kenig, G.~Ponce, and L.~Vega.
			\newblock On the (generalized) {K}orteweg-de {V}ries equation.
			\newblock {\em Duke Math. J.}, 59(3):585--610, 1989.
			
			\bibitem{KPV91Indiana}
			C.~E. Kenig, G.~Ponce, and L.~Vega.
			\newblock Oscillatory integrals and regularity of dispersive equations.
			\newblock {\em Indiana Univ. Math. J.}, 40(1):33--69, 1991.
			
			\bibitem{KPV91JAMS}
			C.~E. Kenig, G.~Ponce, and L.~Vega.
			\newblock Well-posedness of the initial value problem for the {K}orteweg-de
			{V}ries equation.
			\newblock {\em J. Amer. Math. Soc.}, 4(2):323--347, 1991.
			
			\bibitem{KPV93Duke}
			C.~E. Kenig, G.~Ponce, and L.~Vega.
			\newblock The {C}auchy problem for the {K}orteweg-de {V}ries equation in
			{S}obolev spaces of negative indices.
			\newblock {\em Duke Math. J.}, 71(1):1--21, 1993.
			
			\bibitem{KPV93CPAM}
			C.~E. Kenig, G.~Ponce, and L.~Vega.
			\newblock Well-posedness and scattering results for the generalized
			{K}orteweg-de {V}ries equation via the contraction principle.
			\newblock {\em Comm. Pure Appl. Math.}, 46(4):527--620, 1993.
			
			\bibitem{KPV96}
			C.~E. Kenig, G.~Ponce, and L.~Vega.
			\newblock A bilinear estimate with applications to the {K}d{V} equation.
			\newblock {\em J. Amer. Math. Soc.}, 9(2):573--603, 1996.
			
			\bibitem{KV19}
			Rowan Killip and Monica Vi\c{s}an.
			\newblock Kd{V} is well-posed in {$H^{-1}$}.
			\newblock {\em Ann. of Math. (2)}, 190(1):249--305, 2019.
			
			\bibitem{Kis09}
			N.~Kishimoto.
			\newblock Well-posedness of the {C}auchy problem for the {K}orteweg-de {V}ries
			equation at the critical regularity.
			\newblock {\em Differential Integral Equations}, 22(5-6):447--464, 2009.
			
			\bibitem{control-8} C. Laurent, L. Rosier, and B.-Y.Zhang, Control and Stabilization of the Korteweg-de
			Vries Equation on a Periodic Domain, Comm. Partial Differential Equations, 35,
			pp. 707-744, 2010
			
			\bibitem{LP04}
			F.~Linares and M.~Panthee.
			\newblock On the {C}auchy problem for a coupled system of {K}d{V} equations.
			\newblock {\em Commun. Pure Appl. Anal.}, 3(3):417--431, 2004.
			
			\bibitem{MB03}
			A.~J. Majda and J.~A. Biello.
			\newblock The nonlinear interaction of barotropic and equatorial baroclinic
			{R}ossby waves.
			\newblock {\em J. Atmospheric Sci.}, 60(15):1809--1821, 2003.
			
			\bibitem{Mol11}
			L.~Molinet.
			\newblock A note on ill posedness for the {K}d{V} equation.
			\newblock {\em Differential Integral Equations}, 24(7-8):759--765, 2011.
			
			\bibitem{Mol12}
			L.~Molinet.
			\newblock Sharp ill-posedness results for the {K}d{V} and m{K}d{V} equations on
			the torus.
			\newblock {\em Adv. Math.}, 230(4-6):1895--1930, 2012.
			
			\bibitem{Oh09b}
			T.~Oh.
			\newblock Diophantine conditions in global well-posedness for coupled
			{K}d{V}-type systems.
			\newblock {\em Electron. J. Differential Equations}, pages No. 52, 48, 2009.
			
			\bibitem{Oh09a}
			T.~Oh.
			\newblock Diophantine conditions in well-posedness theory of coupled
			{K}d{V}-type systems: local theory.
			\newblock {\em Int. Math. Res. Not. IMRN}, (18):3516--3556, 2009.
			
			\bibitem{control-3} L. Rosier, Exact boundary controllability for the Korteweg-de Vries
			equation on a bounded domain. ESAIM Control Optim. Cal. Var. 2:33--55, 1997.
			
			\bibitem{control-7} L. Rosier and B.-Y. Zhang, Control and stabilization of the Korteweg-de Vries equation:
			Recent progress, J. Syst. Sci. Complex, 22,  pp. 647--682, 2009
			
			\bibitem{contr-1}  D.-L. Russell and B.-Y. Zhang, Exact controllability and stabilizability of
			the Korteweg-de Vries equation. Trans. Amer. Math. Soc. 348:3643--3672, 1996.	
			
			\bibitem{ST76}
			J.~C. Saut and R.~Temam.
			\newblock Remarks on the {K}orteweg-de {V}ries equation.
			\newblock {\em Israel J. Math.}, 24(1):78--87, 1976.
			
			\bibitem{ST00}
			J.-C. Saut and N.~Tzvetkov.
			\newblock On a model system for the oblique interaction of internal gravity
			waves.
			\newblock {\em M2AN Math. Model. Numer. Anal.}, 34(2):501--523, 2000.
			\newblock Special issue for R. Temam's 60th birthday.
			
			\bibitem{Sjo67}
			A.~Sj\"{o}berg.
			\newblock On the {K}orteweg-de {V}ries equation: existence and uniqueness.
			\newblock {\em Department of Computer Sciences, Uppsala University, Uppsala,
				Sweden}, 1967.
			
			\bibitem{Sjo70}
			A.~Sj\"{o}berg.
			\newblock On the {K}orteweg-de {V}ries equation: existence and uniqueness.
			\newblock {\em J. Math. Anal. Appl.}, 29:569--579, 1970.
			
			\bibitem{Sjo87}
			P.~Sj\"{o}lin.
			\newblock Regularity of solutions to the {S}chr\"{o}dinger equation.
			\newblock {\em Duke Math. J.}, 55(3):699--715, 1987.
			
			\bibitem{Tao01}
			T.~Tao.
			\newblock Multilinear weighted convolution of {$L^2$}-functions, and
			applications to nonlinear dispersive equations.
			\newblock {\em Amer. J. Math.}, 123(5):839--908, 2001.
			
			\bibitem{Tar72}
			L.~Tartar.
			\newblock Interpolation non lin\'{e}aire et r\'{e}gularit\'{e}.
			\newblock {\em J. Functional Analysis}, 9:469--489, 1972.
			
			\bibitem{Tem69}
			R.~Temam.
			\newblock Sur un probl\`eme non lin\'{e}aire.
			\newblock {\em J. Math. Pures Appl. (9)}, 48:159--172, 1969.
			
			\bibitem{TM71}
			M.~Tsutsumi and T.~Mukasa.
			\newblock Parabolic regularizations for the generalized {K}orteweg-de {V}ries
			equation.
			\newblock {\em Funkcial. Ekvac.}, 14:89--110, 1971.
			
			\bibitem{TMI70}
			M.~Tsutsumi, T.~Mukasa, and R.~Iino.
			\newblock On the generalized {K}orteweg-de {V}ries equation.
			\newblock {\em Proc. Japan Acad.}, 46:921--925, 1970.
			
			\bibitem{Zha95SIMA}
			B.-Y. Zhang.
			\newblock Analyticity of solutions of the generalized {K}ortweg-de {V}ries
			equation with respect to their initial values.
			\newblock {\em SIAM J. Math. Anal.}, 26(6):1488--1513, 1995.
			
			\bibitem{Zha95DIE}
			B.-Y. Zhang.
			\newblock A remark on the {C}auchy problem for the {K}orteweg-de {V}ries
			equation on a periodic domain.
			\newblock {\em Differential Integral Equations}, 8(5):1191--1204, 1995.
			
			\bibitem{Zha95JFA}
			B.-Y. Zhang.
			\newblock Taylor series expansion for solutions of the {K}orteweg-de {V}ries
			equation with respect to their initial values.
			\newblock {\em J. Funct. Anal.}, 129(2):293--324, 1995.
			
			\bibitem{control-2} B.-Y. Zhang, Exact boundary controllability of the Korteweg-de Vries
			equation. SIAM J. Cont. Optim. 37:543--565, 1999.
			
			\bibitem{control-6} B.-Y. Zhang, Well-posedness and control of the Korteweg-de Vries equation on a bounded domain, in Fifth International Congress of Chinese Mathematicians, AMS/IP Stud. Adv.  Math. 51, AMS, Providence, RI, 2012, pp. 931--956.
			
		\end{thebibliography}
	}

	\bigskip
	
	\centering
	
	\thanks{(X. Yang) Department of Mathematics, University of California, Riverside, CA 92521, USA\\
		Email: xiny@ucr.edu}
	
	\bigskip
	
	\thanks{(B.-Y. Zhang) Department of Mathematical Sciences, University of Cincinnati, Cincinnati, OH 45221, USA\\
		\quad Email: zhangb@ucmail.uc.edu}
	
\end{document}